\documentclass[a4paper]{article}

\usepackage{latexsym,amsthm,amscd}
\usepackage[all]{xy}
\usepackage{rotating}
\usepackage[utf8]{inputenc}
\usepackage[T1]{fontenc}
\usepackage[english]{babel}
\usepackage{amsmath}
\usepackage{graphicx}
\usepackage{amssymb}
\usepackage{amsfonts}
\usepackage[svgnames]{xcolor}
\usepackage{geometry}
\usepackage{fancyhdr}
\usepackage{lastpage} 
\usepackage{array}
\usepackage{stmaryrd}
\usepackage{hyperref}
\usepackage{tikz}
\usepackage{bm}
\usepackage[french]{minitoc}
\usepackage{silence}
\usepackage{amsthm}

\usetikzlibrary{patterns}
\usepgflibrary{decorations.pathmorphing}
\usetikzlibrary{decorations.pathmorphing}
\usetikzlibrary{shadows}
\usetikzlibrary{positioning}
\usetikzlibrary{arrows,automata}
\usetikzlibrary{arrows.meta}
\usepackage{physics}
\usepackage{bigints}

\newtheorem{theorem}{Theorem}
\newtheorem{definition}{Definition}
\newtheorem{proposition}{Proposition}
\newtheorem{lemma}{Lemma}
\newtheorem{fact}{Fact}
\newtheorem{remark}{Remark}
\newtheorem{corollary}{Corollary}

\newtheorem*{rep@theorem}{\rep@title}
\newcommand{\newreptheorem}[2]{%
\newenvironment{rep#1}[1]{%
 \def\rep@title{#2 \ref{##1}}%
 \begin{rep@theorem}}%
 {\end{rep@theorem}}}

\title{A
proof that square ice entropy 
is $\frac{3}{2}\log_2 (4/3)$}

\author{Silvère Gangloff \\
LIP, ENS Lyon, 46 all\'ee d'Italie, 69342 Lyon  \\ 
 \href{mailto:silvere.gangloff@ens-lyon.fr}{silvere.gangloff@ens-lyon.fr} \\
 +33602386355}

\newcounter{quest}
\newcounter{squest}
        {\begin{list}
        {\bfseries #1.\arabic{quest}.}
        {\usecounter{quest}
        \setlength{\listparindent}{0cm}
        \setlength{\labelwidth}{1.5cm}
        \setlength{\leftmargin}{1cm}
        \setlength{\parsep}{0.5ex plus0.2ex minus0.1ex}
        \setlength{\itemsep}{0.7ex plus0.2ex minus0.1ex}}
        }{\end{list}}
       {\begin{list}
        {\bfseries \alph{squest}.}
        {\usecounter{squest}}
        }{\end{list}}

\usetikzlibrary{arrows}

\def\N{\mathbb N}

\def\Z{\mathbb Z}

\renewcommand{\vec}[1]{\textbf{#1}}

\newtheorem{notation}{Notation}
\newtheorem{computation}{Computation}

\newreptheorem{theorem}{Theorem}

\begin{document}
\maketitle

\begin{abstract}
In this text, we provide a fully rigorous and complete proof of E.H.Lieb's 
statement that (topological) entropy 
of square ice (or six vertex model, XXZ spin chain for anisotropy parameter $\Delta=1/2$) 
is equal to 
$\frac{3}{2}\log_2 (4/3)$. For this purpose, 
we gather and expose in full
detail various arguments dispersed in 
the literature on the subject, and complete several of them that were left partial. 
\end{abstract}

\section{Introduction}

\begin{center}
\textit{Separate the earth from the fire, the subtle from the raw, sweetly with great industry. } - \textbf{The Emeralt tablet}
\end{center}

\subsection{Computing entropy of multidimensional 
subshifts of finite type}

This work is the consequence of a renewal 
of interest from the field of symbolic dynamics to entropy
computation methods developped in quantum and statistical physics for lattice models. This interest 
comes from \textbf{constructive methods} for multidimensional subshifts of finite type (some equivalent formulation in symbolic dynamics of lattice models), 
that are involved in the characterization 
by M. Hochman and T. Meyerovitch~\cite{Hochman-Meyerovitch-2010} of the possible values 
of topological entropy for these dynamical systems
(where the dynamics are provided by the 
action of the $\mathbb{Z}^2$ shift action) with 
a recursion-theoretic criterion.
The consequences of this theorem are 
not only that entropy may be algorithmically uncomputable for a multidimensional subshift of finite type, 
which was previously known \cite{HurdKariCulik}, 
but also a strong evidence that the study 
of these systems as a class is intertwined 
with computability theory. Moreover, it 
is an important tool in order to localize 
sub-classes for which the entropy is 
computable in a uniform way, as ones 
defined by strong dynamical 
constraints~\cite{Pavlov-Schraudner-2014}.
Some current research attempts to understand the 
\textbf{frontier between the uncomputability 
and the computability} of entropy 
for multidimensional SFT. 
For instance, approaching 
the frontier from the uncomputable, the author, together 
with M. Sablik~\cite{GS17} proved 
that the characterization 
of M. Hochman and T. Meyerovitch stands under a relaxed form 
of the constraint studied in~\cite{Pavlov-Schraudner-2014}, which includes notably all 
square lattice models considered exactly solvable 
in quantum and statistical physics. Solvable means 
here that some exact (not necessary proved) 
values are 
provided for some caracteristics of the 
model, such as entropy. In 
order to approach the frontier from the computable, 
it is natural to attempt understanding (in particular proving)
and extending the computation methods developped for these models.

\subsection{Content of this text}
 
Our study in the present text focuses on square ice (or equivalently the 
six vertex model, or the XXZ spin chain for anisotropy parameter $\Delta=1/2$). Since it 
is central amongst 
quantum solvable models, 
this work will serve as a ground for further 
connections between entropy computation methods 
and constructive methods 
coming from symbolic dynamics. The entropy of 
square ice was argued by E.H. Lieb~\cite{Lieb67} to be exactly $\frac{3}{2} \log_2 \left( \frac{4}{3} \right)$. However, his 
proof was not complete, as it relied on 
a non verified hypothesis (the condensation 
of Bethe roots, defined in the text, 
according to a density function, proved in 
Section~\ref{section.asymptotics}). Moreover, 
some arguments of an article of C.N Yang and C.P Yang~\cite{YY66} on which it relied were left partial (the analycity of the roots according 
to the anisotropy parameter, proved in 
Section~\ref{section.existence}). We propose a rigorous proof of E.H.Lieb's statement:

\begin{theorem}
\label{theorem.main}
The entropy of square ice is equal 
to $\frac{3}{2} \log_2 \left( \frac{4}{3} \right)$.
\end{theorem} 

This proof relies on the argumentation of 
E.H.Lieb and on some 
ideas, developped in order 
to prove the hypothesis of E.H.Lieb,
that one can find in~\cite{kozlowski}. Besides partial arguments, various hurdles prevent the readers (in particuar with mathematical background) to have an overview of the subject in reasonable time. That is why we include some exposition of what can be 
considered as background material.
The proof is thus self-contained, except that it relies on the 
coordinate Bethe ansatz, exposed 
in a clear way by H. Duminil-Copin et al.~\cite{Duminil-Copin.ansatz}.
In proving Theorem~\ref{theorem.main}, an important difficulty was at first  
to connect and assemble the 
arguments that were found in the literature;
we encountered many obstacles,
that are related to the form 
of the literature itself.
This could be explained by 
the fact that mathematics 
and mathematical physics, 
although they share the same 
language, are different 
\textit{discursive formations},
notion introduced in 
the \textit{Archeology of 
knowledge} by M.Foucault. 
The mathematical 
text is (in our view) 
in particular serve 
a neat separation between statements 
that have different natures (theorem and proof, or comments); 
the hermetic axiom quoted 
above reflects this separation, 
where the earth could represents 
the ground in thought and 
statements (as proven theorems), 
and the fire the continuing transformation.
In Section~\ref{section.comments} one can 
find a short analysis on this matter and also some comments on the limits 
of the computing method presented in 
this text. We hope this short analysis 
could enlight the nature of the difficulties 
of this work, and will be followed by 
further developments.
\bigskip

One can find a summary of the proof of 
Theorem~\ref{theorem.main} in Section~\ref{section.overview}, after some recall 
of definitions related to symbolic dynamics 
and representations of square ice in Section~\ref{section.background}. \bigskip

\noindent \textbf{Aknowledgements:}
The author was funded by the ANR project
CoCoGro (ANR-16-CE40-0005) and 
is thankful to K.K.Kozlowski
for helpful discussions.

\section{\label{section.background} Background: square ice and its entropy}

\subsection{Subshifts of finite type}

\subsubsection{Definitions}
Let $\mathcal{A}$ be some finite set, called
\textbf{alphabet}. For all $d \ge 1$,
the set $\mathcal{A}^{\Z^d}$, 
whose elements are called \textbf{configurations}, is a topological 
space with the infinite power of the discrete 
topology on $\mathcal{A}$. 
Let us denote $\sigma$ the \textbf{shift} action 
of $\Z^d$ on this space defined 
by the following equality for all $\vec{u} \in 
\Z^d$ and $x$ element of the space:
$\left(\sigma^{\vec{u}}.x\right)_{\vec{v}} 
= x_{\vec{v}+\vec{u}}.$
A compact subset 
$X$ of this space is called a $d$-dimensional \textbf{subshift} 
when this subset is stable 
under the action of the shift, which means 
that for all $\vec{u} \in \Z^d$: 
$\sigma^{\vec{u}}.X \subset X.$
For any finite subset $\mathbb{U}$ 
of $\Z^d$, an element $p$ 
of $\mathcal{A}^{\mathbb{U}}$ 
is called a \textbf{pattern} on the 
alphabet $\mathcal{A}$ and on \textbf{support}
$\mathbb{U}$. We say that this 
pattern \textbf{appears} in 
a configuration $x$ when there 
exists a translate $\mathbb{V}$ of $\mathbb{U}$ 
such that $x_{\mathbb{V}}=p$. We say 
that it appears in another pattern 
$q$ on support containing $\mathbb{U}$ 
such that the restriction of $q$ 
on $\mathbb{U}$ is $p$.
We say that it appears in a subshift $X$ when it appears 
in a configuration of $X$. Such a 
pattern is also called \textbf{ globally admissible} 
for $X$.
For all $d \ge 1$, the number of patterns on 
support $\mathbb{U}^{(d)}_N \equiv 
\llbracket 1,N \rrbracket ^d$ that 
appear in a $d$-dimensional subshift 
$X$ is denoted $\mathcal{N}_N (X)$. 
When $d=2$, the number of patterns on support 
$\mathbb{U}^{(2)}_{M,N} \equiv \llbracket 1,M \rrbracket \times\llbracket 1,N \rrbracket $ that appear in $X$ 
is denoted $\mathcal{N}_{M,N}(X)$.
A $d$-dimensional subshift $X$ defined by forbidding patterns 
in some finite set $\mathcal{F}$ 
to appear in the configurations, 
formally: 
\[X= \left\{ x \in \mathcal{A}^{\Z^d} : \forall \mathbb{U} \subset \mathbb{Z}^d, 
x_{\mathbb{U}} \notin \mathcal{F} \right\}\]
is called 
a subshift of \textbf{finite type} (SFT).
In a context where the 
set of forbidden patterns 
defining the SFT is fixed, 
a pattern is called \textbf{locally admissible} for this SFT
when no forbidden pattern appears in it.
A \textbf{morphism} between two $\Z^d$-subshifts $X,Z$ is 
a continuous map $\varphi : X \rightarrow Z$ such that 
$\varphi \circ \sigma^{\vec{v}} =  \sigma^{\vec{v}} \circ \varphi$ 
for all $\vec{v} \in \Z^d$ (the map commutes with the shift action). An \textbf{isomorphism} is an invertible morphism.

\subsubsection{Topological entropy}

\begin{definition}
Let $X$ be a $d$-dimensional subshift. The \textbf{topological entropy}  
of $X$ is defined as:
\[h (X)
\equiv \inf_{N \ge 1} \frac{\log_2 
(\mathcal{N}_N (X))}{N^d}.
\]
\end{definition} 

\noindent It is a well known fact in topological dynamics
that this infimum is a limit: 
\[\boxed{h (X)
= \lim_{N \ge 1} \frac{\log_2 
(\mathcal{N}_N (X))}{N^d}}
\]

It is a topological invariant, meaning 
that when there is an isomorphism 
between two subshifts, these 
two subshifts have the same entropy~\cite{Lind-Marcus-1995}. \bigskip

\begin{definition}
\label{definition.stripes.subshift}
Let $X$ be a bidimensional subshift ($d=2$).
For all $n \ge 1$, 
we denote $X_N$ the subshift obtained from $X$ 
by restricting to the width $N$ infinite strip
$\{1,...,N\} \times \mathbb{Z}$. 
Formally, this subshift is defined
on alphabet $\mathcal{A}^{N}$ and by that 
$z \in X_N$ if and only if 
there exists $x \in X$ such 
that for all $k \in \mathbb{Z}$, 
$z_k = (x_{1,k},...,x_{N,k})$. See Figure~\ref{figure.stripes.subshift}.
\end{definition}

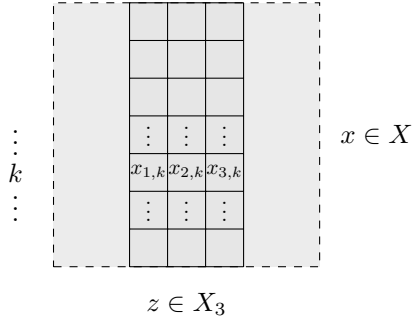
\begin{figure}[!h]
\[\begin{tikzpicture}[scale=0.5]
\fill[gray!15] (0,0) rectangle (7,7);
\fill[gray!20] (2,0) rectangle (5,7);
\draw (2,0) grid (5,7);
\draw[dashed] (0,0) rectangle (7,7);
\node at (8.5,3.5) {$x \in X$};
\node at (3.5,-1) {$z \in X_3$};
\node at (-1,2.5) {$k$};
\node[scale=0.8] at (2.5,2.5) {$x_{1,k}$};
\node[scale=0.8] at (3.5,2.5) {$x_{2,k}$};
\node[scale=0.8] at (4.5,2.5) {$x_{3,k}$};
\node at (-1,3.5) {$\vdots$};
\node at (-1,1.75) {$\vdots$};

\node[scale=0.8] at (2.5,3.625) {$\vdots$};
\node[scale=0.8] at (2.5,1.625) {$\vdots$};
\node[scale=0.8] at (3.5,3.625) {$\vdots$};
\node[scale=0.8] at (3.5,1.625) {$\vdots$};
\node[scale=0.8] at (4.5,3.625) {$\vdots$};
\node[scale=0.8] at (4.5,1.625) {$\vdots$};
\end{tikzpicture}\]
\caption{\label{figure.stripes.subshift} Illustration of Definition~\ref{definition.stripes.subshift} for $N=3$.}
\end{figure}

\begin{proposition} \label{proposition.stripes}
The entropy of $X$ can be computed 
through the sequence $(h(X_N))_N$:
\[\boxed{h(X)= \lim_{N} \frac{h(X_N)}{N}}\]
\end{proposition}

We include a proof of this statement, for completeness:

\begin{proof}
From the definition of 
$X_N$:
\[h(X)= \lim_{N} \lim_M \frac{\log_2(\mathcal{N}_{M,N} (X))}{NM}.\]
We prove this by an upper bound on 
the $\limsup_N$ and a lower bound on 
the $\liminf_N$ of the sequence in this formula.

\begin{itemize}

\item \textbf{Upper bound by cutting squares in rectangles:}

Since for any $M,N,k$, the set $\mathbb{U}^{(2)}_{kM,kN}$ 
is the union of $MN$ translates of $\mathbb{U}^{(2)}_{k}$,
a pattern on support $\mathbb{U}^{(2)}_{kM,kN}$ can 
be seen as an array of patterns on $\mathbb{U}^{(2)}_{k}$. 
As a consequence, 
\[\mathcal{N}_{kM,kN} (X)\le (\mathcal{N}_{k,k} (X)) ^{MN},\]
and using this inequality, we get:
\[\lim_{M} \frac{\log_2(\mathcal{N}_{M,N} (X))}{NM} 
= \lim_M \frac{\log_2(\mathcal{N}_{kM,kN} (X))}{k^2 NM} 
\le \lim_M \frac{\log_2(\mathcal{N}_{k,k} (X))}{k^2} = 
\frac{\log_2 (\mathcal{N}_{k,k} (X))}{k^2}.\]
As a consequence, for all $k$, 
\[\limsup_{N} \lim_M \frac{\log_2 (\mathcal{N}_{M,N} (X))}{NM} \le \frac{\log_2 (\mathcal{N}_{k,k} (X))}{k^2},\]
and this implies 
\[\limsup_{N} \lim_M \frac{\log_2 (\mathcal{N}_{M,N} (X))}{NM} \le h(X),\]
by taking $ k \rightarrow + \infty$ in the last 
inequality.

\item \textbf{Lower bound by cutting rectangles in squares:}
For all $M,N$, by considering a pattern on $\mathbb{U}^{(2)}_{MN,NM}$ 
as a union of translates of $\mathbb{U}^{(2)}_{M,N}$, 
we get that:
\[\mathcal{N}_{MN,NM} (X)\le (\mathcal{N}_{M,N} (X)) ^{MN}.\]
Thus, 
\[h(X) = \lim_M \frac{\log_2 (\mathcal{N}_{MN,NM} (X))}{M^2 N^2}
\le \lim_M \frac{\log_2(\mathcal{N}_{M,N} (X))}{M N}.\]
As a consequence, 
\[h(X) \le \liminf_{N} \lim_M 
\frac{\log_2 (\mathcal{N}_{M,N} (X))}{NM}.\]
\end{itemize}
These two inequalities prove that 
the sequence $\left( \frac{h(X_N)}{N} \right)_N$ 
converges, and that the limit is $h(X)$. 
\end{proof}

\noindent In the following, for all $N$ and $M$, 
we assimilate patterns on $\mathbb{U}^{(1)}_M$ 
of $X_N$ with patterns of $X$ on $\mathbb{U}^{(2)}_{M,N}$.

\subsection{ \label{section.representations} Representations of square ice}

The square ice can be defined as an isomorphic class 
of subshifts of finite type, whose elements 
can be thought as its representations.
The most widely used is the six vertex model 
(whose name derives from that the elements 
of the alphabet represent vertices of a regular 
grid) and is presented in Section~\ref{section.six.vertex}. In this text, we 
will use another representation, presented in 
Section~\ref{section.discrete.curves}, 
whose configurations consist of drifting 
discrete curves, representing possible 
particle trajectories.
In Section~\ref{section.toroidal}, 
we provide a proof that one can restrict 
to a particular subset of patterns 
in order to compute entropy of square ice.

\subsubsection{\label{section.six.vertex} The six vertex model}

The \textbf{six vertex model} is the subshift of finite type 
described as follows: \bigskip

\noindent \textit{\textbf{Symbols:}}
$\begin{tikzpicture}[scale=0.6,baseline=2mm]
\draw (0,0) rectangle (1,1);
\draw[-latex,line width=0.2mm] (0.5,0) -- (0.5,0.5);
\draw[-latex,line width=0.2mm] (0.5,0.5) -- (0.5,1); 
\draw[-latex,line width=0.2mm] (0,0.5) -- (0.5,0.5);
\draw[-latex,line width=0.2mm] (0.5,0.5) -- (1,0.5); 
\end{tikzpicture}, \begin{tikzpicture}[scale=0.6,baseline=2mm]
\draw (0,0) rectangle (1,1);
\draw[latex-,line width=0.2mm] (0.5,0) -- (0.5,0.5);
\draw[latex-,line width=0.2mm] (0.5,0.5) -- (0.5,1); 
\draw[-latex,line width=0.2mm] (0,0.5) -- (0.5,0.5);
\draw[-latex,line width=0.2mm] (0.5,0.5) -- (1,0.5); 
\end{tikzpicture}, \begin{tikzpicture}[scale=0.6,baseline=2mm]
\draw (0,0) rectangle (1,1);
\draw[latex-,line width=0.2mm] (0.5,0) -- (0.5,0.5);
\draw[-latex,line width=0.2mm] (0.5,0.5) -- (0.5,1); 
\draw[-latex,line width=0.2mm] (0,0.5) -- (0.5,0.5);
\draw[latex-,line width=0.2mm] (0.5,0.5) -- (1,0.5); 
\end{tikzpicture}, \begin{tikzpicture}[scale=0.6,baseline=2mm]
\draw (0,0) rectangle (1,1);
\draw[-latex,line width=0.2mm] (0.5,0) -- (0.5,0.5);
\draw[latex-,line width=0.2mm] (0.5,0.5) -- (0.5,1); 
\draw[latex-,line width=0.2mm] (0,0.5) -- (0.5,0.5);
\draw[-latex,line width=0.2mm] (0.5,0.5) -- (1,0.5); 
\end{tikzpicture}, \begin{tikzpicture}[scale=0.6,baseline=2mm]
\draw (0,0) rectangle (1,1);
\draw[latex-,line width=0.2mm] (0.5,0) -- (0.5,0.5);
\draw[latex-,line width=0.2mm] (0.5,0.5) -- (0.5,1); 
\draw[latex-,line width=0.2mm] (0,0.5) -- (0.5,0.5);
\draw[latex-,line width=0.2mm] (0.5,0.5) -- (1,0.5); 
\end{tikzpicture}, 
\begin{tikzpicture}[scale=0.6,baseline=2mm]
\draw (0,0) rectangle (1,1);
\draw[-latex,line width=0.2mm] (0.5,0) -- (0.5,0.5);
\draw[-latex,line width=0.2mm] (0.5,0.5) -- (0.5,1); 
\draw[latex-,line width=0.2mm] (0,0.5) -- (0.5,0.5);
\draw[latex-,line width=0.2mm] (0.5,0.5) -- (1,0.5); 
\end{tikzpicture}$. \bigskip

\noindent \textit{\textbf{Local rules:}} 
Considering two adjacent positions in 
$\mathbb{Z}^2$, the arrows corresponding 
to the common edge of the symbols on the 
two positions have to be directed the same way. 
For instance, the pattern $\begin{tikzpicture}[scale=0.6,baseline=2mm]
\begin{scope}
\draw (0,0) rectangle (1,1);
\draw[-latex,line width=0.2mm] (0.5,0) -- (0.5,0.5);
\draw[-latex,line width=0.2mm] (0.5,0.5) -- (0.5,1); 
\draw[-latex,line width=0.2mm] (0,0.5) -- (0.5,0.5);
\draw[-latex,line width=0.2mm] (0.5,0.5) -- (1,0.5); 
\end{scope} 

\begin{scope}[xshift=1cm]
\draw (0,0) rectangle (1,1);
\draw[latex-,line width=0.2mm] (0.5,0) -- (0.5,0.5);
\draw[latex-,line width=0.2mm] (0.5,0.5) -- (0.5,1); 
\draw[-latex,line width=0.2mm] (0,0.5) -- (0.5,0.5);
\draw[-latex,line width=0.2mm] (0.5,0.5) -- (1,0.5); 
\end{scope}
\end{tikzpicture}$ is allowed, 
while $\begin{tikzpicture}[scale=0.6,baseline=2mm]
\begin{scope}
\draw (0,0) rectangle (1,1);
\draw[-latex,line width=0.2mm] (0.5,0) -- (0.5,0.5);
\draw[-latex,line width=0.2mm] (0.5,0.5) -- (0.5,1); 
\draw[-latex,line width=0.2mm] (0,0.5) -- (0.5,0.5);
\draw[-latex,line width=0.2mm] (0.5,0.5) -- (1,0.5); 
\end{scope} 

\begin{scope}[xshift=1cm]
\draw (0,0) rectangle (1,1);
\draw[-latex,line width=0.2mm] (0.5,0) -- (0.5,0.5);
\draw[latex-,line width=0.2mm] (0.5,0.5) -- (0.5,1); 
\draw[latex-,line width=0.2mm] (0,0.5) -- (0.5,0.5);
\draw[-latex,line width=0.2mm] (0.5,0.5) -- (1,0.5); 
\end{scope}
\end{tikzpicture}$ is not. \bigskip

\noindent \textit{\textbf{Global behavior:}} 
The symbols draw a lattice whose 
edges are oriented in such a way 
that all the vertices have two incoming 
arrows and two outgoing ones. This 
is called an Eulerian orientation 
of the square lattice. See an example 
of admissible pattern 
on Figure~\ref{figure.six.vertex}.

\begin{figure}[h!]
\[\begin{tikzpicture}[scale=0.6]

\begin{scope}
\draw (0,0) rectangle (1,1);
\draw[latex-,line width=0.2mm] (0.5,0) -- (0.5,0.5);
\draw[latex-,line width=0.2mm] (0.5,0.5) -- (0.5,1); 
\draw[-latex,line width=0.2mm] (0,0.5) -- (0.5,0.5);
\draw[-latex,line width=0.2mm] (0.5,0.5) -- (1,0.5); 
\end{scope} 

\begin{scope}[xshift=1cm]
\draw (0,0) rectangle (1,1);
\draw[latex-,line width=0.2mm] (0.5,0) -- (0.5,0.5);
\draw[-latex,line width=0.2mm] (0.5,0.5) -- (0.5,1); 
\draw[-latex,line width=0.2mm] (0,0.5) -- (0.5,0.5);
\draw[latex-,line width=0.2mm] (0.5,0.5) -- (1,0.5); 
\end{scope}

\begin{scope}[xshift=2cm]
\draw (0,0) rectangle (1,1);
\draw[latex-,line width=0.2mm] (0.5,0) -- (0.5,0.5);
\draw[latex-,line width=0.2mm] (0.5,0.5) -- (0.5,1); 
\draw[latex-,line width=0.2mm] (0,0.5) -- (0.5,0.5);
\draw[latex-,line width=0.2mm] (0.5,0.5) -- (1,0.5); 
\end{scope}

\begin{scope}[xshift=3cm]
\draw (0,0) rectangle (1,1);
\draw[latex-,line width=0.2mm] (0.5,0) -- (0.5,0.5);
\draw[latex-,line width=0.2mm] (0.5,0.5) -- (0.5,1); 
\draw[latex-,line width=0.2mm] (0,0.5) -- (0.5,0.5);
\draw[latex-,line width=0.2mm] (0.5,0.5) -- (1,0.5); 
\end{scope}

\begin{scope}[yshift=1cm]
\draw (0,0) rectangle (1,1);
\draw[latex-,line width=0.2mm] (0.5,0) -- (0.5,0.5);
\draw[latex-,line width=0.2mm] (0.5,0.5) -- (0.5,1); 
\draw[-latex,line width=0.2mm] (0,0.5) -- (0.5,0.5);
\draw[-latex,line width=0.2mm] (0.5,0.5) -- (1,0.5); 
\end{scope} 

\begin{scope}[xshift=1cm,yshift=1cm]
\draw (0,0) rectangle (1,1);
\draw[-latex,line width=0.2mm] (0.5,0) -- (0.5,0.5);
\draw[-latex,line width=0.2mm] (0.5,0.5) -- (0.5,1); 
\draw[-latex,line width=0.2mm] (0,0.5) -- (0.5,0.5);
\draw[-latex,line width=0.2mm] (0.5,0.5) -- (1,0.5); 
\end{scope}

\begin{scope}[xshift=2cm,yshift=1cm]
\draw (0,0) rectangle (1,1);
\draw[latex-,line width=0.2mm] (0.5,0) -- (0.5,0.5);
\draw[-latex,line width=0.2mm] (0.5,0.5) -- (0.5,1); 
\draw[-latex,line width=0.2mm] (0,0.5) -- (0.5,0.5);
\draw[latex-,line width=0.2mm] (0.5,0.5) -- (1,0.5); 
\end{scope}

\begin{scope}[xshift=3cm,yshift=1cm]
\draw (0,0) rectangle (1,1);
\draw[latex-,line width=0.2mm] (0.5,0) -- (0.5,0.5);
\draw[latex-,line width=0.2mm] (0.5,0.5) -- (0.5,1); 
\draw[latex-,line width=0.2mm] (0,0.5) -- (0.5,0.5);
\draw[latex-,line width=0.2mm] (0.5,0.5) -- (1,0.5); 
\end{scope}

\begin{scope}[yshift=2cm]
\draw (0,0) rectangle (1,1);
\draw[latex-,line width=0.2mm] (0.5,0) -- (0.5,0.5);
\draw[-latex,line width=0.2mm] (0.5,0.5) -- (0.5,1); 
\draw[-latex,line width=0.2mm] (0,0.5) -- (0.5,0.5);
\draw[latex-,line width=0.2mm] (0.5,0.5) -- (1,0.5); 
\end{scope}

\begin{scope}[xshift=1cm,yshift=2cm]
\draw (0,0) rectangle (1,1);
\draw[-latex,line width=0.2mm] (0.5,0) -- (0.5,0.5);
\draw[latex-,line width=0.2mm] (0.5,0.5) -- (0.5,1); 
\draw[latex-,line width=0.2mm] (0,0.5) -- (0.5,0.5);
\draw[-latex,line width=0.2mm] (0.5,0.5) -- (1,0.5); 
\end{scope}

\begin{scope}[xshift=2cm,yshift=2cm]
\draw (0,0) rectangle (1,1);
\draw[-latex,line width=0.2mm] (0.5,0) -- (0.5,0.5);
\draw[-latex,line width=0.2mm] (0.5,0.5) -- (0.5,1); 
\draw[-latex,line width=0.2mm] (0,0.5) -- (0.5,0.5);
\draw[-latex,line width=0.2mm] (0.5,0.5) -- (1,0.5); 
\end{scope}

\begin{scope}[xshift=3cm,yshift=2cm]
\draw (0,0) rectangle (1,1);
\draw[latex-,line width=0.2mm] (0.5,0) -- (0.5,0.5);
\draw[-latex,line width=0.2mm] (0.5,0.5) -- (0.5,1); 
\draw[-latex,line width=0.2mm] (0,0.5) -- (0.5,0.5);
\draw[latex-,line width=0.2mm] (0.5,0.5) -- (1,0.5); 
\end{scope}

\begin{scope}[xshift=0cm,yshift=3cm]
\draw (0,0) rectangle (1,1);
\draw[-latex,line width=0.2mm] (0.5,0) -- (0.5,0.5);
\draw[-latex,line width=0.2mm] (0.5,0.5) -- (0.5,1); 
\draw[-latex,line width=0.2mm] (0,0.5) -- (0.5,0.5);
\draw[-latex,line width=0.2mm] (0.5,0.5) -- (1,0.5); 
\end{scope}

\begin{scope}[xshift=1cm,yshift=3cm]
\draw (0,0) rectangle (1,1);
\draw[latex-,line width=0.2mm] (0.5,0) -- (0.5,0.5);
\draw[-latex,line width=0.2mm] (0.5,0.5) -- (0.5,1); 
\draw[-latex,line width=0.2mm] (0,0.5) -- (0.5,0.5);
\draw[latex-,line width=0.2mm] (0.5,0.5) -- (1,0.5); 
\end{scope}

\begin{scope}[xshift=2cm,yshift=3cm]
\draw (0,0) rectangle (1,1);
\draw[-latex,line width=0.2mm] (0.5,0) -- (0.5,0.5);
\draw[-latex,line width=0.2mm] (0.5,0.5) -- (0.5,1); 
\draw[latex-,line width=0.2mm] (0,0.5) -- (0.5,0.5);
\draw[latex-,line width=0.2mm] (0.5,0.5) -- (1,0.5); 
\end{scope}

\begin{scope}[xshift=3cm,yshift=3cm]
\draw (0,0) rectangle (1,1);
\draw[-latex,line width=0.2mm] (0.5,0) -- (0.5,0.5);
\draw[-latex,line width=0.2mm] (0.5,0.5) -- (0.5,1); 
\draw[latex-,line width=0.2mm] (0,0.5) -- (0.5,0.5);
\draw[latex-,line width=0.2mm] (0.5,0.5) -- (1,0.5); 
\end{scope}

\end{tikzpicture}\]
\caption{\label{figure.six.vertex} An example of locally and thus globally 
admissible pattern of the six vertex model.}
\end{figure}
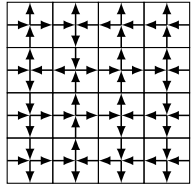

\begin{remark}
The name of square ice of the class of SFT appears 
clearly when considering the following application 
on the alphabet of the six vertex model to local
configurations of dihydrogen monoxide: 

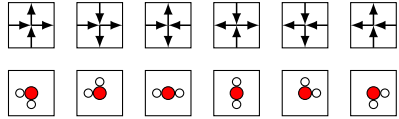
\[\begin{tikzpicture}[scale=0.6]
\begin{scope}
\draw (0,0) rectangle (1,1);
\draw[-latex,line width=0.2mm] (0.5,0) -- (0.5,0.5);
\draw[-latex,line width=0.2mm] (0.5,0.5) -- (0.5,1); 
\draw[-latex,line width=0.2mm] (0,0.5) -- (0.5,0.5);
\draw[-latex,line width=0.2mm] (0.5,0.5) -- (1,0.5);
\end{scope}

\begin{scope}[xshift=1.5cm]
\draw (0,0) rectangle (1,1);
\draw[latex-,line width=0.2mm] (0.5,0) -- (0.5,0.5);
\draw[latex-,line width=0.2mm] (0.5,0.5) -- (0.5,1); 
\draw[-latex,line width=0.2mm] (0,0.5) -- (0.5,0.5);
\draw[-latex,line width=0.2mm] (0.5,0.5) -- (1,0.5); 
\end{scope}

\begin{scope}[xshift=3cm]
\draw (0,0) rectangle (1,1);
\draw[latex-,line width=0.2mm] (0.5,0) -- (0.5,0.5);
\draw[-latex,line width=0.2mm] (0.5,0.5) -- (0.5,1); 
\draw[-latex,line width=0.2mm] (0,0.5) -- (0.5,0.5);
\draw[latex-,line width=0.2mm] (0.5,0.5) -- (1,0.5); 
\end{scope}

\begin{scope}[xshift=4.5cm]
\draw (0,0) rectangle (1,1);
\draw[-latex,line width=0.2mm] (0.5,0) -- (0.5,0.5);
\draw[latex-,line width=0.2mm] (0.5,0.5) -- (0.5,1); 
\draw[latex-,line width=0.2mm] (0,0.5) -- (0.5,0.5);
\draw[-latex,line width=0.2mm] (0.5,0.5) -- (1,0.5); 
\end{scope} 

\begin{scope}[xshift=6cm]
\draw (0,0) rectangle (1,1);
\draw[latex-,line width=0.2mm] (0.5,0) -- (0.5,0.5);
\draw[latex-,line width=0.2mm] (0.5,0.5) -- (0.5,1); 
\draw[latex-,line width=0.2mm] (0,0.5) -- (0.5,0.5);
\draw[latex-,line width=0.2mm] (0.5,0.5) -- (1,0.5); 
\end{scope}

\begin{scope}[xshift=7.5cm]
\draw (0,0) rectangle (1,1);
\draw[-latex,line width=0.2mm] (0.5,0) -- (0.5,0.5);
\draw[-latex,line width=0.2mm] (0.5,0.5) -- (0.5,1); 
\draw[latex-,line width=0.2mm] (0,0.5) -- (0.5,0.5);
\draw[latex-,line width=0.2mm] (0.5,0.5) -- (1,0.5); 
\end{scope}

\begin{scope}[yshift=-1.5cm]
\draw (0,0) rectangle (1,1);
\fill[red] (0.5,0.5) circle (4pt);
\draw (0.5,0.5) circle (4pt);
\draw[fill=white] (0.25,0.5) circle (2.6pt);
\draw[fill=white] (0.5,0.25) circle (2.6pt);
\end{scope}

\begin{scope}[yshift=-1.5cm,xshift=1.5cm]
\draw (0,0) rectangle (1,1);
\fill[red] (0.5,0.5) circle (4pt);
\draw (0.5,0.5) circle (4pt);
\draw[fill=white] (0.25,0.5) circle (2.6pt);
\draw[fill=white] (0.5,0.75) circle (2.6pt);
\end{scope}

\begin{scope}[yshift=-1.5cm,xshift=3cm]
\draw (0,0) rectangle (1,1);
\fill[red] (0.5,0.5) circle (4pt);
\draw (0.5,0.5) circle (4pt);
\draw[fill=white] (0.25,0.5) circle (2.6pt);
\draw[fill=white] (0.75,0.5) circle (2.6pt);
\end{scope}

\begin{scope}[yshift=-1.5cm,xshift=4.5cm]
\draw (0,0) rectangle (1,1);
\fill[red] (0.5,0.5) circle (4pt);
\draw (0.5,0.5) circle (4pt);
\draw[fill=white] (0.5,0.75) circle (2.6pt);
\draw[fill=white] (0.5,0.25) circle (2.6pt);
\end{scope}

\begin{scope}[yshift=-1.5cm,xshift=6cm]
\draw (0,0) rectangle (1,1);
\fill[red] (0.5,0.5) circle (4pt);
\draw (0.5,0.5) circle (4pt);
\draw[fill=white] (0.5,0.75) circle (2.6pt);
\draw[fill=white] (0.75,0.5) circle (2.6pt);
\end{scope}

\begin{scope}[yshift=-1.5cm,xshift=7.5cm]
\draw (0,0) rectangle (1,1);
\fill[red] (0.5,0.5) circle (4pt);
\draw (0.5,0.5) circle (4pt);
\draw[fill=white] (0.5,0.25) circle (2.6pt);
\draw[fill=white] (0.75,0.5) circle (2.6pt);
\end{scope}
\end{tikzpicture}\]

\end{remark}

\subsubsection{\label{section.discrete.curves} Drifting discrete curves}

From the six vertex model, we derive another 
representation of square ice
through an isomorphism, 
which consist in transforming 
the letters via an application $\pi_s$ 
on the alphabet of the six vertex model, 
described as follows:

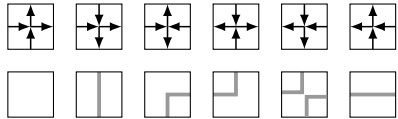
\[\begin{tikzpicture}[scale=0.6]
\begin{scope}
\draw (0,0) rectangle (1,1);
\draw[-latex,line width=0.2mm] (0.5,0) -- (0.5,0.5);
\draw[-latex,line width=0.2mm] (0.5,0.5) -- (0.5,1); 
\draw[-latex,line width=0.2mm] (0,0.5) -- (0.5,0.5);
\draw[-latex,line width=0.2mm] (0.5,0.5) -- (1,0.5);
\end{scope}

\begin{scope}[xshift=1.5cm]
\draw (0,0) rectangle (1,1);
\draw[latex-,line width=0.2mm] (0.5,0) -- (0.5,0.5);
\draw[latex-,line width=0.2mm] (0.5,0.5) -- (0.5,1); 
\draw[-latex,line width=0.2mm] (0,0.5) -- (0.5,0.5);
\draw[-latex,line width=0.2mm] (0.5,0.5) -- (1,0.5); 
\end{scope}

\begin{scope}[xshift=3cm]
\draw (0,0) rectangle (1,1);
\draw[latex-,line width=0.2mm] (0.5,0) -- (0.5,0.5);
\draw[-latex,line width=0.2mm] (0.5,0.5) -- (0.5,1); 
\draw[-latex,line width=0.2mm] (0,0.5) -- (0.5,0.5);
\draw[latex-,line width=0.2mm] (0.5,0.5) -- (1,0.5); 
\end{scope}

\begin{scope}[xshift=4.5cm]
\draw (0,0) rectangle (1,1);
\draw[-latex,line width=0.2mm] (0.5,0) -- (0.5,0.5);
\draw[latex-,line width=0.2mm] (0.5,0.5) -- (0.5,1); 
\draw[latex-,line width=0.2mm] (0,0.5) -- (0.5,0.5);
\draw[-latex,line width=0.2mm] (0.5,0.5) -- (1,0.5); 
\end{scope} 

\begin{scope}[xshift=6cm]
\draw (0,0) rectangle (1,1);
\draw[latex-,line width=0.2mm] (0.5,0) -- (0.5,0.5);
\draw[latex-,line width=0.2mm] (0.5,0.5) -- (0.5,1); 
\draw[latex-,line width=0.2mm] (0,0.5) -- (0.5,0.5);
\draw[latex-,line width=0.2mm] (0.5,0.5) -- (1,0.5); 
\end{scope}

\begin{scope}[xshift=7.5cm]
\draw (0,0) rectangle (1,1);
\draw[-latex,line width=0.2mm] (0.5,0) -- (0.5,0.5);
\draw[-latex,line width=0.2mm] (0.5,0.5) -- (0.5,1); 
\draw[latex-,line width=0.2mm] (0,0.5) -- (0.5,0.5);
\draw[latex-,line width=0.2mm] (0.5,0.5) -- (1,0.5); 
\end{scope}

\begin{scope}[yshift=-1.5cm]
\draw (0,0) rectangle (1,1);
\end{scope}

\begin{scope}[yshift=-1.5cm,xshift=1.5cm]
\draw[line width = 0.5mm,color=gray!80]
(0.5,0) -- (0.5,1);
\draw (0,0) rectangle (1,1);
\end{scope}

\begin{scope}[yshift=-1.5cm,xshift=3cm]
\draw[line width = 0.5mm,color=gray!80]
(0.5,0) -- (0.5,0.5) -- (1,0.5);
\draw (0,0) rectangle (1,1);
\end{scope}

\begin{scope}[yshift=-1.5cm,xshift=4.5cm]
\draw[line width = 0.5mm,color=gray!80]
(0,0.5) -- (0.5,0.5) -- (0.5,1);
\draw (0,0) rectangle (1,1);
\end{scope}

\begin{scope}[yshift=-1.5cm,xshift=6cm]
\draw[line width = 0.5mm,color=gray!80]
(0,0.55) -- (0.45,0.55) -- (0.45,1);
\draw (0,0) rectangle (1,1);
\draw[line width = 0.5mm,color=gray!80] (0.55,0) -- (0.55,0.45) -- (1,0.45);
\end{scope}

\begin{scope}[yshift=-1.5cm,xshift=7.5cm]
\draw[line width = 0.5mm,color=gray!80]
(0,0.5) -- (1,0.5);
\draw (0,0) rectangle (1,1);
\end{scope}
\end{tikzpicture}\]

The pattern on Figure~\ref{figure.six.vertex} can 
be represented as on Figure~\ref{figure.six.vertex.curves}. 
In this SFT, the local rules are that 
any outgoing segment of curve in a 
non-blank symbol is extended 
in its direction on the next position.

\begin{figure}[h!]
\[\begin{tikzpicture}[scale=0.6]
\begin{scope}
\draw[color=gray!80,line width=0.4mm] (0.5,0) -- 
(0.5,1);
\draw (0,0) rectangle (1,1);
\end{scope} 

\begin{scope}[xshift=1cm]
\draw[line width = 0.5mm,color=gray!80]
(0.55,0) -- (0.55,0.55) -- (1,0.55);
\draw (0,0) rectangle (1,1);
\end{scope}

\begin{scope}[xshift=2cm]
\draw[line width = 0.5mm,color=gray!80]
(0,0.55) -- (0.45,0.55) -- (0.45,1);
\draw (0,0) rectangle (1,1);
\draw[line width = 0.5mm,color=gray!80] (0.55,0) -- (0.55,0.55) -- (1,0.55);
\end{scope}
\begin{scope}[xshift=3cm]
\draw[line width = 0.5mm,color=gray!80]
(0,0.55) -- (0.45,0.55) -- (0.45,1);
\draw (0,0) rectangle (1,1);
\draw[line width = 0.5mm,color=gray!80] (0.55,0) -- (0.55,0.55) -- (1,0.55);
\end{scope}

\begin{scope}[yshift=1cm]
\draw[color=gray!80,line width=0.4mm] (0.5,0) -- 
(0.5,1);
\draw (0,0) rectangle (1,1);
\end{scope} 

\begin{scope}[xshift=1cm,yshift=1cm]
\draw (0,0) rectangle (1,1);
\end{scope}

\begin{scope}[xshift=2cm,yshift=1cm]
\draw[line width = 0.5mm,color=gray!80]
(0.45,0) -- (0.45,0.55) -- (1,0.55);
\draw (0,0) rectangle (1,1);
\end{scope}

\begin{scope}[xshift=3cm,yshift=1cm]
\draw[line width = 0.5mm,color=gray!80]
(0,0.55) -- (0.45,0.55) -- (0.45,1);
\draw (0,0) rectangle (1,1);
\draw[line width = 0.5mm,color=gray!80] (0.45,0) -- (0.45,0.45) -- (1,0.45);
\end{scope}

\begin{scope}[xshift=0cm,yshift=2cm]
\draw[line width = 0.5mm,color=gray!80]
(0.5,0) -- (0.5,0.5) -- (1,0.5);
\draw (0,0) rectangle (1,1);
\end{scope}

\begin{scope}[xshift=1cm,yshift=2cm]
\draw[line width = 0.5mm,color=gray!80]
(0,0.5) -- (0.5,0.5) -- (0.5,1);
\draw (0,0) rectangle (1,1);
\end{scope}

\begin{scope}[xshift=2cm,yshift=2cm]
\draw (0,0) rectangle (1,1);
\end{scope}

\begin{scope}[xshift=3cm,yshift=2cm]
\draw[line width = 0.5mm,color=gray!80]
(0.45,0) -- (0.45,0.5) -- (1,0.5);
\draw (0,0) rectangle (1,1);
\end{scope}

\begin{scope}[xshift=0cm,yshift=3cm]
\draw (0,0) rectangle (1,1);
\end{scope}

\begin{scope}[xshift=1cm,yshift=3cm]
\draw[line width = 0.5mm,color=gray!80]
(0.5,0) -- (0.5,0.5) -- (1,0.5);
\draw (0,0) rectangle (1,1);
\end{scope}

\begin{scope}[xshift=2cm,yshift=3cm]
\draw[line width = 0.5mm,color=gray!80]
(0,0.5) -- (1,0.5);
\draw (0,0) rectangle (1,1);
\end{scope}

\begin{scope}[xshift=3cm,yshift=3cm]
\draw[line width = 0.5mm,color=gray!80]
(0,0.5) -- (1,0.5);
\draw (0,0) rectangle (1,1);
\end{scope}

\end{tikzpicture}\]
\caption{\label{figure.six.vertex.curves} 
Representation of pattern on Figure~\ref{figure.six.vertex}.}
\end{figure}
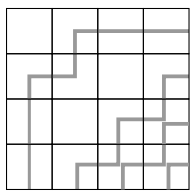

In the following, we denote $X^s$ this 
SFT.

\begin{remark}
One can see straightforwardly that locally admissible patterns 
of this SFT are always globally admissible, 
since any locally admissible pattern 
can be extended into a configuration by 
extending the curves in a straight way. 
\end{remark}

\subsubsection{\label{section.toroidal} Entropy of $X^s$ and cylindric stripes subshifts of square ice}

Consider some alphabet $\mathcal{A}$, and $X$ 
a bidimensional subshift of finite type on 
this alphabet.
For all $N \ge 1$, we denote $\Pi_{N}= 
\mathbb{Z} / (N\mathbb{Z}) \times \mathbb{Z}$.
Let us also denote $\pi_N : \llbracket 1,N\rrbracket \times \mathbb{Z} \rightarrow \Pi_{N}$ 
the canonical projection, 
and $\phi_{N} : \mathcal{A} ^{\llbracket 1,N\rrbracket \times \mathbb{Z} } \rightarrow \mathcal{A}^{\Pi_{N}}$
the application that wraps configurations of $X_N$ on the infinite cylinder $\Pi_{N}$. Formally, for all $\vec{u} \in \llbracket 1,N\rrbracket \times \mathbb{Z}$ and $x \in \mathcal{A}^{\llbracket 1,N\rrbracket \times \mathbb{Z}}$,
\[(\phi_{N} (x))_{\pi_{N} (\vec{u})} = x_{\vec{u}}. \]

We say that a pattern $p$ 
on support $\mathbb{U} \subset 
\llbracket 1,N\rrbracket \times \mathbb{Z} $ 
appears in a configuration $\overline{x}$ on 
$\Pi_{N}$ when there exists 
a configuration in $X_N$ whose image by 
$\pi_{N}$ is $\overline{x}$ and 
there exists an element 
$\vec{u} \in \Pi_{N}$ 
such that for 
all $\vec{v} \in \mathbb{U},
{\overline{x}}_{\vec{u}+\pi_{N} (\vec{v})}=x_{\vec{v}}$.
 
\begin{notation}
Let us denote $\overline{X}_{N}$
the set of configurations in $X_{N}$ 
whose image by $\phi_{N}$ does not contain 
any forbidden pattern for $X$ (in other words 
this pattern can be wrapped 
on an infinite cylinder without breaking 
the rules defining $X$).
\end{notation}

We also call $(M,N)$-\textbf{cylindric pattern}
of $X$ a pattern on $\mathbb{U}_{M,N}$ that can be wrapped on a finite cylinder $\mathbb{Z}/N\mathbb{Z} \times \{1,...,M\}$.
Let us prove a preliminary 
result on entropy of square ice, which 
relates the entropy of $X^s$ to the sequence 
$(h(\overline{X}^s_N))_N$: 

\begin{lemma} \label{lemma.toroidal.ice}
The subshift $X^s$ has 
entropy equal to \[\boxed{h (X^s) = \lim_N  \frac{h(\overline{X}^s_N)}{N}}.\]
\end{lemma}

\begin{remark}
In order to prove this lemma, we use a technique 
that first appeared in a work of S.Friedland~\cite{F97}, that relies on a symmetry of 
the alphabet and rules of the SFT.
\end{remark}

\begin{proof}

\begin{enumerate}
\item \textbf{Lower bound:}
Since for all $N$, $\overline{X}^s_N \subset X^s_{N}$, then $h(\overline{X}^s_N) \le h(X^s_{N})$. We deduce by Proposition~\ref{proposition.stripes}, 
that \[\limsup_N \frac{h(\overline{X}^s_N)}{N} 
\le h(X^s).\]

\item \textbf{Upper bound:}

Consider the transformation $\tau$ on the six-vertex model alphabet 
that consists in a horizontal 
symmetry of the symbols and then the inversion 
of all the arrows. The symmetry can be represented 
as follows:

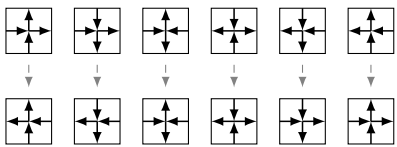
\[\begin{tikzpicture}[scale=0.6]
\begin{scope}
\draw (0,0) rectangle (1,1);
\draw[-latex,line width=0.2mm] (0.5,0) -- (0.5,0.5);
\draw[-latex,line width=0.2mm] (0.5,0.5) -- (0.5,1); 
\draw[-latex,line width=0.2mm] (0,0.5) -- (0.5,0.5);
\draw[-latex,line width=0.2mm] (0.5,0.5) -- (1,0.5);
\end{scope}

\begin{scope}[xshift=1.5cm]
\draw (0,0) rectangle (1,1);
\draw[latex-,line width=0.2mm] (0.5,0) -- (0.5,0.5);
\draw[latex-,line width=0.2mm] (0.5,0.5) -- (0.5,1); 
\draw[-latex,line width=0.2mm] (0,0.5) -- (0.5,0.5);
\draw[-latex,line width=0.2mm] (0.5,0.5) -- (1,0.5); 
\end{scope}

\begin{scope}[xshift=3cm]
\draw (0,0) rectangle (1,1);
\draw[latex-,line width=0.2mm] (0.5,0) -- (0.5,0.5);
\draw[-latex,line width=0.2mm] (0.5,0.5) -- (0.5,1); 
\draw[-latex,line width=0.2mm] (0,0.5) -- (0.5,0.5);
\draw[latex-,line width=0.2mm] (0.5,0.5) -- (1,0.5); 
\end{scope}

\begin{scope}[xshift=4.5cm]
\draw (0,0) rectangle (1,1);
\draw[-latex,line width=0.2mm] (0.5,0) -- (0.5,0.5);
\draw[latex-,line width=0.2mm] (0.5,0.5) -- (0.5,1); 
\draw[latex-,line width=0.2mm] (0,0.5) -- (0.5,0.5);
\draw[-latex,line width=0.2mm] (0.5,0.5) -- (1,0.5); 
\end{scope} 

\begin{scope}[xshift=6cm]
\draw (0,0) rectangle (1,1);
\draw[latex-,line width=0.2mm] (0.5,0) -- (0.5,0.5);
\draw[latex-,line width=0.2mm] (0.5,0.5) -- (0.5,1); 
\draw[latex-,line width=0.2mm] (0,0.5) -- (0.5,0.5);
\draw[latex-,line width=0.2mm] (0.5,0.5) -- (1,0.5); 
\end{scope}

\begin{scope}[xshift=7.5cm]
\draw (0,0) rectangle (1,1);
\draw[-latex,line width=0.2mm] (0.5,0) -- (0.5,0.5);
\draw[-latex,line width=0.2mm] (0.5,0.5) -- (0.5,1); 
\draw[latex-,line width=0.2mm] (0,0.5) -- (0.5,0.5);
\draw[latex-,line width=0.2mm] (0.5,0.5) -- (1,0.5); 
\end{scope}

\foreach \x in {0,1,2,3,4,5} {
\draw[dashed,-latex,gray!98] (\x*1.5+0.5,-0.25) -- (\x*1.5+0.5,-0.75);}

\begin{scope}[yshift=-0.5cm]

\begin{scope}[yshift=-1.5cm]
\draw (0,0) rectangle (1,1);
\draw[-latex,line width=0.2mm] (0.5,0) -- (0.5,0.5);
\draw[-latex,line width=0.2mm] (0.5,0.5) -- (0.5,1); 
\draw[latex-,line width=0.2mm] (0,0.5) -- (0.5,0.5);
\draw[latex-,line width=0.2mm] (0.5,0.5) -- (1,0.5);
\end{scope}

\begin{scope}[yshift=-1.5cm,xshift=1.5cm]
\draw (0,0) rectangle (1,1);
\draw[latex-,line width=0.2mm] (0.5,0) -- (0.5,0.5);
\draw[latex-,line width=0.2mm] (0.5,0.5) -- (0.5,1); 
\draw[latex-,line width=0.2mm] (0,0.5) -- (0.5,0.5);
\draw[latex-,line width=0.2mm] (0.5,0.5) -- (1,0.5); 
\end{scope}

\begin{scope}[yshift=-1.5cm,xshift=3cm]
\draw (0,0) rectangle (1,1);
\draw[latex-,line width=0.2mm] (0.5,0) -- (0.5,0.5);
\draw[-latex,line width=0.2mm] (0.5,0.5) -- (0.5,1); 
\draw[-latex,line width=0.2mm] (0,0.5) -- (0.5,0.5);
\draw[latex-,line width=0.2mm] (0.5,0.5) -- (1,0.5); 
\end{scope}

\begin{scope}[yshift=-1.5cm,xshift=4.5cm]
\draw (0,0) rectangle (1,1);
\draw[-latex,line width=0.2mm] (0.5,0) -- (0.5,0.5);
\draw[latex-,line width=0.2mm] (0.5,0.5) -- (0.5,1); 
\draw[latex-,line width=0.2mm] (0,0.5) -- (0.5,0.5);
\draw[-latex,line width=0.2mm] (0.5,0.5) -- (1,0.5);
\end{scope}

\begin{scope}[yshift=-1.5cm,xshift=6cm]
\draw (0,0) rectangle (1,1);
\draw[latex-,line width=0.2mm] (0.5,0) -- (0.5,0.5);
\draw[latex-,line width=0.2mm] (0.5,0.5) -- (0.5,1); 
\draw[-latex,line width=0.2mm] (0,0.5) -- (0.5,0.5);
\draw[-latex,line width=0.2mm] (0.5,0.5) -- (1,0.5);
\end{scope}

\begin{scope}[yshift=-1.5cm,xshift=7.5cm]
\draw (0,0) rectangle (1,1);
\draw[-latex,line width=0.2mm] (0.5,0) -- (0.5,0.5);
\draw[-latex,line width=0.2mm] (0.5,0.5) -- (0.5,1); 
\draw[-latex,line width=0.2mm] (0,0.5) -- (0.5,0.5);
\draw[-latex,line width=0.2mm] (0.5,0.5) -- (1,0.5);
\end{scope}
\end{scope}
\end{tikzpicture}\]

The inversion is represented: 

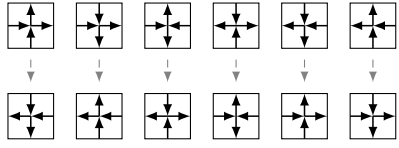
\[\begin{tikzpicture}[scale=0.6]
\begin{scope}
\draw (0,0) rectangle (1,1);
\draw[-latex,line width=0.2mm] (0.5,0) -- (0.5,0.5);
\draw[-latex,line width=0.2mm] (0.5,0.5) -- (0.5,1); 
\draw[-latex,line width=0.2mm] (0,0.5) -- (0.5,0.5);
\draw[-latex,line width=0.2mm] (0.5,0.5) -- (1,0.5);
\end{scope}

\begin{scope}[xshift=1.5cm]
\draw (0,0) rectangle (1,1);
\draw[latex-,line width=0.2mm] (0.5,0) -- (0.5,0.5);
\draw[latex-,line width=0.2mm] (0.5,0.5) -- (0.5,1); 
\draw[-latex,line width=0.2mm] (0,0.5) -- (0.5,0.5);
\draw[-latex,line width=0.2mm] (0.5,0.5) -- (1,0.5); 
\end{scope}

\begin{scope}[xshift=3cm]
\draw (0,0) rectangle (1,1);
\draw[latex-,line width=0.2mm] (0.5,0) -- (0.5,0.5);
\draw[-latex,line width=0.2mm] (0.5,0.5) -- (0.5,1); 
\draw[-latex,line width=0.2mm] (0,0.5) -- (0.5,0.5);
\draw[latex-,line width=0.2mm] (0.5,0.5) -- (1,0.5); 
\end{scope}

\begin{scope}[xshift=4.5cm]
\draw (0,0) rectangle (1,1);
\draw[-latex,line width=0.2mm] (0.5,0) -- (0.5,0.5);
\draw[latex-,line width=0.2mm] (0.5,0.5) -- (0.5,1); 
\draw[latex-,line width=0.2mm] (0,0.5) -- (0.5,0.5);
\draw[-latex,line width=0.2mm] (0.5,0.5) -- (1,0.5); 
\end{scope} 

\begin{scope}[xshift=6cm]
\draw (0,0) rectangle (1,1);
\draw[latex-,line width=0.2mm] (0.5,0) -- (0.5,0.5);
\draw[latex-,line width=0.2mm] (0.5,0.5) -- (0.5,1); 
\draw[latex-,line width=0.2mm] (0,0.5) -- (0.5,0.5);
\draw[latex-,line width=0.2mm] (0.5,0.5) -- (1,0.5); 
\end{scope}

\begin{scope}[xshift=7.5cm]
\draw (0,0) rectangle (1,1);
\draw[-latex,line width=0.2mm] (0.5,0) -- (0.5,0.5);
\draw[-latex,line width=0.2mm] (0.5,0.5) -- (0.5,1); 
\draw[latex-,line width=0.2mm] (0,0.5) -- (0.5,0.5);
\draw[latex-,line width=0.2mm] (0.5,0.5) -- (1,0.5); 
\end{scope}

\foreach \x in {0,1,2,3,4,5} {
\draw[dashed,-latex,gray!98] (\x*1.5+0.5,-0.25) -- (\x*1.5+0.5,-0.75);}

\begin{scope}[yshift=-0.5cm]

\begin{scope}[yshift=-1.5cm]
\draw (0,0) rectangle (1,1);
\draw[latex-,line width=0.2mm] (0.5,0) -- (0.5,0.5);
\draw[latex-,line width=0.2mm] (0.5,0.5) -- (0.5,1); 
\draw[latex-,line width=0.2mm] (0,0.5) -- (0.5,0.5);
\draw[latex-,line width=0.2mm] (0.5,0.5) -- (1,0.5);
\end{scope}

\begin{scope}[yshift=-1.5cm,xshift=1.5cm]
\draw (0,0) rectangle (1,1);
\draw[-latex,line width=0.2mm] (0.5,0) -- (0.5,0.5);
\draw[-latex,line width=0.2mm] (0.5,0.5) -- (0.5,1); 
\draw[latex-,line width=0.2mm] (0,0.5) -- (0.5,0.5);
\draw[latex-,line width=0.2mm] (0.5,0.5) -- (1,0.5); 
\end{scope}

\begin{scope}[yshift=-1.5cm,xshift=3cm]
\draw (0,0) rectangle (1,1);
\draw[-latex,line width=0.2mm] (0.5,0) -- (0.5,0.5);
\draw[latex-,line width=0.2mm] (0.5,0.5) -- (0.5,1); 
\draw[latex-,line width=0.2mm] (0,0.5) -- (0.5,0.5);
\draw[-latex,line width=0.2mm] (0.5,0.5) -- (1,0.5); 
\end{scope}

\begin{scope}[yshift=-1.5cm,xshift=4.5cm]
\draw (0,0) rectangle (1,1);
\draw[latex-,line width=0.2mm] (0.5,0) -- (0.5,0.5);
\draw[-latex,line width=0.2mm] (0.5,0.5) -- (0.5,1); 
\draw[-latex,line width=0.2mm] (0,0.5) -- (0.5,0.5);
\draw[latex-,line width=0.2mm] (0.5,0.5) -- (1,0.5);
\end{scope}

\begin{scope}[yshift=-1.5cm,xshift=6cm]
\draw (0,0) rectangle (1,1);
\draw[-latex,line width=0.2mm] (0.5,0) -- (0.5,0.5);
\draw[-latex,line width=0.2mm] (0.5,0.5) -- (0.5,1); 
\draw[-latex,line width=0.2mm] (0,0.5) -- (0.5,0.5);
\draw[-latex,line width=0.2mm] (0.5,0.5) -- (1,0.5);
\end{scope}

\begin{scope}[yshift=-1.5cm,xshift=7.5cm]
\draw (0,0) rectangle (1,1);
\draw[latex-,line width=0.2mm] (0.5,0) -- (0.5,0.5);
\draw[latex-,line width=0.2mm] (0.5,0.5) -- (0.5,1); 
\draw[-latex,line width=0.2mm] (0,0.5) -- (0.5,0.5);
\draw[-latex,line width=0.2mm] (0.5,0.5) -- (1,0.5);
\end{scope} \end{scope}
\end{tikzpicture}\]

As a consequence $\tau$ is: 

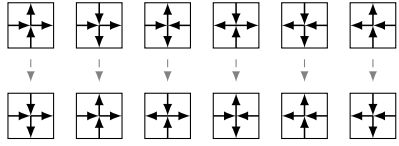
\[\begin{tikzpicture}[scale=0.6]
\begin{scope}
\draw (0,0) rectangle (1,1);
\draw[-latex,line width=0.2mm] (0.5,0) -- (0.5,0.5);
\draw[-latex,line width=0.2mm] (0.5,0.5) -- (0.5,1); 
\draw[-latex,line width=0.2mm] (0,0.5) -- (0.5,0.5);
\draw[-latex,line width=0.2mm] (0.5,0.5) -- (1,0.5);
\end{scope}

\begin{scope}[xshift=1.5cm]
\draw (0,0) rectangle (1,1);
\draw[latex-,line width=0.2mm] (0.5,0) -- (0.5,0.5);
\draw[latex-,line width=0.2mm] (0.5,0.5) -- (0.5,1); 
\draw[-latex,line width=0.2mm] (0,0.5) -- (0.5,0.5);
\draw[-latex,line width=0.2mm] (0.5,0.5) -- (1,0.5); 
\end{scope}

\begin{scope}[xshift=3cm]
\draw (0,0) rectangle (1,1);
\draw[latex-,line width=0.2mm] (0.5,0) -- (0.5,0.5);
\draw[-latex,line width=0.2mm] (0.5,0.5) -- (0.5,1); 
\draw[-latex,line width=0.2mm] (0,0.5) -- (0.5,0.5);
\draw[latex-,line width=0.2mm] (0.5,0.5) -- (1,0.5); 
\end{scope}

\begin{scope}[xshift=4.5cm]
\draw (0,0) rectangle (1,1);
\draw[-latex,line width=0.2mm] (0.5,0) -- (0.5,0.5);
\draw[latex-,line width=0.2mm] (0.5,0.5) -- (0.5,1); 
\draw[latex-,line width=0.2mm] (0,0.5) -- (0.5,0.5);
\draw[-latex,line width=0.2mm] (0.5,0.5) -- (1,0.5); 
\end{scope} 

\begin{scope}[xshift=6cm]
\draw (0,0) rectangle (1,1);
\draw[latex-,line width=0.2mm] (0.5,0) -- (0.5,0.5);
\draw[latex-,line width=0.2mm] (0.5,0.5) -- (0.5,1); 
\draw[latex-,line width=0.2mm] (0,0.5) -- (0.5,0.5);
\draw[latex-,line width=0.2mm] (0.5,0.5) -- (1,0.5); 
\end{scope}

\begin{scope}[xshift=7.5cm]
\draw (0,0) rectangle (1,1);
\draw[-latex,line width=0.2mm] (0.5,0) -- (0.5,0.5);
\draw[-latex,line width=0.2mm] (0.5,0.5) -- (0.5,1); 
\draw[latex-,line width=0.2mm] (0,0.5) -- (0.5,0.5);
\draw[latex-,line width=0.2mm] (0.5,0.5) -- (1,0.5); 
\end{scope}

\foreach \x in {0,1,2,3,4,5} {
\draw[dashed,-latex,gray!98] (\x*1.5+0.5,-0.25) -- (\x*1.5+0.5,-0.75);}

\begin{scope}[yshift=-0.5cm]

\begin{scope}[yshift=-1.5cm,xshift=7.5cm]
\draw (0,0) rectangle (1,1);
\draw[latex-,line width=0.2mm] (0.5,0) -- (0.5,0.5);
\draw[latex-,line width=0.2mm] (0.5,0.5) -- (0.5,1); 
\draw[latex-,line width=0.2mm] (0,0.5) -- (0.5,0.5);
\draw[latex-,line width=0.2mm] (0.5,0.5) -- (1,0.5);
\end{scope}

\begin{scope}[yshift=-1.5cm,xshift=6cm]
\draw (0,0) rectangle (1,1);
\draw[-latex,line width=0.2mm] (0.5,0) -- (0.5,0.5);
\draw[-latex,line width=0.2mm] (0.5,0.5) -- (0.5,1); 
\draw[latex-,line width=0.2mm] (0,0.5) -- (0.5,0.5);
\draw[latex-,line width=0.2mm] (0.5,0.5) -- (1,0.5); 
\end{scope}

\begin{scope}[yshift=-1.5cm,xshift=3cm]
\draw (0,0) rectangle (1,1);
\draw[-latex,line width=0.2mm] (0.5,0) -- (0.5,0.5);
\draw[latex-,line width=0.2mm] (0.5,0.5) -- (0.5,1); 
\draw[latex-,line width=0.2mm] (0,0.5) -- (0.5,0.5);
\draw[-latex,line width=0.2mm] (0.5,0.5) -- (1,0.5); 
\end{scope}

\begin{scope}[yshift=-1.5cm,xshift=4.5cm]
\draw (0,0) rectangle (1,1);
\draw[latex-,line width=0.2mm] (0.5,0) -- (0.5,0.5);
\draw[-latex,line width=0.2mm] (0.5,0.5) -- (0.5,1); 
\draw[-latex,line width=0.2mm] (0,0.5) -- (0.5,0.5);
\draw[latex-,line width=0.2mm] (0.5,0.5) -- (1,0.5);
\end{scope}

\begin{scope}[yshift=-1.5cm,xshift=1.5cm]
\draw (0,0) rectangle (1,1);
\draw[-latex,line width=0.2mm] (0.5,0) -- (0.5,0.5);
\draw[-latex,line width=0.2mm] (0.5,0.5) -- (0.5,1); 
\draw[-latex,line width=0.2mm] (0,0.5) -- (0.5,0.5);
\draw[-latex,line width=0.2mm] (0.5,0.5) -- (1,0.5);
\end{scope}

\begin{scope}[yshift=-1.5cm]
\draw (0,0) rectangle (1,1);
\draw[latex-,line width=0.2mm] (0.5,0) -- (0.5,0.5);
\draw[latex-,line width=0.2mm] (0.5,0.5) -- (0.5,1); 
\draw[-latex,line width=0.2mm] (0,0.5) -- (0.5,0.5);
\draw[-latex,line width=0.2mm] (0.5,0.5) -- (1,0.5);
\end{scope} \end{scope}
\end{tikzpicture}\]

We define then a horizontal symmetry operation 
$\mathcal{T}_{N}$ (see Figure~\ref{figure.transformation.symmetry}
for an illustration) on patterns whose 
support is some $\mathbb{U}^{(2)}_{M,N}$, with $M \ge 1$, such that for all $M \ge 1$, 
$p$ having support  $\mathbb{U}^{(2)}_{M,N}$, $\mathcal{T}_{M} (p)$ has also 
support $\mathbb{U}^{(2)}_{M,N}$ and for all
$(i,j) \in \mathbb{U}^{(2)}_{M,N}$, 
\[\mathcal{T}_{N}(p)_{i,j} = \tau (p_{N-i,j}).\]

We define also the applications
$\partial^{r}_{N}$ (resp. $\partial^{l}_{N}$,  $\partial^{t}_{N}$) that acts on patterns of 
the six vertex model
whose support is some 
$\mathbb{U}^{(2)}_{M,N}$, $M\ge 1$ 
and such that for all $M \ge 1$ 
and $p$ on support $\mathbb{U}^{(2)}_{M,N}$, 
$\partial^{r}_{N} (p)$ (resp. $\partial^{l}_{N} (p)$) is a length $M$ (resp. $M$, $N$)
word and for all $j$ between $1$ and $M$ (resp. $M$), 
$\partial^{r}_{N} (p)_j$ (resp. $\partial^{l}_{N} (p)_j$) is the 
east (resp. west) arrow in the symbol $p_{N,j}$
(resp. $p_{1,j}$). 
For instance, if $p$ is the pattern on the left 
on Figure~\ref{figure.transformation.symmetry},
then $\partial^{r}_{N} (p)$ (resp. $\partial^{l}_{N} (p)$) is 
the word: 
\[\leftarrow  \leftarrow \rightarrow \rightarrow \quad (\text{resp.} \rightarrow \rightarrow \rightarrow \rightarrow)\]

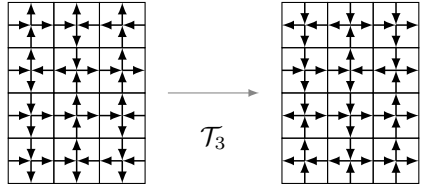
\begin{figure}[h!]

\[\begin{tikzpicture}[scale=0.6]

\begin{scope}

\begin{scope}
\draw (0,0) rectangle (1,1);
\draw[latex-,line width=0.2mm] (0.5,0) -- (0.5,0.5);
\draw[latex-,line width=0.2mm] (0.5,0.5) -- (0.5,1); 
\draw[-latex,line width=0.2mm] (0,0.5) -- (0.5,0.5);
\draw[-latex,line width=0.2mm] (0.5,0.5) -- (1,0.5); 
\end{scope} 

\begin{scope}[xshift=1cm]
\draw (0,0) rectangle (1,1);
\draw[latex-,line width=0.2mm] (0.5,0) -- (0.5,0.5);
\draw[-latex,line width=0.2mm] (0.5,0.5) -- (0.5,1); 
\draw[-latex,line width=0.2mm] (0,0.5) -- (0.5,0.5);
\draw[latex-,line width=0.2mm] (0.5,0.5) -- (1,0.5); 
\end{scope}

\begin{scope}[xshift=2cm]
\draw (0,0) rectangle (1,1);
\draw[latex-,line width=0.2mm] (0.5,0) -- (0.5,0.5);
\draw[latex-,line width=0.2mm] (0.5,0.5) -- (0.5,1); 
\draw[latex-,line width=0.2mm] (0,0.5) -- (0.5,0.5);
\draw[latex-,line width=0.2mm] (0.5,0.5) -- (1,0.5); 
\end{scope}

\begin{scope}[yshift=1cm]
\draw (0,0) rectangle (1,1);
\draw[latex-,line width=0.2mm] (0.5,0) -- (0.5,0.5);
\draw[latex-,line width=0.2mm] (0.5,0.5) -- (0.5,1); 
\draw[-latex,line width=0.2mm] (0,0.5) -- (0.5,0.5);
\draw[-latex,line width=0.2mm] (0.5,0.5) -- (1,0.5); 
\end{scope} 

\begin{scope}[xshift=1cm,yshift=1cm]
\draw (0,0) rectangle (1,1);
\draw[-latex,line width=0.2mm] (0.5,0) -- (0.5,0.5);
\draw[-latex,line width=0.2mm] (0.5,0.5) -- (0.5,1); 
\draw[-latex,line width=0.2mm] (0,0.5) -- (0.5,0.5);
\draw[-latex,line width=0.2mm] (0.5,0.5) -- (1,0.5); 
\end{scope}

\begin{scope}[xshift=2cm,yshift=1cm]
\draw (0,0) rectangle (1,1);
\draw[latex-,line width=0.2mm] (0.5,0) -- (0.5,0.5);
\draw[-latex,line width=0.2mm] (0.5,0.5) -- (0.5,1); 
\draw[-latex,line width=0.2mm] (0,0.5) -- (0.5,0.5);
\draw[latex-,line width=0.2mm] (0.5,0.5) -- (1,0.5); 
\end{scope}

\begin{scope}[yshift=2cm]
\draw (0,0) rectangle (1,1);
\draw[latex-,line width=0.2mm] (0.5,0) -- (0.5,0.5);
\draw[-latex,line width=0.2mm] (0.5,0.5) -- (0.5,1); 
\draw[-latex,line width=0.2mm] (0,0.5) -- (0.5,0.5);
\draw[latex-,line width=0.2mm] (0.5,0.5) -- (1,0.5); 
\end{scope}

\begin{scope}[xshift=1cm,yshift=2cm]
\draw (0,0) rectangle (1,1);
\draw[-latex,line width=0.2mm] (0.5,0) -- (0.5,0.5);
\draw[latex-,line width=0.2mm] (0.5,0.5) -- (0.5,1); 
\draw[latex-,line width=0.2mm] (0,0.5) -- (0.5,0.5);
\draw[-latex,line width=0.2mm] (0.5,0.5) -- (1,0.5); 
\end{scope}

\begin{scope}[xshift=2cm,yshift=2cm]
\draw (0,0) rectangle (1,1);
\draw[-latex,line width=0.2mm] (0.5,0) -- (0.5,0.5);
\draw[-latex,line width=0.2mm] (0.5,0.5) -- (0.5,1); 
\draw[-latex,line width=0.2mm] (0,0.5) -- (0.5,0.5);
\draw[-latex,line width=0.2mm] (0.5,0.5) -- (1,0.5); 
\end{scope}

\begin{scope}[xshift=0cm,yshift=3cm]
\draw (0,0) rectangle (1,1);
\draw[-latex,line width=0.2mm] (0.5,0) -- (0.5,0.5);
\draw[-latex,line width=0.2mm] (0.5,0.5) -- (0.5,1); 
\draw[-latex,line width=0.2mm] (0,0.5) -- (0.5,0.5);
\draw[-latex,line width=0.2mm] (0.5,0.5) -- (1,0.5); 
\end{scope}

\begin{scope}[xshift=1cm,yshift=3cm]
\draw (0,0) rectangle (1,1);
\draw[latex-,line width=0.2mm] (0.5,0) -- (0.5,0.5);
\draw[-latex,line width=0.2mm] (0.5,0.5) -- (0.5,1); 
\draw[-latex,line width=0.2mm] (0,0.5) -- (0.5,0.5);
\draw[latex-,line width=0.2mm] (0.5,0.5) -- (1,0.5); 
\end{scope}

\begin{scope}[xshift=2cm,yshift=3cm]
\draw (0,0) rectangle (1,1);
\draw[-latex,line width=0.2mm] (0.5,0) -- (0.5,0.5);
\draw[-latex,line width=0.2mm] (0.5,0.5) -- (0.5,1); 
\draw[latex-,line width=0.2mm] (0,0.5) -- (0.5,0.5);
\draw[latex-,line width=0.2mm] (0.5,0.5) -- (1,0.5); 
\end{scope}

\end{scope}

\draw[color=gray!95,-latex] (3.5,2) -- (5.5,2);

\begin{scope}[xshift=5cm]

\begin{scope}[xshift=3cm]
\draw (0,0) rectangle (1,1);
\draw[-latex,line width=0.2mm] (0.5,0) -- (0.5,0.5);
\draw[-latex,line width=0.2mm] (0.5,0.5) -- (0.5,1); 
\draw[-latex,line width=0.2mm] (0,0.5) -- (0.5,0.5);
\draw[-latex,line width=0.2mm] (0.5,0.5) -- (1,0.5); 
\end{scope} 

\begin{scope}[xshift=2cm]
\draw (0,0) rectangle (1,1);
\draw[-latex,line width=0.2mm] (0.5,0) -- (0.5,0.5);
\draw[latex-,line width=0.2mm] (0.5,0.5) -- (0.5,1); 
\draw[latex-,line width=0.2mm] (0,0.5) -- (0.5,0.5);
\draw[-latex,line width=0.2mm] (0.5,0.5) -- (1,0.5); 
\end{scope}

\begin{scope}[xshift=1cm]
\draw (0,0) rectangle (1,1);
\draw[-latex,line width=0.2mm] (0.5,0) -- (0.5,0.5);
\draw[-latex,line width=0.2mm] (0.5,0.5) -- (0.5,1); 
\draw[latex-,line width=0.2mm] (0,0.5) -- (0.5,0.5);
\draw[latex-,line width=0.2mm] (0.5,0.5) -- (1,0.5); 
\end{scope}

\begin{scope}[xshift=3cm,yshift=1cm]
\draw (0,0) rectangle (1,1);
\draw[-latex,line width=0.2mm] (0.5,0) -- (0.5,0.5);
\draw[-latex,line width=0.2mm] (0.5,0.5) -- (0.5,1); 
\draw[-latex,line width=0.2mm] (0,0.5) -- (0.5,0.5);
\draw[-latex,line width=0.2mm] (0.5,0.5) -- (1,0.5); 
\end{scope} 

\begin{scope}[xshift=2cm,yshift=1cm]
\draw (0,0) rectangle (1,1);
\draw[latex-,line width=0.2mm] (0.5,0) -- (0.5,0.5);
\draw[latex-,line width=0.2mm] (0.5,0.5) -- (0.5,1); 
\draw[-latex,line width=0.2mm] (0,0.5) -- (0.5,0.5);
\draw[-latex,line width=0.2mm] (0.5,0.5) -- (1,0.5); 
\end{scope}

\begin{scope}[xshift=1cm,yshift=1cm]
\draw (0,0) rectangle (1,1);
\draw[-latex,line width=0.2mm] (0.5,0) -- (0.5,0.5);
\draw[latex-,line width=0.2mm] (0.5,0.5) -- (0.5,1); 
\draw[latex-,line width=0.2mm] (0,0.5) -- (0.5,0.5);
\draw[-latex,line width=0.2mm] (0.5,0.5) -- (1,0.5); 
\end{scope}

\begin{scope}[xshift=3cm,yshift=2cm]
\draw (0,0) rectangle (1,1);
\draw[-latex,line width=0.2mm] (0.5,0) -- (0.5,0.5);
\draw[latex-,line width=0.2mm] (0.5,0.5) -- (0.5,1); 
\draw[latex-,line width=0.2mm] (0,0.5) -- (0.5,0.5);
\draw[-latex,line width=0.2mm] (0.5,0.5) -- (1,0.5); 
\end{scope}

\begin{scope}[xshift=2cm,yshift=2cm]
\draw (0,0) rectangle (1,1);
\draw[latex-,line width=0.2mm] (0.5,0) -- (0.5,0.5);
\draw[-latex,line width=0.2mm] (0.5,0.5) -- (0.5,1); 
\draw[-latex,line width=0.2mm] (0,0.5) -- (0.5,0.5);
\draw[latex-,line width=0.2mm] (0.5,0.5) -- (1,0.5); 
\end{scope}

\begin{scope}[xshift=1cm,yshift=2cm]
\draw (0,0) rectangle (1,1);
\draw[latex-,line width=0.2mm] (0.5,0) -- (0.5,0.5);
\draw[latex-,line width=0.2mm] (0.5,0.5) -- (0.5,1); 
\draw[-latex,line width=0.2mm] (0,0.5) -- (0.5,0.5);
\draw[-latex,line width=0.2mm] (0.5,0.5) -- (1,0.5); 
\end{scope}

\begin{scope}[xshift=3cm,yshift=3cm]
\draw (0,0) rectangle (1,1);
\draw[latex-,line width=0.2mm] (0.5,0) -- (0.5,0.5);
\draw[latex-,line width=0.2mm] (0.5,0.5) -- (0.5,1); 
\draw[-latex,line width=0.2mm] (0,0.5) -- (0.5,0.5);
\draw[-latex,line width=0.2mm] (0.5,0.5) -- (1,0.5); 
\end{scope}

\begin{scope}[xshift=2cm,yshift=3cm]
\draw (0,0) rectangle (1,1);
\draw[-latex,line width=0.2mm] (0.5,0) -- (0.5,0.5);
\draw[latex-,line width=0.2mm] (0.5,0.5) -- (0.5,1); 
\draw[latex-,line width=0.2mm] (0,0.5) -- (0.5,0.5);
\draw[-latex,line width=0.2mm] (0.5,0.5) -- (1,0.5); 
\end{scope}

\begin{scope}[xshift=1cm,yshift=3cm]
\draw (0,0) rectangle (1,1);
\draw[latex-,line width=0.2mm] (0.5,0) -- (0.5,0.5);
\draw[latex-,line width=0.2mm] (0.5,0.5) -- (0.5,1); 
\draw[latex-,line width=0.2mm] (0,0.5) -- (0.5,0.5);
\draw[latex-,line width=0.2mm] (0.5,0.5) -- (1,0.5); 
\end{scope}

\end{scope}

\node at (4.5,1) {$\mathcal{T}_{3}$};

\end{tikzpicture}\]
\caption{\label{figure.transformation.symmetry}
Illustration of the definition of 
$\mathcal{T}_{3}$: the pattern on 
left (on support $\mathbb{U}^{(2)}_{3,4}$)
is transformed into the pattern 
on the right via this transformation.}
\end{figure}

For the purpose of notation, we denote 
also $\pi_s$ the application that transforms 
patterns of the six vertex model into 
patterns of $X_s$ via the application of $\pi_s$ 
letter by letter.
Let us consider the transformation 
$ \mathcal{T}^s_{N} \equiv \pi_s \circ  \mathcal{T}_{N} \circ \pi_s ^{-1}$ 
on patterns of $X^s$ on some $\mathbb{U}^{(2)}_{M,N}$. We also denote 
$\partial_{N}^{l,s} 
\equiv  \partial_{N}^{l}  \circ \pi_s ^{-1}$, $\partial_{N}^{r,s} 
\equiv  \partial_{N}^{r}  \circ \pi_s ^{-1}$. Let us prove some properties of 
these transformations. For any word 
$\vec{w}$ on the alphabet $\{\leftarrow,\rightarrow\}$
(or $\{\uparrow,\downarrow\}$), we 
denote $\overline{\vec{w}}$ the word obtained by 
exchanging the two letters in the word $\vec{w}$.

\begin{enumerate}

\item \textbf{Preservation of global admissibility:}

For any $p$ globally admissible, 
$\mathcal{T}_{N} (p)$ is also locally admissible, 
and as a consequence globally admissible: 
indeed, it is sufficient to check that 
for all $u,v$ in the alphabet, 
if $uv$ is not a forbidden pattern in the 
six vertex model, then 
$\tau (v) \tau (u)$ is also not a forbidden pattern and that if $\begin{array}{c} u \\ v \end{array}$ 
is not forbidden, then $\begin{array}{c} \tau (u)  \\ \tau (v)  \end{array}$ is also not forbidden. 

The first assertion is verified because 
$u v$ is not forbidden if and only if 
the arrows of these symbols attached to their 
adjacent edge are pointing in the same direction, 
and this property is conserved when 
changing $uv$ into $\tau(v)\tau(u)$. The second 
one is verified for a similar reason. 

\item \textbf{Gluing patterns:}

Let us consider any $N,M \ge 1$ and 
$p,p'$ two patterns of $X^s$ on support 
$\mathbb{U}^{(2)}_{M,N}$, such that 
$\partial_{N}^{r,s} (p) = \partial_{N}^{r,s} (p') $ and $\partial_{N}^{l,s} (p) = \partial_{N}^{l,s} (p')$. Let us denote pattern $p''$ on support
$\mathbb{U}^{(2)}_{M,2N}$ such that the restriction 
of $p''$ on $\mathbb{U}^{(2)}_{M,N}$ is $p$ 
and the restriction on $(0,N) + 
\mathbb{U}^{(2)}_{M,N}$ is $\mathcal{T}_{N} (p')$.
\begin{itemize}
\item \textbf{This pattern is admissible} (locally and 
thus globally). Indeed, this is sufficient 
to check that gluing the two patterns 
$p$ and $p'$ does not make appear 
forbidden patterns, and this 
comes from that for all letter $u$, 
$ u \tau (u)$ is not forbidden. This can 
be checked directly, letter by letter.

\item \textbf{Moreover, $p''$ is in $\mathcal{N}_M (\overline{X}_{2N})$}. Indeed, this pattern can be \textit{wrapped 
on a cylinder}, and this comes from the fact 
that if $u$ is a symbol of the six 
vertex model, $\tau(u) u$ is not forbidden.
\end{itemize}
\end{enumerate}

\item \textbf{From the gluing property to an 
upper bound:} Given $\vec{w} = (\vec{w}^l,\vec{w}^r)$ some pair of words on $\{\rightarrow,\leftarrow\}$, we denote 
$\mathcal{N}^{\vec{w}}_{M,N}$ 
the number of patterns of $X^s$ on 
support $\mathbb{U}^{(2)}_{M,N}$ such that 
$\partial_{N}^{l,s}=\vec{w}^l$ and 
$\partial_{N}^{r,s}=\vec{w}^r$. Since $\mathcal{T}_{N}$ 
is a bijection, denoting $\overline{\vec{w}} = (\overline{\vec{w}^l},\overline{\vec{w}^r})$, we have 
\[\mathcal{N}_{M,N}^{\vec{w}} = \mathcal{N}_{M,N}^{\overline{\vec{w}}}.\] 

From last point, for all $\vec{w}$, 
\[\mathcal{N}_{M} (\overline{X}^s_{2N}) \ge \mathcal{N}_{M,N}^{\vec{w}} . \mathcal{N}_{M,N}^{\overline{\vec{w}}} =  \left(\mathcal{N}_{M,N}^{\vec{w}}\right)^2\]
\[\left(\mathcal{N}_{M} (\overline{X}^s_{2N})\right)^{\frac{1}{2}} 
\ge \mathcal{N}_{M,2N}^{\vec{w}}.\]
By summing over all possible $\vec{w}$: 

\[2^{2M} . \left(\mathcal{N}_{M}(\overline{X}^s_{2N})\right)^{\frac{1}{2}} = \sum_{\vec{w}} \left(\mathcal{N}_{M} (\overline{X}^s_{2N})\right)^{\frac{1}{2}} 
\ge \sum_{\vec{w}} \mathcal{N}_{M,N}^{\vec{w}}
= \mathcal{N}_{M,N} (X^s).\]

As a consequence for all $N$, 

\[2 + \frac{1}{2} h(\overline{X}^s_{2N}) \ge h(X^s_{2N}).\]

This implies that
\[\liminf_N \frac{h(\overline{X}^s_{2N})}{2N}
\ge \liminf_N \frac{h(X^s_{N})}{N} = h(X^s).\]

For similar reasons 

\[\liminf_N \frac{h(\overline{X}^s_{2N+1})}{2N+1}
\ge \liminf_N \frac{h(X^s_{N})}{N} = h(X^s),\]
and thus: 

\[\liminf_N \frac{h(\overline{X}^s_{N})}{N}
\ge \liminf_N \frac{h(X^s_{N})}{N} = h(X^s).\]

\end{enumerate}

\end{proof}

\section{\label{section.overview} Overview of the text}

In the following, we provide a complete 
proof of the following theorem: 

\begin{reptheorem}{theorem.main}
The entropy of square ice is equal 
to 
\[\boxed{h(X^s)= \frac{3}{2} \log_2 \left( \frac{4}{3} \right)}\]
\end{reptheorem} 

The proof of Theorem~\ref{theorem.main} can 
be overwieved as follows: 

\begin{itemize}
\item The strategy is primarily: 
\begin{enumerate}
\item  
to compute the entropies 
$h(\overline{X}^s_{N})$,
\item then 
to use Lemma~\ref{lemma.toroidal.ice} in order to compute $h(X^s)$. 
\end{enumerate}

\item 

The first point is derived from the \textbf{transfer matrix} method, which 
allows to express $h(\overline{X}^s_{N})$ with a formula 
involving a sequence of numbers defined implicity 
through a system of non-linear equations 
called \textbf{Bethe equations}.
This method 
is itself decomposed in several steps:

\begin{enumerate}
\item  \textbf{Formulation with 
transfer matrices [Section~\ref{section.lieb.path}]:} it is usual, when dealing 
with unidimensional subshifts of finite type, 
to express their entropy as the greatest 
eigenvalue of a matrix
which relates which couples of rows of symbols 
can be adjacent. 
In this text, we use the adjacent matrix $V_N^{*}$
of a factor subshift, thought as acting 
on $\Omega_{N} = \mathbb{C}^2 \otimes ... \otimes \mathbb{C}^2$. Lemma~\ref{lemma.toroidal.ice} tells that 
one can compute $h(\overline{X}^s_N)$ by computing 
the maximal eigenvalue of the 
adjacency matrix of the factor of $\overline{X}^s_{N}$. 

\item \textbf{Lieb path - transport of 
information through analycity [Section~\ref{section.lieb.path}]:}

In quantum physics, 
\textbf{transfer matrices}, which are 
complexifications of the adjacency matrix 
in a local way (in the sense that the coefficient 
relative to a couple of rows is a product 
of coefficients in $\mathbb{C}$ relative to the symbols in 
the two rows)
are used to derive properties of the system. 
In our study, 
the adjacency matrix is seen 
as a particular value of an analytic path of such transfer matrices,
$t \mapsto V_{N} (t)$ such that for all $t$, $V_{N} (t)$ is an irreducible non-negative and symmetric matrix, and such that $V_{N} (1) = V_{N}^{*}$.
Such a path is called (for the clarity of the exposition) 
a \textbf{Lieb path} 
in this text. In Section~\ref{section.lieb.path} 
we define the Lieb path
that will be used in the following. This part is a detailed exposition 
of notions defined in the article of E.H.Lieb~\cite{Lieb67}.

\item \textbf{Coordinate Bethe ansatz [Section~\ref{section.existence.identification}]:}

We use the coordinate Bethe ansatz (exposed 
in~\cite{Duminil-Copin.ansatz} and related in the present text), 
that provides candidate eigenvectors for the matrix $V_N(t)$ 
for all $t$ in each of spaces 
$\Omega^{(n)}_N$, $n \le N$ that form 
a decomposition of $\Omega_N$: 
\[\Omega_N = \bigoplus_{n=0}^N \Omega^{(n)}_N.\]
The candidate eigenvectors and 
eigenvalues depend each on 
a sequence $(p_j)_{j=1..n}$ that verifies 
a non-linear system of equations called Bethe equations.

It is shown that the system of Bethe equations on the parameters $p_j$
admits a unique solution for each $n$,$N$ and $t$, denoted $(\vec{p}_j (t))_j$ for all $t \in (0,\sqrt{2})$, in a context where $n,N$ are fixed, using 
convexity arguments on an auxiliary function.
The analycity of the two types of paths and the convexity of 
the auxiliary function 
ensures that $t \mapsto (\vec{p}_j (t))_j$ is analytic. This part completes the 
proof of an argument left uncomplete 
in~\cite{YY66}.
In order to identify the greatest eigenvalue, 
we use the fact that $V_N (\sqrt{2})$ 
commutes with some Hamiltonian $H_N$ that 
is completely diagonalised (following \cite{LSM61}). The vector is non zero and 
associated to the maximal eigenvalue of 
$H_N$ on $\Omega_N^{(n)}$. By Perron-Frobenius theorem, the vector has positive coordinates, 
and by the same theorem, the associated 
Bethe value is effectively an eigenvalue of the transfer matrix and it 
is equal to the greatest eigenvalue of 
the restriction to $\Omega_{N}^{(n)}$.
By continuity, this is true also 
for $t$ in a neighborhood of $\sqrt{2}$. By analycity, this identity is true 
for all $t \in (0,\sqrt{2})$. 
\end{enumerate}

\item The second point is derived in two steps:

\begin{enumerate}
\item \textbf{Asymptotic condensation of Bethe roots [Section~\ref{section.asymptotics}]:} 

The sequences $(\vec{p}_j (t))_j$ are transformed 
into sequences $(\boldsymbol{\alpha}_j (t))_j$ through 
an analytic bijection. The values of these 
sequences are called \textbf{Bethe roots}.

We first prove that the sequences of Bethe 
roots are condensed according to 
a density function $\rho_t$ over $\mathbb{R}$, relatively to any continuous decreasing and 
integrable function $f : (0,+\infty) \rightarrow 
(0,+\infty)$, 
which means that the Cesaro mean of 
the finite sequence $(f(\boldsymbol{\alpha}_j (t)))_j$ converges towards $\int \rho_t (x) f(x) dx$.
This part involves rigorous proofs, some simplifications and adaptations of arguments 
that appeared in~\cite{kozlowski}. 
The density $\rho_t$ is defined through Fredholm integral equation, 
corresponding the asymptotic version of the 
Bethe equations. This equation is solved through Fourier analysis, following a computation done  in~\cite{YY66II}. 

\item \textbf{Computation of integrals [Section~\ref{section.computation.entropy}]:}
The condensation property proved in the last point implies that the formula obtained for $\frac{1}{N} h(\overline{X}^{s}_{N})$ converges to an integral involving $\rho_1$. 
The formula obtained for $\rho_1$ 
allows the computation of this integral, through lace integrals techniques. This 
part is a detailed version of computations exposed in~\cite{Lieb67}.
\end{enumerate} 

\end{itemize}

\section{\label{section.lieb.path} A Lieb path for square ice}

In this section, we define 
the matrices $V_N^{*}$ [Section~\ref{section.interlacing.relation}] and 
define an example of Lieb path $t \mapsto V_N (t)$ for 
the discrete curves shift $X^s$ [Section~\ref{section.example.lieb.path}], 
and relate $h(X^s)$ to $V_N (1) = V^{*}_N$ [Section~\ref{section.weighted.entropy}].

\subsection{\label{section.interlacing.relation}The interlacing relation and the matrices $V_N^{*}$}

In the following, for a square matrix $M$, 
we will denote $M[u,v]$ its entry on $(u,v)$. Moreover, we denote $\{0,1\}^{*}_N$ the set of length 
$N$ words on $\{0,1\}$.

\begin{notation}
Consider $\vec{u},\vec{v}$ two words 
in $\{0,1\}^{*}_N$, and 
$w$ some $(N,1)$-cylindric pattern 
of the subshift $X$. We say that the pattern $w$ 
\textbf{connects} $\vec{u}$ to $\vec{v}$ (we denote 
this $\vec{u} \mathcal{R}[w] \vec{v}$), when 
for all 
$k \in \llbracket 1,N\rrbracket$,
$\vec{u}_k = 1$ (resp. $\vec{v}_k=1$) if and only if $w$ has an 
incoming (resp. outgoing) curve on the bottom (resp. top) of its $k$th symbol.
This notation 
is illustrated on Figure~\ref{figure.definition.relation}. 
\end{notation}

\begin{definition}
Let us denote $\mathcal{R} \subset 
\{0,1\}^N \times \{0,1\}^N$ the relation defined 
by $\vec{u} \mathcal{R} \vec{v}$ if and only 
if there exists a $(N,1)$-cylindric pattern $w$
of the discrete curves shift $X^s$ such that $\vec{u} \mathcal{R}[w] \vec{v}$. 
\end{definition}

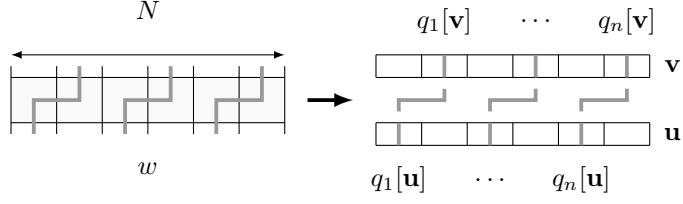
\begin{figure}[!h]
\[\begin{tikzpicture}[scale=0.6]

\begin{scope}

\fill[gray!05] (0,0) rectangle (6,1);

\draw[latex-latex] (0,1.5) -- (6,1.5);
\node at (3,2.5) {$N$};
\draw (0,0) -- (6,0); 
\draw (0,1) -- (6,1);
\foreach \x in {0,...,6}{
\draw (\x,-0.25) -- (\x,1.25);}
\draw[color=gray!80,line width=0.5mm] 
(2.5,-0.25) -- (2.5,0.5) -- (3.5,0.5) -- (3.5,1.25);
\draw[color=gray!80,line width=0.5mm] 
(0.5,-0.25) -- (0.5,0.5) -- (1.5,0.5) -- (1.5,1.25);
\draw[color=gray!80,line width=0.5mm] 
(4.5,-0.25) -- (4.5,0.5) -- (5.5,0.5) -- (5.5,1.25);

\draw[line width = 0.5mm,-latex] (6.5,0.5) -- (7.5,0.5);
\node at (3,-1) {$w$};
\end{scope}

\begin{scope}[xshift=8cm,yshift=-0.25cm]
\draw (0,0.25) -- (6,0.25);
\draw (0,-0.25) -- (6,-0.25);
\foreach \x in {0,...,6}{
\draw (\x,-0.25) -- (\x,0.25);}
\draw[line width=0.5mm,color=gray!80] 
(0.5,-0.25) -- (0.5,0.25);

\draw[line width=0.5mm,color=gray!80] 
(2.5,-0.25) -- (2.5,0.25);

\draw[line width=0.5mm,color=gray!80] 
(4.5,-0.25) -- (4.5,0.25);

\node at (6.5,0) {$\vec{u}$};
\node at (0.5,-1) {$q_1 [\vec{u}]$};
\node at (2.5,-1) {$\hdots$};
\node at (4.5,-1) {$q_n [\vec{u}]$};
\end{scope}

\begin{scope}[xshift=8cm,yshift=0.5cm]

\draw[line width=0.5mm,color=gray!80] 
(0.5,-0.25) -- (0.5,0) -- (1.5,0) -- (1.5,0.25);

\draw[line width=0.5mm,color=gray!80] 
(2.5,-0.25) -- (2.5,0) -- (3.5,0) -- (3.5,0.25);

\draw[line width=0.5mm,color=gray!80] 
(4.5,-0.25) -- (4.5,0) -- (5.5,0) -- (5.5,0.25);
\end{scope}

\begin{scope}[xshift=8cm,yshift=1.25cm]
\draw (0,0.25) -- (6,0.25);
\draw (0,-0.25) -- (6,-0.25);
\foreach \x in {0,...,6}{
\draw (\x,-0.25) -- (\x,0.25);}
\draw[line width=0.5mm,color=gray!80] 
(1.5,-0.25) -- (1.5,0.25);

\draw[line width=0.5mm,color=gray!80] 
(3.5,-0.25) -- (3.5,0.25);

\draw[line width=0.5mm,color=gray!80] 
(5.5,-0.25) -- (5.5,0.25);
\node at (6.5,0) {$\vec{v}$};
\node at (1.5,1) {$q_1 [\vec{v}]$};
\node at (3.5,1) {$\hdots$};
\node at (5.5,1) {$q_n [\vec{v}]$};
\end{scope}

\end{tikzpicture}\]
\caption{\label{figure.definition.relation} Illustration for the definition of the notation 
$\vec{u} \mathcal{R}[w] \vec{v}$.}
\end{figure}

\begin{notation}
For all $\vec{u} \in \{0,1\}_N^{*}$, 
we denote $|\vec{u}|_1$ the number of $k \in \llbracket 1, N\rrbracket$ 
such that $\vec{u}_k = 1$. If $|\vec{u}|_1 =n$, 
we denote $q_1 [\vec{u}]< ... < q_n [\vec{u}]$
the integers such that $\vec{u}_k = 1$ if 
and only if $k=q_i[\vec{u}]$ for some $i \in \llbracket 1,n\rrbracket$.
\end{notation}

Let us also notice that 
$\vec{u} \mathcal{R} \vec{v}$ 
implies that the number of $1$ symbols in 
$\vec{u}$ is equal to the number of $1$ symbols in $\vec{v}$.

\begin{definition}
We say that two words $\vec{u},\vec{v}$ in 
$\{0,1\}^*_N$ such that $|\vec{u}|_1=|\vec{v}|_1
\equiv n$
are \textbf{interlaced} when one of 
the two following conditions 
is satisfied: 
\[q_1 [\vec{u}] \le q_1 [\vec{v}] \le q_2 [\vec{u}] \le ... \le q_n [\vec{u}] \le q_n [\vec{v}]\] 
\[q_1 [\vec{v}] \le q_1 [\vec{u}] \le q_2 [\vec{v}] \le ... \le q_n [\vec{v}] \le q_n [\vec{u}].\]
\end{definition}

\begin{proposition} \label{proposition.interlacing}
For two length $N$ words $\vec{u},\vec{v}$, 
we have $\vec{u} \mathcal{R} \vec{v}$ if and only if $|\vec{u}|_1 = |\vec{v}|_1 \equiv n$ and $\vec{u},\vec{v}$ are interlaced. 
\end{proposition}

\begin{proof}
\begin{itemize}
\item $(\Rightarrow)$: assume that $\vec{u} \mathcal{R}[w] \vec{v}$ for some $w$.

First, since $w$ is a $(N,1)$-cylindric pattern, 
each of the curves that crosses its bottom side 
also crosses its top side, which implies 
that $|\vec{u}|_1 = |\vec{v}|_1$.

Let us assume that $\vec{u}$ and $\vec{v}$ 
are not interlaced. Without loss 
of generality, one can assume that 
$q_1 [\vec{u}] \le q_1 [\vec{v}]$.

\begin{enumerate}
\item \textbf{The position 
$q_1 [\vec{u}]$ is connected 
to $q_1 [\vec{v}]$:} 

Indeed, 
if it did not, another curve 
would connect another position $q_k [\vec{u}]$,
$k \neq 1$ 
of $\vec{u}$ to $q_1 [\vec{v}]$. 
Since $q_k [\vec{u}] > q_1 [\vec{u}]$ 
(by definition), this curve 
would cross the left border of 
$w$. It would imply that in the 
$q_1 [\vec{u}]$th symbol of $w$, two 
pieces of curves would appear: one horizontal, 
corresponding to the curve connecting the 
position $q_k [\vec{u}]$ to $q_1 [\vec{v}]$, 
and the one that connects
$q_1 [\vec{u}]$ to another position in $\vec{u}$, 
which is not possible, by the definition 
of the alphabet of $X^s$: this is illustrated on Figure~\ref{figure.crossing.curves}.

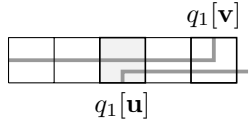
\begin{figure}
\[\begin{tikzpicture}[scale=0.6]
\fill[gray!10] (2,0) rectangle (3,1);
\draw[line width = 0.5mm,color=gray!80]
(0,0.5) -- (4.5,0.5) -- (4.5,1);
\draw[line width = 0.5mm,color=gray!80]
(2.5,0) -- (2.5,0.25) -- (5.25,0.25);
\draw (0,0) grid (5,1);
\draw[line width =0.2mm] (2,0) rectangle (3,1);
\draw[line width =0.2mm] (4,0) rectangle (5,1);
\node at (2.5,-0.5) {$q_1 [\vec{u}]$};
\node at (4.5,1.5) {$q_1 [\vec{v}]$};
\end{tikzpicture}\]
\caption{\label{figure.crossing.curves} Illustration of (impossible) 
crossing situation, which would imply 
non-authorized symbols.}
\end{figure}

\item \textbf{The position $q_2 [\vec{u}]$ 
is connected to $q_2 [\vec{v}]$:}

The curve crossing the position 
$q_1 [\vec{u}]$ at the bottom of $w$ can not
cross the position $q_2 [\vec{u}]+1$ of $w$ (since 
it would imply symbols that are not 
in the alphabet). Thus $q_1 [\vec{v}] \le q_2 [\vec{u}]$. Moreover, $q_2 [\vec{v}] \ge 
q_2 [\vec{u}]$, since if it was not the case, 
there would be a curve connecting some $q_k [\vec{u}] > q_2 [\vec{v}]$ to $q_2 [\vec{v}]$, 
thus crossing the left border of $w$, which 
would imply non-existant symbols in 
position $q_1 [\vec{u}]$ of $w$. 
Thus, for the same reason as in 
the first point 
argument, $w$ connects position $q_2 [\vec{u}]$ 
to $q_2 [\vec{v}]$.

\item \textbf{Repetition:}

We repeat the argument of the last point 
in order to obtain:
\[q_1 [\vec{u}] \le q_1 [\vec{v}] \le q_2 [\vec{u}] \le ... \le q_n [\vec{u}] \le q_n [\vec{v}],\]
meaning that $\vec{u}$ and $\vec{v}$ are 
interlaced.
\end{enumerate}

\item $(\Leftarrow)$: if $|\vec{u}|_1 
= \vec{v}_1$ and $\vec{u},\vec{v}$
are interlaced, then we define $w$ by connecting $q_i [\vec{u}]$ to $q_i 
 [\vec{v}]$ for all $i \in \llbracket 1 , n\rrbracket$. We thus have directly $\vec{u} \mathcal{R} [w] \vec{v}$.  
\end{itemize}
\end{proof}

\begin{proposition}
\label{proposition.unicity.connection}
When $\vec{u} \mathcal{R} \vec{v}$ and $\vec{u} \neq \vec{v}$, 
there exists a unique $w$ such that 
$\vec{u} \mathcal{R} [w] \vec{v}$. When $\vec{u} = \vec{v}$ there 
are exactly two possibilities, 
either the word $w$ that connects $q_i [\vec{u}]$ to itself for all $i$, or the one connecting $q_i [\vec{u}]$ to $q_{i+1} [\vec{u}]$ 
for all $i$. 
\end{proposition}

\begin{proof}
Consider two words $\vec{u} \neq \vec{v}$ such that $\vec{u} \mathcal{R} \vec{v}$. There exists 
at least one $i$ such that $q_i [\vec{u}]$ is different from any $q_j [\vec{v}]$. 
This forces any $w$ such that $\vec{u} \mathcal{R}[w] \vec{v}$ to 
connect the position $q_i [\vec{u}]$ to: 
\begin{itemize}
\item $q_{j_0} [\vec{v}]$ minimal amongst the 
positions $q_j [\vec{v}] \ge q_i [\vec{u}]$,
if the set of $j$ such that this is verified 
is not empty;
\item else the position $q_1 [\vec{v}]$.
\end{itemize}
If $j < n$, then since $\vec{u}$ and $\vec{v}$ are interlaced (Proposition~\ref{proposition.interlacing}), $q_{i+1}[\vec{u}] \ge q_{j_0} [\vec{v}]$. 
Thus it has to be connected to $q_{j_0+1} [\vec{v}]$ if $j_0 < n$, 
else to $q_1 [\vec{v}]$. 
If $j=n$, then for the same reason 
$q_1 [\vec{u}]$ has to be connected to $q_{j_0+1}
[\vec{v}]$ if $j_0 < n$ , else to $q_1 [\vec{v}]$.
Repeating this argument, we get 
the unicity of $w$.
\end{proof}

\subsection{\label{section.example.lieb.path} 
The Lieb path $t \mapsto V_N (t)$}

\begin{notation}
Let $N \ge 1$ be an integer, and $t>0$. Let 
us denote $\Omega_N$ the space $\mathbb{C}^2 \bigotimes ... \bigotimes \mathbb{C}^2$, 
tensor product of $N$ copies 
of $\mathbb{C}^2$, whose canonical basis elements 
are denoted indifferently by $\boldsymbol{\epsilon}= \ket{\boldsymbol{\epsilon}_1 ... \boldsymbol{\epsilon}_N}$ or the words 
$\boldsymbol{\epsilon}_1 ... \boldsymbol{\epsilon}_N$,
for $(\boldsymbol{\epsilon}_1, ..., \boldsymbol{\epsilon}_N) \in \{0,1\}^N$, 
according to quantum mechanics 
notations, in order to distinguish them from 
the coordinate definition of 
vectors of $\Omega_N$. 
\end{notation}

\begin{notation}
For all $N$ and $(N,1)$-cylindric
pattern $w$, let $|w|$ denote the number 
of symbols 
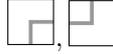
\[\begin{tikzpicture}[scale=0.6]
\begin{scope}[yshift=-1.5cm,xshift=3cm]
\draw[line width = 0.5mm,color=gray!80]
(0.5,0) -- (0.5,0.5) -- (1,0.5);
\draw (0,0) rectangle (1,1);
\end{scope}
\end{tikzpicture}, \begin{tikzpicture}[scale=0.6]
\begin{scope}[yshift=-1.5cm,xshift=3cm]
\draw[line width = 0.5mm,color=gray!80]
(0,0.5) -- (0.5,0.5) -- (0.5,1);
\draw (0,0) rectangle (1,1);
\end{scope}
\end{tikzpicture}\]
in this pattern. For instance, for the word $w$ on Figure~\ref{figure.definition.relation}, $|w|=6$. 
\end{notation}

\begin{definition}
For all $t \ge 0$, 
let us define $V_N (t) \in \mathcal{M}_{2^N} (\mathbb{C})$ 
the matrix such that for all $\boldsymbol{\epsilon},\boldsymbol{\eta} \in \{0,1\}^{*}_N$, 

\[V_N (t) [\boldsymbol{\epsilon},\boldsymbol{\eta}] =\displaystyle{\sum_{\boldsymbol{\epsilon} \mathcal{R}[w] \boldsymbol{\eta}}} t^{|w|}\]
\end{definition}

For all $N$ and $n \le N$, let 
us denote $\Omega_N^{(n)} \subset \Omega_N$ 
the vector space generated by the 
$\boldsymbol{\epsilon}=\ket{\boldsymbol{\epsilon}_1 ... \boldsymbol{\epsilon}_N}$ such that $|\boldsymbol{\epsilon}|_1 = n$.

\begin{proposition}
For all $N$ and $n \le N$, the matrix 
$V_N (t)$ stabilizes
the vector subspaces $\Omega_N^{(n)}$:
\[V_N (t) . \Omega_N^{(n)} 
\subset \Omega_N^{(n)}.\]
\end{proposition}

\begin{proof}
This is a direct consequence of Proposition~\ref{proposition.interlacing}, since if
$V_N (t) [\boldsymbol{\epsilon},
\boldsymbol{\eta}] \neq 0$ 
for $\boldsymbol{\epsilon},\boldsymbol{\eta}$ 
two elements of the canonical basis of $\Omega_N$, 
then $|\boldsymbol{\epsilon}|_1 = 
|\boldsymbol{\eta}|_1$.
\end{proof}

Let us recall that a non-negative matrix $A$ is called 
\textbf{irreducible} when there exists some $k \ge 1$ 
such that all the coefficients of $A^k$ are positive. 
Let us also recall the Perron-Frobenius theorem 
for symmetric, non-negative and irreducible matrices.

\begin{theorem}[Perron-Frobenius]
\label{theorem.perron}
Let $A$ be a symmetric, non-negative and irreducible 
matrix. Then $A$ has a positive eigenvalue $\lambda$ 
such that any other eigenvalue $\mu$ of $A$ satisfies 
$|\mu| \le \lambda$. Moreover, there exists 
some eigenvector $u$ for the 
eigenvalue $\lambda$ with positive coordinates such that 
if $v$ is another eigenvector (not necessarily for 
$\lambda$) with 
positive coordinates, then $v= \alpha.u$ for
some $\alpha > 0$. 
\end{theorem}

Let us prove the unicity of the positive 
eigenvector up to a multiplicative constant:

\begin{proof}
Let us denote $u \in \Omega_N$ the Perron-Frobenius 
eigenvector and $v \in \Omega_N$ another vector 
whose coordinates are all positive, associated 
to the eigenvalue $\mu$. Then 
\[\mu u^t.v = (Au)^t . v = u^{t} A v = \lambda u^t.v\]
Thus, since $u^t.v>0$, then $\mu = \lambda$, 
and by (usual version of) Perron-Frobenius, there exists some $\alpha \in \mathbb{R}$ 
such that $v= \alpha.u$. Since $v$ has positive 
coordinates, $\alpha>0$.
\end{proof}

\begin{lemma} \label{lemma.properties.matrix.transfer}
The matrix $V_N (t)$ is symmetric, non-negative and for all $n \le N$, its restriction 
to $\Omega_N^{(n)}$ is irreducible whenever $t >0$.
\end{lemma}

\begin{proof}
\begin{itemize}
\item \textbf{Symmetry:} since the 
interlacing relation is symmetric, for all 
$\boldsymbol{\epsilon},\boldsymbol{\eta} \in \{0,1\}^{*}_N$, we have that 
$V_N (t) (\boldsymbol{\epsilon},\boldsymbol{\eta}) >0$ if and only if $V_N (t) [\boldsymbol{\eta},\boldsymbol{\epsilon}]>0$. When 
this is the case, and $\boldsymbol{\epsilon} \neq \boldsymbol{\eta}$ (the case $\boldsymbol{\epsilon}=\boldsymbol{\eta}$ 
is trivial), there exists a unique (Proposition~\ref{proposition.unicity.connection}) $w$ connecting 
$\boldsymbol{\epsilon}$ to $\boldsymbol{\eta}$. The coefficient of this word 
is exactly $t^{2(n-|\{k: \boldsymbol{\epsilon}_k = \boldsymbol{\eta}_k = 1\}|)}$, 
where $n= |\{k:\boldsymbol{\epsilon}_k=1\}|=|\{k:\boldsymbol{\eta}_k=1\}|$,
and this coefficient is indifferent to 
the exchange of $\boldsymbol{\epsilon}$ and $\boldsymbol{\eta}$.  

\item \textbf{Irreducibility:} Let $\boldsymbol{\epsilon}$, 
$\boldsymbol{\eta}$ 
be two elements of the canonical 
basis of $\Omega_N$ such that $|\boldsymbol{\epsilon}|_1 
= |\boldsymbol{\eta}|_1 = n$.
We shall prove that $V_N ^N (t) 
[\boldsymbol{\epsilon},\boldsymbol{\eta}] >0$. 
\begin{enumerate}
\item \textbf{Interlacing case:}

 If they are interlacing, 
$V_N (t) [\boldsymbol{\epsilon},\boldsymbol{\eta}] >0$. As a consequence, since $V_N (t) [\boldsymbol{\epsilon}',\boldsymbol{\epsilon}']>0$ for all $\boldsymbol{\epsilon}'$, one keeps the positivity by 
repeating the action of $V_N (t)$. Thus 
$V_N (t)^N [\boldsymbol{\epsilon},\boldsymbol{\eta}]>0$. 
\item \textbf{Non-interlacing case:} 

\begin{itemize}
\item \textbf{Decreasing the interlacing 
degree:}

If they are not interlaced, let us 
denote $\omega(\boldsymbol{\epsilon},\boldsymbol{\eta})$ the maximal number 
of $q_j [\boldsymbol{\eta}]$ that lie 
in some $\llbracket q_i [\boldsymbol{\epsilon}],q_{i+1}[\boldsymbol{\epsilon}]\llbracket$.
This number is greater or equal to $2$. 
Let us see that there exists some $\boldsymbol{\epsilon}'$ 
such that $\boldsymbol{\epsilon} \mathcal{R} \boldsymbol{\eta}'$ and 
$\omega(\boldsymbol{\epsilon}',\boldsymbol{\eta}) < \omega(\boldsymbol{\epsilon},\boldsymbol{\eta})$.
Let us consider some $i$ such that 
$\llbracket q_i [\boldsymbol{\epsilon}],q_{i+1}[\boldsymbol{\epsilon}]\llbracket
\cap \{q_j[\boldsymbol{\eta}] : j \in \llbracket 1 , n \rrbracket \}$ is empty and 
$\llbracket q_{i+1} [\boldsymbol{\epsilon}],q_{i+2}[\boldsymbol{\epsilon}]\llbracket
\cap \{ q_j[\boldsymbol{\eta}] : j \in \llbracket 1 , n \rrbracket \}$ has more than one element (this case happens because $\boldsymbol{\epsilon},\boldsymbol{\eta}$ are not 
interlaced and $|\boldsymbol{\epsilon}|_1 = |\boldsymbol{\eta}|_1$), 
and consider the word $w$ that connects
the curve crossing position $q_{i+1} [\boldsymbol{\epsilon}]$ to the maximal 
$q_j [\boldsymbol{\eta}] \in \llbracket q_{i+1} [\boldsymbol{\epsilon}],q_{i+2} [\boldsymbol{\epsilon}]\llbracket$ and 
fixes the other positions. 
Let us call $\boldsymbol{\epsilon}'$ the vector such that 
$w$ connects $\boldsymbol{\epsilon}$ to $\boldsymbol{\epsilon}'$. 
We have indeed that $\omega(\boldsymbol{\epsilon}',\boldsymbol{\eta}) 
< \omega(\boldsymbol{\epsilon},\boldsymbol{\eta})$.
\item \textbf{A sequence with decreasing interlacing degree:}
As a consequence, since $\omega(\boldsymbol{\epsilon},\boldsymbol{\eta}) \le N$, one can construct a finite sequence 
of words $\boldsymbol{\epsilon}^{(k)}$, $k=1...m$ such that $m \le N$,  
$\boldsymbol{\epsilon}^{(1)} = \boldsymbol{\epsilon}$, $\boldsymbol{\epsilon}^{(m)}$ 
and $\boldsymbol{\eta}$ are interlaced, and for all $k <m$, 
$\boldsymbol{\epsilon}^{(k)} \mathcal{R} \boldsymbol{\epsilon}^{(k+1)}$.
This means that for all $k<m$, 
$V_N[\boldsymbol{\epsilon}^{(k)}, \boldsymbol{\epsilon}^{(k+1)}]>0$ 
and $V_N[\boldsymbol{\epsilon}^{(m)},\boldsymbol{\eta}]>0$.
As a consequence, $V_N (t)^N [\boldsymbol{\epsilon},\boldsymbol{\eta}] >0$.
\end{itemize}
\end{enumerate}
\end{itemize}
Since for all $\boldsymbol{\epsilon},\boldsymbol{\eta}$ with same number of curves, $V_N (t)^N [\boldsymbol{\epsilon},\boldsymbol{\eta}]>0$, 
this means that $V_N (t)$ is irreducible on 
$\Omega_N^{(n)}$ for all $n \le N$.
\end{proof}

\subsection{\label{section.weighted.entropy} Relation between $h(X^s)$ and the matrices $V_N (1)$}

\begin{notation}
For all $N$ and $n \le N$, 
let us denote $\overline{X}^s_{n,N}$ 
the subset (which 
is also a subshift) of $\overline{X}^s_{N}$ 
which consists in the set of configurations 
of  $\overline{X}^s_{N}$ such that 
the number of curves that cross each of 
its rows is $n$, and $\overline{X}_{n,N}$
the subset of $\overline{X}_{N}$ 
such that the number of arrows 
pointing south in the south part of 
the symbols in any raw is $n$. 
\end{notation}

\begin{notation}
Let us denote, for all $N$ and $n \le N$, 
$\lambda_{n,N} (t)$ the greatest 
eigenvalue of $V_N (t)$ on $\Omega_{N}^{(n)}$.
\end{notation}

\begin{proposition}
For all $N$ and $n \le N$: $h(\overline{X}^s_{n,N}) = \log_2 (\lambda_{n,N} (1))$.
\end{proposition}

\begin{proof}
\begin{itemize}

\item \textbf{Correspondance 
between $\overline{X}^s_{n,N}$ patterns 
and trajectories under action of $V_N(1)$:}

Since for all $N$, $n \le N$ 
and $\boldsymbol{\epsilon},\boldsymbol{\eta}$
in the canonical basis of $\Omega_N^{(n)}$, 
$V_{N} (1) [\boldsymbol{\epsilon},\boldsymbol{\eta}]$ 
is the number of ways to connect $\boldsymbol{\epsilon}$ to $\boldsymbol{\eta}$
by a $(N,1)$-cylindric pattern, 
and that there is a natural invertible map from 
the set of $(M,N)$-cylindric patterns to 
the sequences $(w_i)_{i=1...M}$ of $(N,1)$-cylindric patterns such that there exists some 
$\left(\boldsymbol{\epsilon}_i\right)_{i=1...M+1}$ 
such that for all $i$, $|\boldsymbol{\epsilon}_i|=n$ and for all $i \le M$, $\boldsymbol{\epsilon}_i \mathcal{R}[w_i]\boldsymbol{\epsilon}_{i+1}$,
\[||(V_{N}(1)_{\Omega_N^{(n)}}^M)||_1=\mathcal{N}_M (\overline{X}^s_{n,N}).\] 

\item \textbf{Gelfand's formula:}

It is known (Gelfand's formula) that: 
$||(V_{N} (1)_{\Omega_N^{(n)}}^M)||_1^{1/M} \rightarrow 
\lambda_{n,N} (1).$

As a consequence of the first point:
$h (\overline{X}^s_{n,N}) = \log_2 (\lambda_{n,N} (1)) .$

\end{itemize}

\end{proof}

\begin{proposition}
For all $N$: $h (X^{s}) = \lim_N \frac{1}{N} \max_{n \le N} h(\overline{X}^s_{n,N})$.
\end{proposition}

\begin{proof}

We have the decomposition
\[\overline{X}^s_{N} = \displaystyle{\bigcup_{n=0}^N} \overline{X}^s_{n,N}.\]
Moreover, these subshifts 
are disjoint. As a consequence: 
\[h(\overline{X}^s_{N}) = 
\max_{n \le N} h(\overline{X}^s_{n,N}).\]
From this we deduce the statement.
\end{proof}

As a consequence of Lemma~\ref{lemma.toroidal.ice},  

\begin{lemma}
\label{lemma.symmetry.eigenvalues}
For all $N \ge 1$ and $n \le N$, $h(\overline{X}^s_{n,N}) = h(\overline{X}^s_{N-n,N})$.
\end{lemma}

\begin{proof}

For the 
purpose of notation, we also denote
$\pi_s$ the application from $\overline{X}_{n,N}$ to $\overline{X}^s_{n,N}$ that consists in 
an application of $\pi_s$ letter by letter. 
This map is invertible.
Let us consider the application $\overline{\mathcal{T}}_{n,N}$ from 
$\overline{X}_{n,N}$ to $\overline{X}_{N-n,N}$ 
that inverts all the arrows. This map 
is an isomorphism, and thus 
the map $\pi_s \circ\overline{X}_{n,N} \circ \pi_s^{-1}$ is also an isomorphism 
from $\overline{X}^s_{n,N}$ 
to $\overline{X}^s_{N-n,N}$.
As a consequence, the two subshifts 
have the same entropy: 
\[h(\overline{X}^s_{n,N}) = h(\overline{X}^s_{N-n,N}).\]
\end{proof}

The following corollary is a straightforward 
consequence of Lemma~\ref{lemma.symmetry.eigenvalues}.

\begin{corollary}
The entropy of $X^s$ is given by 
the following formula: 
\[h(X^s) = \lim_N \frac{1}{N}\displaystyle{\max_{n \le N/2+1}} \log_2 (\lambda_{n,N} (1)).\]
\end{corollary}

\begin{lemma}
We deduce that:
\[\boxed{h(X^s) = \lim_N \frac{1}{N}\displaystyle{\max_{n \le (N-1)/4}} \log_2 (\lambda_{2n+1,N} (1))}.\]
\end{lemma}

\begin{proof}
Let us fix some integer $N$ and for all $n$ between $1$ and $N/2+1$, and 
consider the application that 
from the set of patterns of 
$\overline{X}^s_{n,N}$ on $\mathbb{U}^{(1)}_M$
associates a pattern of  
$\overline{X}^s_{n-1,N}$ on $\mathbb{U}^{(1)}_M$
by suppressing the curve that crosses 
the leftmost symbol in the bottom row of the 
pattern crossed by a curve [See an schema on 
Figure~\ref{figure.transformation.curve.suppression}]

\begin{figure}[h!]
\[\begin{tikzpicture}[scale=0.6]

\begin{scope}[xshift=7cm]
\draw[color=gray!80,line width=0.5mm] 
(0,1.5) -- (0.5,1.5) -- (0.5,4);

\draw[color=gray!80,line width=0.5mm] 
(3.5,0) -- (3.5,1.5) -- (4,1.5);

\draw (0,0) grid (4,4);

\end{scope}

\fill[gray!10] (1,0) rectangle (2,1);

\draw[-latex] (4.5,2) -- (6.5,2);

\draw[color=gray!80,line width=0.5mm] 
(0,1.5) -- (0.5,1.5) -- (0.5,4);

\draw[color=gray!80,line width=0.5mm] 
(3.5,0) -- (3.5,1.5) -- (4,1.5);

\draw[color=gray!80,line width=0.5mm] 
(1.5,0) -- (1.5,4) ;

\draw (0,0) grid (4,4);

\end{tikzpicture}\]
\caption{\label{figure.transformation.curve.suppression} Illustration of the curve suppressing operation; 
the leftmost position of the bottom raw crossed by 
a curve is colored gray on the left pattern.}
\end{figure}

For each pattern of $\overline{X}^s_{n-1,N}$, 
the number of patterns in its pre-image 
by this transformation is bounded from 
above by $N^M$. As a consequence, for all $M$:

\[\mathcal{N}_M (\overline{X}^s_{n-1,N}) . N^M \ge \mathcal{N}_M (\overline{X}^s_{n,N}),\]
and thus 
\[h(\overline{X}^s_{n-1,N}) + \log_2 (N) 
\ge h(\overline{X}^s_{n,N})\]

As a consequence: 

\begin{align*}
h(X^s) & = \lim_N \frac{1}{N} \max \left( \max_{2n+1 \le N/2+1}
h(\overline{X}^s_{2n+1,N}), \max_{2n \le N/2+1}
h(\overline{X}^s_{2n,N}) \right)\\
& \le \lim_N \frac{1}{N} \max \left( \max_{2n+1 \le N/2+1}
h(\overline{X}^s_{2n+1,N}), \max_{2n-1 \le N/2+1}
h(\overline{X}^s_{2n-1,N}) + \log_2 (N) \right)\\
& \le \lim_N \frac{1}{N} \max_{2n+1 \le N/2+1}
h(\overline{X}^s_{2n+1,N})\\
& = \lim_N \frac{1}{N} \max_{n \le N/4}
h(\overline{X}^s_{2n+1,N})
\end{align*}

Moreover, 
\[h(X^s) \ge \lim_N \frac{1}{N} \max_{2n+1 \le N/2+1}
h(\overline{X}^s_{2n+1,N}),\]
thus we have the following equality: 

\[h(X^s) = \lim_N \frac{1}{N}\displaystyle{\max_{n \le N/4}} \log_2 (\lambda_{2n+1,N} (1)).\]

With a similar argument, we get: 

\[h(X^s) = \lim_N \frac{1}{N}\displaystyle{\max_{n \le (N-1)/4}} \log_2 (\lambda_{2n+1,N} (1)).\]

\end{proof}

\section{\label{section.existence.identification} Coordinate Bethe ansatz}

In this section, we recall 
the statement of the coordinate Bethe ansatz 
[Section~\ref{section.statement.ansatz}], 
after defining some auxiliary functions [Section~\ref{section.auxiliary.functions}]. We then prove the existence 
and analycity of solutions 
of the system of equations $(E_j) [t,n,N]$, 
$j \le n$. [Section~\ref{section.existence}]. 
Then, following~\cite{LSM61}, we diagonalise 
a Hamiltonian related to the transfer matrix 
$V_N (\sqrt{2})$ [Section~\ref{section.hamiltonian}]. 
In the end, we use this analysis in order 
to  identify the largest eigenvalue of
 the restriction of $V_N (t)$ to 
 $\Omega_{N}^{(n)}$ for $t \in (0,\sqrt{2})$
 and $n \le N/2+1$ [Section~\ref{subsection.identification}]. 

\subsection{\label{section.auxiliary.functions} Auxiliary functions}

\subsubsection{\label{section.notations} Notations} 

Let us denote $\mu : (-1,1) \rightarrow (0,\pi)$ the inverse of the function $\cos : 
(0,\pi) \rightarrow (-1,1)$.
For all $t \in (0,\sqrt{2})$, we will denote 
$\Delta_t = \frac{2-t^2}{2}$, $\mu_t 
= \mu(-\Delta_t)$, and $I_t = (-(\pi-\mu_t), 
(\pi-\mu_t))$. 

\begin{notation}
Let us denote 
$\Theta$ the unique analytic function $(t,x,y) 
\mapsto \Theta_t (x,y)$ 
from the set $\{(t,x,y): x,y \in I_t\}$ 
to $\mathbb{R}$ such that $\Theta_{\sqrt{2}} 
(0,0) = 0$ and for all $t,x,y$,

\[\boxed{\exp(-i\Theta_t (x,y)) = \exp(i(x-y)). \frac{e^{-ix} + e^{iy} -  2 \Delta_t}{e^{-iy} + e^{ix} -  2 \Delta_t}}.\]
\end{notation} 

\noindent By a unicity argument, one can see 
that for all $t,x,y$, $\Theta_t(x,y) = - \Theta_t(y,x).$ As a consequence, for all $x$, 
$\Theta_t (x,x) = 0$.
For the same reason, $\Theta_t(x,-y)= -\Theta_t(-x,y)$ and $\Theta_t (-x,-y) = 
- \Theta (x,y)$. Moreover, $\Theta_t$ and all 
its derivatives can be extended by 
continuity on $I_t ^2 \backslash \{(x,x): x \in \partial I_t\}$. For the purpose of 
notation, we will denote also $\Theta_t$ 
the extended function. We will use 
the following: 

\begin{computation}
\label{computation.theta.border}
For all $y \neq (\pi-\mu_t)$, 
$\Theta_t ((\pi-\mu_t),y) = 2\mu_t - \pi$. 
\end{computation}

\begin{proof}
From the definition of $\Theta_t$: 
\[\exp(-i\Theta_t ((\pi-\mu_t),y)) = e^{i(\pi-\mu_t-y)} . \frac{e^{iy} - e^{i\mu_t} -  2 \Delta_t}{e^{-iy} - e^{-i\mu_t} -  2 \Delta_t} 
= e^{i(\pi-\mu_t-y)} . \frac{e^{iy} + e^{-i\mu_t}}{e^{-iy} + e^{i\mu_t}} \]
As a consequence,
\[\exp(-i\Theta_t ((\pi-\mu_t),y)) = e^{i(\pi-\mu_t-y)}. \frac{e^{iy}}{e^{i\mu_t}} 
\frac{1+ e^{-i(\mu_t + y)}}{e^{-i(y+\mu_t)}+1} 
= e^{i(\pi-2\mu_t)}.\]
This yields the statement as a consequence. 
\end{proof}

\begin{notation}
Let us denote $\kappa$ the 
unique analytic map $(t,\alpha) \mapsto \kappa_t (\alpha)$ 
from $(0,\sqrt{2}) \times \mathbb{R}$ to $\mathbb{R}$
such that $\kappa_{\sqrt{2}/2} (0)=0$ 
and for all $t,\alpha$,
\[\boxed{e^{i\kappa_t (\alpha)} = \frac{e^{i\mu_t}-e^{\alpha}}{e^{i\mu_t+\alpha} - 1}}.\]
\end{notation}

\noindent With the argument of unicity, 
we have that for all $t,\alpha$, 
$\kappa_t (-\alpha) 
= -\kappa_t (\alpha)$, and as a consequence, 
$\kappa_t (0) = 0$. 
We also denote, for all $t,\alpha,\beta$, 
\[\theta_t (\alpha,\beta) = \Theta_t (\kappa_t (\alpha),\kappa_t (\beta)).\]

\subsubsection{\label{section.properties.auxiliary.functions} Properties of the auxiliary 
functions}

\paragraph{ \label{paragraph.computation} Computation of 
the derivative $\kappa'_t$:} \bigskip

\begin{computation} \label{computation.kprime}
Let ux fix some $t \in (0,\sqrt{2})$. 
For all $\alpha \in \mathbb{R}$, 
\[\boxed{\kappa'_t(\alpha) = \frac{\sin(\mu_t)}{\cosh(\alpha)-\cos(\mu_t)}.}\]
\end{computation}

\begin{proof}

\begin{itemize}
\item \textbf{Computation of $\cos(\kappa_t(\alpha))$ and $sin(\kappa_t(\alpha))$:}

\[e^{i\kappa_t (\alpha)} = 
 \frac{\left(e^{-i\mu_t +\alpha}-1\right)\left(e^{i\mu_t}-e^{\alpha}\right)}{\left|e^{i\mu_t+\alpha}-1\right|^2} = \frac{e^{\alpha}
+ e^{2\alpha} e^{-i\mu_t} - e^{i\mu_t} + e^{\alpha}}{(\cos(\mu_t)e^{\alpha}-1)^2+(\sin(\mu_t)e^{\alpha})^2}.\]

Thus by taking the real part,
\[\cos(\kappa_t(\alpha)) = \frac{2e^{\alpha} 
+ (e^{2\alpha}-1)\cos(\mu_t)}{\cos^2(\mu_t)e^{2\alpha} - 2\cos(\mu_t) e^{\alpha} + 1 + (1- \cos^2 (\mu_t))e^{2\alpha}}
\]

\[\cos(\kappa_t(\alpha)) = \frac{2e^{\alpha} 
+ (e^{2\alpha}-1)\cos(\mu_t)}{e^{2\alpha} - 2\cos(\mu_t) e^{\alpha} + 1}
= \frac{1
- \cos(\mu_t)\cosh(\alpha)}{\cosh(\alpha)-\cos(\mu_t)},\]
where we factorized by $2e^{\alpha}$ for
the second equality. As a consequence:

\[\cos(\kappa_t(\alpha)) =\frac{\sin^2(\mu_t)+\cos^2(\mu_t)-\cos(\mu_t) \cosh(\alpha)}{\cosh(\alpha)-\cos(\mu_t)} = \frac{\sin^2(\mu_t)}{\cosh(\alpha)-\cos(\mu_t)} 
- \cos(\mu_t).\]

A similar computation gives 

\[\sin(\kappa_t(\alpha)) = \frac{\sin(\mu_t) \sinh(\alpha)}{\cosh(\alpha)-\cos(\mu_t)}\]

\item \textbf{Deriving the expression $\cos(\kappa_t(\alpha))$:}

As a consequence, for all $\alpha$: 

\[-\kappa'_t(\alpha) \sin(\kappa_t(\alpha)) = - \frac{\sin^2 (\mu_t) \sinh(\alpha)}{(\cosh(\alpha)-\cos(\mu_t))^2} = -\frac{\sin(\kappa_t(\alpha))^2}{\sinh(\alpha)}.\]
\[\]

Thus, for all $\alpha$ but in 
a discrete subset of 
$\mathbb{R}$, 
\[\kappa'_t(\alpha) = \frac{\sin(\mu_t)}{\cosh(\alpha)-\cos(\mu_t)}.\]
This identity is thus verified 
on all $\mathbb{R}$, by continuity.
\end{itemize}

\end{proof}

\paragraph{Domain and invertibility} 

\begin{proposition}
\label{proposition.k}
For all $t$, $\kappa_t (\mathbb{R}) 
\subset I_t$. Moreover, 
$\kappa_t$ considered as a function 
from $\mathbb{R}$ to $I_t$ is bijective.
\end{proposition}

\begin{proof}
\begin{itemize}

\item \textbf{Injectivity:}

Since $\mu_t \in (0,\pi)$, then
$\sin(\mu_t) > 0$ and we have the inequality
$\cosh(\alpha) \ge 1 > \cos (\mu_t)$.
As a consequence, $\kappa_t$ is strictly 
increasing, and thus injective. 

\item \textbf{The equality $\kappa_t(\alpha)=n\pi$
imples $\alpha=0$:}

Assume that for some $\alpha$, 
$\kappa_t(\alpha) = n \pi$ for some integer $n$. 
If $n$ is odd, then: 
\[e^{\alpha} - e^{i\mu_t} = e^{i\mu_t+\alpha} -1.\]
\[e^{\alpha}+1 = e^{i\mu_t}. (e^{\alpha}+1),\]
and thus $e^{i\mu_t}=0$, which is impossible, 
since $\mu_t \in (0,\pi)$.
If $n$ is even, 
then \[-e^{\alpha} + e^{i\mu_t} = e^{i\mu_t+\alpha} -1.\]
As a consequence, since $e^{i\mu_t} \neq -1$,
we have $e^{\alpha}=1$, and thus $\alpha=0$.

\item \textbf{Extension of the images:}

Since when $\alpha$ tends towards 
$+\infty$ (resp. $-\infty$), the function 
tends towards $-e^{i\mu}$ (resp. 
$e^{i\mu}$), $\kappa_t(\alpha)$ tends 
towards some $n\pi-\mu_t$ (resp. $m\pi+\mu_t$). 
and from the 
above property, $n=1$ (reps. $m=-1$). 
Thus the image of $\kappa_t$ is 
the set  $I_t$. 

Thus $\kappa_t$ is an invertible map 
from $\mathbb{R}$ to $I_t$. 
\end{itemize}
\end{proof}

\paragraph{A relation between $\theta_t$ 
and $\kappa_t$:}

The following equality originates in~\cite{YY66}. We provide some details 
of a relatively simple way to compute it.

\begin{computation} \label{computation.derivative.theta}
For any numbers $t,\alpha,\beta$:
\[\boxed{ \frac{\partial \theta_t}{\partial \alpha}  (\alpha,\beta) = 
- \frac{\partial \theta_t}{\partial \beta}  (\alpha,\beta) = 
- \frac{\sin(2\mu_t)}{\cosh(\alpha-\beta) - \cos(2\mu_t)}}\]
\end{computation}

\begin{proof}

\begin{itemize}
\item \textbf{Deriving the equation 
that defines $\Theta_t$:}

Let us denote, 
for all $x,y$: 
\[G_t (x,y) = \frac{x(1-2\Delta_t y)+y}{x+y-2\Delta_t}.\]
Then we have that for all $x,y$
\[\frac{\partial G_t}{\partial x}  (x,y) = \frac{(1-2\Delta_t y).(x+y-2\Delta_t) - (x(1-2\Delta_t y) + y)}{(x+y-2\Delta_t)^2}\]

\[\frac{\partial G_t}{\partial x} (x,y) = -2\Delta_t \frac{1 +y^2 - 2\Delta_t y}{(x+y-2\Delta_t)^2}\]

For all $t,\alpha$, let use 
denote $\alpha_t \equiv \kappa_t (\alpha)$.
By definition of $\Theta_t$, 
for all $\alpha,\beta$,
\[\exp(-i \Theta_t (\alpha_t,\beta_t)) 
= G(e^{i\alpha_t},e^{-i\beta_t}).\]
Thus we have, by deriving this equality:
\[-i\frac{d}{d\alpha} (\Theta_t (\alpha_t,\beta_t))
\exp(-i \Theta_t (\alpha_t,\beta_t))
= i\kappa'_t(\alpha) e^{i\alpha_t} \frac{\partial}{\partial x} G_t (e^{i\alpha_t},e^{-i\beta_t}).\]

\[\frac{d}{d\alpha} (\Theta_t (\alpha_t,\beta_t))
= - \kappa'_t(\alpha)e^{i\alpha_t} \frac{\frac{\partial}{\partial x} G_t (e^{i\alpha_t},e^{-i\beta_t})}{G_t(e^{i\alpha_t},e^{-i\beta_t})}\]

\[\frac{d}{d\alpha} \Theta_t (\alpha_t,\beta_t)
= \frac{(\kappa'_t (\alpha) 2\Delta e^{i\alpha_t})(1+e^{-2i\beta_t}-2\Delta e^{-i\beta_t})}{(e^{i\alpha_t}+e^{-i\beta_t}-2\Delta_t)(e^{i\alpha_t}+e^{-i\beta_t}-2\Delta_t .e^{i(\alpha_t-\beta_t)})}\]

Factoring by $e^{i(\alpha_t-\beta_t)}$:

\[\frac{d}{d\alpha} \Theta_t (\alpha_t,\beta_t)
 = 2\Delta_t \kappa'_t(\alpha). \frac{e^{i\beta_t} + 
e^{-i\beta_t} - 2\Delta_t}{(e^{i\alpha_t}+e^{-i\beta_t}-2\Delta_t)(e^{i\beta_t}+e^{-i\alpha_t}-2\Delta_t)}.\]

\item \textbf{Simplification 
of a term $e^{i\alpha_t}+ 
e^{-i\beta_t} - 2\Delta_t$:}

Let us denote the function $F$ 
defined by 
\[F_t(\alpha,\beta) = e^{i\alpha_t}+ 
e^{-i\beta_t} - 2\Delta_t.\]
By definition of $\kappa_t$ and 
$-2\Delta_t=e^{-i\mu_t}+e^{i\mu_t}$
 we have: 

\[F_t(\alpha,\beta) = \frac{e^{i\mu_t}
-e^{\alpha}}{e^{i\mu_t+\alpha}
-1} + \frac{ e^{i\mu_t+\beta}
-1}{e^{i\mu_t}
-e^{\beta}} + e^{i\mu_t} + e^{-i\mu_t}.\]

\[F_t(\alpha,\beta) = \frac{(e^{i\mu_t}
-e^{\alpha})(e^{i\mu_t}
-e^{\beta}) + (e^{i\mu_t+\alpha}
-1)(e^{i\mu_t+\beta}
-1)+ (e^{i\mu_t}+e^{-i\mu_t}).(e^{i\mu_t+\alpha}
-1)(e^{i\mu_t}
-e^{\beta})}{(e^{i\mu_t+\alpha}
-1)(e^{i\mu_t}
-e^{\beta})}.\]

\[F_t(\alpha,\beta) = \frac{e^{3i\mu_t+\alpha} 
+e^{\beta-i\mu_t} - e^{i\mu_t}.(e^{\alpha}+e^{\beta})}{(e^{i\mu_t+\alpha}
-1)(e^{i\mu_t}
-e^{\beta})}.\]

\item \textbf{Simplification of 
$\Theta_t$'s derivative:}

For all $\alpha,\beta$, we have 
\[\frac{1}{\kappa'_t(\alpha)} \frac{d}{d\alpha} \Theta_t(\alpha,\beta) 
= 2\Delta \frac{F_t(\beta,\beta)}{
F_t(\alpha,\beta).F_t(\beta,\alpha)}.\]
As a consequence of last point, 

\[\frac{1}{\kappa'_t(\alpha)} \frac{d}{d\alpha} \Theta_t( \alpha_t,\beta_t) 
= 
\frac{-(e^{i\mu_t+\alpha} -1)(e^{i\mu_t}-e^{\alpha})(e^{-i\mu_t}+e^{i\mu_t})(e^{3i\mu_t+\beta} +e^{\beta-i\mu_t} - 2e^{i\mu_t}.e^{\beta})}{(e^{3i\mu_t+\alpha} 
+e^{\beta-i\mu_t} - e^{i\mu_t}.(e^{\alpha}+e^{\beta}))(e^{3i\mu_t+\beta} 
+e^{\alpha-i\mu_t} - e^{i\mu_t}.(e^{\beta}+e^{\alpha}))}.\]

\[\frac{1}{\kappa'_t(\alpha)}  \frac{d}{d\alpha} \Theta_t(\alpha_t,\beta_t) 
= -
\frac{(e^{i\mu_t+\alpha} -1)(e^{i\mu_t}-e^{\alpha})e^{\beta} .(e^{2i\mu_t}-1).(e^{2i\mu_t}-e^{-2i\mu_t})}{e^{2i\mu_t}(e^{2i\mu_t+\alpha} 
+e^{\beta-2i\mu_t} - (e^{\alpha}+e^{\beta}))(e^{2i\mu_t+\beta} 
+e^{\alpha-2i\mu_t} - (e^{\beta}+e^{\alpha}))}.\]
Since in the denominator of the 
fraction in square of the modulus of 
some number, we rewrite it. 

\[\frac{1}{\kappa'_t(\alpha)} \frac{d}{d\alpha} \Theta_t(\alpha_t,\beta_t) 
= -
\frac{(e^{i\mu_t+\alpha} -1)(e^{i\mu_t}-e^{\alpha}) e^{\beta} .(e^{2i\mu_t}-1).(e^{2i\mu_t}-e^{-2i\mu_t})}{e^{2i\mu_t} \left((e^{\alpha}+e^{\beta})^2 (\cos(2\mu_t)-1)^2 + (e^{\alpha}-e^{\beta})^2 
\sin^2 (2\mu_t)\right)}.\]
We rewrite also the other terms, by splitting the $e^{2i\mu_t}$ in the 
denominator in two parts, one 
makes appear $\sin(\mu_t)$, 
and the other one, the square modulus:

\[\frac{1}{\kappa'_t(\alpha)} \frac{d}{d\alpha} \Theta_t (\alpha_t,\beta_t) 
= -4|e^{i\mu_t+\alpha} -1|^2 
\frac{e^{\beta} .\sin(\mu_t).\sin(2\mu_t)}{(e^{\alpha}+e^{\beta})^2 (\cos(2\mu_t)-1)^2 + (e^{\alpha}-e^{\beta})^2 
\sin^2 (2\mu_t)}.\]

By writing $\sin^2(2\mu_t)=1-\cos^2(2\mu_t)$ 
and then factoring 
by $1-\cos(2\mu_t)$:

\[\frac{1}{\kappa'_t(\alpha)} \frac{d}{d\alpha} \Theta_t(\alpha_t,\beta_t) 
= -4\frac{|e^{i\mu_t+\alpha} -1|^2}{1-\cos(2\mu_t)}
\frac{e^{\beta} .\sin(\mu_t).\sin(2\mu_t)}{(e^{\alpha}+e^{\beta})^2 (1-\cos(2\mu_t)) + (e^{\alpha}-e^{\beta})^2 
(1+\cos(2\mu_t))}.\]

Developping the denominator 
and factoring it by $4e^{\alpha+\beta}$, 
we obtain:

\[\frac{1}{\kappa'_t(\alpha)} \frac{d}{d\alpha} \Theta_t(\alpha_t,\beta_t) 
= -\frac{|e^{i\mu_t+\alpha} -1|^2}{e^{\alpha}(1-\cos(2\mu_t))}.
\frac{\sin(\mu_t).\sin(2\mu_t)}{\cosh(\alpha-\beta)- \cos(2\mu_t)}.\]

We have left to see that 
\[\frac{\sin(\mu_t) \kappa'_t(\alpha).|e^{i\mu_t+\alpha} -1|^2}{e^{\alpha}(1-\cos(2\mu_t))} = 1.\]
This derives directly from 
$1-\cos(2\mu_t) = 2\sin^2(\mu_t)$ and 
the value of $\kappa'_t(\alpha)$ given by Computation~\ref{computation.kprime}.

\item \textbf{The other equality:}

We obtain the value of 
$\frac{d}{d\beta} (\Theta_t(\alpha_t,\beta_t))$ through 
the equality $\Theta(x,y) = -\Theta(y,x)$ 
for all $x,y$.
\end{itemize}

\end{proof}

\begin{lemma}
\label{lemma.theta.translative}
For all $t,\alpha,\beta$,
\[\theta_t (\alpha+\beta,\alpha) = \theta_t (\beta,0).\]
\end{lemma}

\begin{proof}
Let us fix some $\alpha \in \mathbb{R}$, 
By Computation~\ref{computation.derivative.theta}, 
the derivative of the function 
$\beta \mapsto \theta_t (\alpha+\beta,\alpha)$ 
is equal to the derivative of the 
function $\beta \mapsto \theta_t (\beta,0)$. 
As a consequence, these two functions 
differ by a constant. Since they have 
the same value in $\beta=0$, they are equal.
\end{proof}

 \subsection{\label{section.statement.ansatz} Statement of the ansatz}

 \begin{notation}
 \label{notation.vector.ansatz}
For all $(p_1,...,p_n) \in I_t^n$, let us denote $\psi_{\mu_t,n,N} (p_1,...,p_n)$ the 
 vector in $\Omega_N$ such that for all 
 $\boldsymbol{\epsilon} \in \{0,1\}^*_N$, 
 \[\psi_{\mu_t,n,N}(p_1,...,p_n)[\boldsymbol{\epsilon}]
 = \sum_{\sigma \in \Sigma_n} C_{\sigma} (t) [p_1,...,p_n] \prod_{k=1}^n e^{ip_{\sigma(k)} q_k [\boldsymbol{\epsilon}]},\]
 where (for $\epsilon(\sigma)$ denotion the 
 signature of $\sigma$):
\[C_{\sigma} (t) [p_1,...,p_n] = \epsilon(\sigma)\prod_{1 \le k <l\le n} \left( 1 + e^{i(p_{\sigma(k)}+p_{\sigma(l)})} - 2\Delta_t e^{ip_{\sigma(k)}}\right).\]
 \end{notation}
 
 \begin{theorem}
 \label{theorem.coordinate.ansatz}
 For all $N$ and $n \le N/2$, and $(p_1,...,p_n) \in I_t$ distinct such that for all $j$ the following equation is verified:
 
 \[(E_j)[t,n,N]: \qquad Np_j(t) = 2\pi j - (n+1)\pi - \sum_{k=1}^n \Theta_t (p_j (t),p_k(t)).\]
 Then we have: 
 \[V_N (t).\psi_{n,N} (p_1,...,p_n) = \Lambda_{n,N} (t) [p_1,...,p_n] \psi_{n,N} (p_1,...,p_n),\]
 where $\Lambda_{n,N} (t) [p_1,...,p_n]$ is 
 equal to 
 \[\prod_{k=1}^n L_t (e^{ip_k}) + \prod_{k=1}^n M_t (e^{ip_k})\]
 when all the $p_k$ are distinct from $0$. Else, it is equal to: 
\[\left( 2 + t^2 (N-1) + \sum_{k \neq l} 
\frac{\partial \Theta_t}{\partial x} (0,p_k)\right)\prod_{k=1}^n M_t (e^{ip_k})\]
for $l$ such that $p_l = 0$.
 \end{theorem}
 
In~\cite{Duminil-Copin.ansatz} (Theorem 2.2), 
the equations (BE) are implied 
by the equations $(E_j)[t,n,N]$ in Theorem~\ref{theorem.coordinate.ansatz} by taking the 
exponential of the members of $(BE)$.
In order to make the connection easier with 
\cite{Duminil-Copin.ansatz}, here is a list 
of correspondances between the notations: 
in \cite{Duminil-Copin.ansatz}, the notation $t$ 
corresponds to $c$, and it is fixed 
in the formulation of the theorem. Thus, 
$\Delta_t$ corresponds to $\Delta$, 
$\mathcal{I}_t$ to $\mathcal{D}_{\Delta}$, 
$V_N (t)$ to $V$, $\psi_{t,n,N} (p_1,...,p_n)$ 
to $\psi$, $L_t$ and $M_t$ to $L$ and $M$, 
$\Theta_t$ to $\Theta$, $\Lambda_{n,N} (t)[p_1,...,p_n]$ to $\Lambda$, 
$C_{\sigma} (t) [p_1,...,p_n]$ to 
$A_{\sigma}$ and the sequence $(x_k)_k$ 
to the sequence $(q_k [\boldsymbol{\epsilon}])_k$ 
for some $\boldsymbol{\epsilon}$.

\subsection{\label{section.existence} Existence of solutions of Bethe 
equations and analycity}

In this section, we will prove the following, 
which is a rigorous and complete 
version of an argument in 
\cite{YY66}: 

\begin{theorem}
There exists a unique sequence 
of analytic functions 
$\vec{p}_j : (0,\sqrt{2}) \mapsto (-\pi,\pi)$ 
such that for all $t \in (0,\sqrt{2})$, 
$\vec{p}_j (t) \in  I_t $ and we have
the system of Bethe equations: 

\[(E_j)[t,n,N]: \qquad N p_j (t) = 2\pi j - (n+1)\pi - \sum_{k=1}^n \Theta_t (p_j(t) ,p_k(t)).\]

Moreover, for all $t$ and $j$, 
$\vec{p}_{n-j+1} (t) = -\vec{p}_j (t)$; 
for all $t$, the $\vec{p}_j (t)$ are 
all distinct.
\end{theorem}

\noindent \textbf{Idea of the proof:} 
\textit{Following C.N. Yang and C.P. Yang~\cite{YY66}, we use an auxiliary multivariate function $\zeta_t$ whose 
derivative is zero exactly when the equations 
$(E_j) [t,n,N]$ are verified. We prove 
that this function is convex, which means 
that it admits a minimum (this relies on 
the properties of $\theta_t$ and $\kappa_t$). 
Since we rule out the possibility that the 
minimum is on the border of the domain, 
this function admits a point where 
its derivative is zero, and thus the system 
of equations $(E_j)[t,n,N]$ admits 
a unique solution. In order 
to prove the analycity, we then define 
a function of $t$ that verifies an analytic 
differential equation (and thus is analytic), whose value in some point coincides with the minimum of $\zeta_t$. Since the differential equation ensures that $\zeta'_t$ is null 
on the values of this function, this means 
that for all $t$, its value in $t$ 
is the minimum of $\zeta_t$.}
 
\begin{proof}

\begin{itemize}

\item \textbf{The solutions 
are critical points of an auxiliary function 
$\zeta_t$:}

Let us denote, for all $t,p_1,...,p_n$:
\begin{align*}
\zeta_t (p_1,...,p_n) & = N \sum_{j=1}^n 
\int_{0}^{\kappa^{-1}_t (p_j)} \kappa_t (x) dx 
 + \pi (n+1-2j)\sum_{j=1}^n 
 \kappa_t^{-1} (p_j)\\
 & +  \sum_{k < j} \int_{0}^{\kappa_t ^{-1} (p_j)-\kappa_t ^{-1} (p_k)} 
 \theta_t (x,0)dx. 
\end{align*}

The interest of this function lies in 
the fact that for all $j$ (here the argument 
in each of the sums is $k$):

\begin{align*}
\frac{\partial  \zeta_t}{\partial p_j} (p_1, ..., p_n) &  = 
\left( \kappa_t^{-1} \right)' (p_j).
\left( N p_j - 2\pi j + (n+1)\pi -  
\sum_{k < j } \theta_t (\kappa_t ^{-1} (p_j)-\kappa_t ^{-1} (p_k),0) \right) \\
& + \sum_{j< k} \theta_t (\kappa_t ^{-1} (p_k)-\kappa_t ^{-1} (p_j),0) .\\
&  = 
\left( \kappa_t^{-1} \right)' (p_j).
\left( N p_j - 2\pi j + (n+1)\pi -  
\sum_{k < j } \theta_t (\kappa_t ^{-1} (p_k),\kappa_t ^{-1} (p_j)) \right) \\
& + \sum_{j< k} \theta_t (\kappa_t ^{-1} (p_j),\kappa_t ^{-1} (p_k)).\\
& = \left( \kappa_t^{-1} \right)' (p_j). 
\left( N p_j - 2\pi j + (n+1)\pi + \sum_{k}
\Theta_t (p_j,p_k)\right), 
\end{align*}
since for all $x,y$, $\Theta_t (x,y) = 
- \Theta_t (y,x)$. 

Hence, the system of Bethe equations is 
verified for the sequence $(p_j)_j$ if 
and only for all $j$, $ \frac{\partial \zeta_t }{\partial p_j} 
(p_1, ..., p_n )=0$.

\item \textbf{Convexity of $\zeta_t$:}

Let us denote $\tilde{\zeta}_t : \mathbb{R} ^n 
\rightarrow \mathbb{R}$ such that 
for all $\alpha_1, ... , \alpha_n$: 
\[\tilde{\zeta}_t (\alpha_1 , ... , \alpha_n)
= \zeta_t (\kappa_t (\alpha_1), ... , \kappa_t (\alpha_n)).\]
From the last point, we have that for all 
sequence $(\alpha_k)_k$ and all $j$: 
\[\frac{\partial \tilde{\zeta}_t}{\partial p_j}  (\alpha_1, ... , \alpha_n) = N \kappa_t (\alpha_j) 
- 2\pi j + (n+1)\pi + \sum_k \theta_t (\alpha_j,\alpha_k).\]
As a consequence, for all $k \neq j$: 
\[\frac{\partial ^2 \tilde{\zeta}_t}{\partial p_k \partial p_j}  (\alpha_1, ... , \alpha_n) =  \frac{\partial \theta_t}{\partial \beta} 
 (\alpha_j,\alpha_k) = 
\frac{\sin(2\mu_t)}{\cosh(\alpha_j-\alpha_k) 
- \cos(2\mu_t)}.\]

Moreover, for all $j$:
\[ \frac{\partial^2 \tilde{\zeta}_t}{\partial^2 p_j}  (\alpha_1, ... , \alpha_n , t) = N \kappa'_t (\alpha_j) + \sum_{k \neq j}  
\frac{\partial \theta_t}{\partial \alpha}  (\alpha_j, \alpha_k) = 
N \kappa'_t (\alpha_j) - \sum_{k \neq j}  
\frac{\partial \theta_t}{\partial \beta}  (\alpha_j, \alpha_k).\] 

Let us denote $\tilde{H}_t (\alpha_1 , ... , \alpha_n)$ 
the Hessian matrix of $\tilde{\zeta}_t$. For any 
$(x_1, ... , x_n) \in \mathbb{R}^n$, we 
have 
\begin{align*}
\tilde{H}_t (\alpha_1 , ... , \alpha_n ) & = 
N \sum_j \kappa'_t (\alpha_j) x_j ^2 
+ \sum_{j \neq k} \left( \frac{\partial \theta_t}{\partial \beta}  (\alpha_j, \alpha_k) x_j (x_j-x_k).\right)\\
& = N \sum_j \kappa'_t (\alpha_j) x_j ^2 
+ \sum_{j < k} \left( \frac{\partial \theta_t}{\partial \beta}  (\alpha_j, \alpha_k) x_j (x_j-x_k) \right)\\
& \qquad + \sum_{j < k} \left( \frac{\partial \theta_t}{\partial \beta}  (\alpha_k, \alpha_j) x_k (x_k-x_j) \right)\\
& = N \sum_j \kappa'_t (\alpha_j) x_j ^2 
+ \sum_{j < k} \left( \frac{\partial  \theta_t}{\partial \beta} (\alpha_j, \alpha_k) (x_j-x_k)^2 \right) > 0 
\end{align*}

As a consequence $\tilde{\zeta}_t$ 
is a convex function. As a consequence, 
if it has a (local) minimum, it is unique. 
Since $\kappa_t$ is increasing, this property 
is also true for $\zeta_t$. 

\item \textbf{The function $\zeta_t$ 
has a minimum in $I_t ^n$:}

Let us consider $(C_l)_l$ an increasing 
sequence of compact intervals such 
that $\bigcup_l C_l = I_t$.

Let us assume that $\zeta_t$ has no minimum 
in $I_t ^n$. As a consequence, for all $j$, 
the minimum $\vec{p}^{(l)}$ of $\zeta_t$ on $(C_l)^n$ 
is on its border. Without loss of generality, 
we can assume that there exists 
some $\vec{p}^{(\infty)} \in 
\overline{I_t} ^n$ such that $\vec{p}^{(l)} 
\rightarrow \vec{p}^{(\infty)} $. 

We can assume without loss of generality 
that there exists some $j_0 \in 
\llbracket 1,n\rrbracket$ such that 
$j \le j_0$ if and only if
$\vec{p}^{(\infty)}_j = \pi-\mu_t$. In 
this case, there exists $l_0$ 
such that for all $l$ and $j \le j_0$, 
$\vec{p}_j^{(l)} \ge 0$. 
Since $\tilde{\zeta}_t$ is convex 
and that $\vec{p}^{(l)}$ is a minimum 
for this function on the compact $(C_l)^n$, then 
for all $j \le j_0$, 
\[\frac{\partial \zeta_t}{\partial p_j}  (\vec{p}^{(l)}_1, ... , 
\vec{p}^{(l)}_n) \le 0.\]
This is a particular case of the fact that for 
a convex and continuously differentiable function $f : I \mapsto \mathbb{R}$, 
where $I$ is a compact interval of $\mathbb{R}$, if its minimum on $I$ is the maximal 
element of this compact, then $f'$ is negative on this point, as illustrated 
on Figure~\ref{figure.convex.minimum}. 

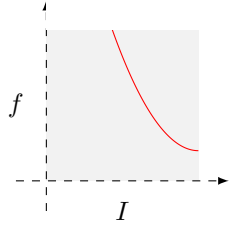
\begin{figure}[!h]
\[\begin{tikzpicture}[scale=0.4]
\fill[gray!10] (0,0) rectangle (5,5);
\draw[dashed,-latex] (-1,0) -- (6,0); 
\draw[dashed, -latex] (0,-1) -- (0,6);
\begin{scope}[xshift=5cm,yshift=1cm]
\draw[color=red,smooth, samples =100, domain=-3:0] plot(\x,0.5*\x*\x);
\end{scope}
\node at (-1,2.5) {$f$};
\node at (2.5,-1) {$I$};
\fill[white] (0,5) rectangle(5,7);
\end{tikzpicture}\]
\caption{\label{figure.convex.minimum} Illustration of the fact that 
the minimum of a convex continuously differentiable function 
on a real compact interval has non-positive derivative.}
\end{figure}

Since $\Theta_t$ cannot 
be defined on $\{(x,x): x \in \partial I_t\}$, 
in order to have an inequality that 
can be transformed by continuity 
into an inequality on $\vec{p}$, we sum 
these inequalities: 

\[\sum_{j=1}^{j_0} \frac{\partial \zeta_t}{\partial p_j}  (\vec{p}^{(l)}_1, ... , 
\vec{p}^{(l)}_n) \le 0.\]
According to the first point of the proof, 
this inequality can be re-writen: 
\[N \sum_{j=1}^{j_0} \vec{p}_j - 2\pi \sum_{j=1}^{j_0} j + j_0 (n+1) \pi + \sum_{j=1}^{j_0} \sum_k 
\Theta_t (\vec{p}_j^{(l)}, \vec{p}_k^{(l)}) \le 0.\]

For all $j,j' \le j_0$, 
the terms $\Theta_t (\vec{p}_j^{(l)},\vec{p}_{j'}^{(l)})$ and $\Theta_t (\vec{p}_{j'}^{(l)},\vec{p}_{j}^{(l)})$ cancel out in this sum. As a consequence: 

\[N \sum_{j=1}^{j_0} \vec{p}_j - 2\pi \sum_{j=1}^{j_0} j + j_0 (n+1) \pi + \sum_{j=1}^{j_0} \sum_{k>j_0} 
\Theta_t (\vec{p}_j^{(l)}, \vec{p}_k^{(l)}) \le 0.\]

This time, the inequality can be extended by 
continuity and we obtain: 

\[N \sum_{j=1}^{j_0} \vec{p}_j - 2\pi \sum_{j=1}^{j_0} j + j_0 (n+1) \pi + \sum_{j=1}^{j_0} \sum_{k>j_0} 
\Theta_t (\vec{p}_j^{(\infty)}, \vec{p}_k^{(\infty)}) \le 0.\]

From Computation~\ref{computation.theta.border}, we have: 

\[N j_0 (\pi-\mu_t) - 2\pi \sum_{j=1}^{j_0} 
j + j_0 (n+1) \pi  + j_0 (n-j_0). (2\mu_t - \pi) \le 0.\]

Since $\mu_t \le \pi$ and that 
$2j_0 (n-j_0) - Nj_0 = -2 j_0 ^2 < 0$, this 
last inequality implies: 

\[j_0 (n+1) \pi  + j_0 (n-j_0). \pi \le 2\pi \sum_{j=1}^{j_0} 
j.\]

On the other hand, we have: 
\[\sum_{j=1}^{j_0} j \le n j_0 - \sum_{j=1}^{j_0} j = nj_0 - \frac{j_0 (j_0+1)}{2}.\]

As a consequence: 

\begin{align*}
j_0 (n+1) \pi  + j_0 (n-j_0). \pi & \le 2\pi n j_0 - j_0 (j_0 +1) \pi\\
(2n+1) j_0 \pi - j_0 ^2 \pi & \le 2\pi n j_0 
- j_0 ^2 \pi - j_0 \pi \\
 j_0 \pi & \le -j_0 \pi
\end{align*}

Since this last inequality is impossible, 
this means that $\zeta_t$ has a minimum 
in $I_t^n$.

\item \textbf{Characterization of the solutions 
with an analytic differential equation:}

Let us denote $\vec{p} (t) 
= (\vec{p}_1 (t) , ... , \vec{p}_n (t))$, for all $t \in (0, \sqrt{2})$, 
the unique minimum of the function $\zeta_t$ 
in $I_t ^n$. 
Let us denote $t \mapsto \vec{s} (t)$ the unique solution 
of the differential equation: 

\[ 
{\vec{s}}'(t) = -\left( H_t (\vec{s}_1 (t) , ... , \vec{s}_n (t))\right) ^{-1} .\left( \frac{\partial^2 \zeta_t} {\partial t \partial p_j }  (\vec{s}_1 (t), ... , \vec{s}_n (t)) \right)_j\]
such that $\vec{s}(t)$ is the minimum 
of the function $\zeta_t$ when $t=\sqrt{2}/2$, 
where $H_t$ is the Hessian 
matrix of $\zeta_t$. Since this is an analytic 
differential equation, it solution $\vec{s}$ 
is analytic. 

Let us rewrite the equation:

\[H_t (\vec{s}_1 (t) , ... , \vec{s}_n (t)) .
\vec{s}'(t) = -\left(\frac{\partial^2 \zeta_t} {\partial t \partial p_j }   (\vec{s}_1 (t), ... , \vec{s}_n (t)) \right)_j \]

\[\frac{\partial^2 \zeta_t} {\partial t \partial p_j }   (\vec{s}_1 (t), ... , \vec{s}_n (t)) 
+ \sum_k \vec{s}'_k (t) . \frac{\partial^2 \zeta_t} {\partial p_k \partial p_j }  (\vec{s}_1 (t), ... , \vec{s}_n (t)) = 0 \]

This means that for all $j$, 
\[\frac{\partial \zeta_t} {\partial p_j }  (\vec{s}_1 (t), ... , 
\vec{s}_n (t) )\]
is a constant. Since $\vec{s}(t)$ 
is the minimum of $\zeta_t$ when $t=\sqrt{2}/2$, 
this constant is zero. As a consequence, 
by unicity of the minimum of $\zeta_t$ for 
all $t$, $\vec{s}(t) = \vec{p}(t)$. 
This means that $t \mapsto \vec{p} (t)$ 
is analytic. 

\item \textbf{Antisymmetry of 
the solutions:} 

For all $t,j$, since $\vec{p} (t)$ 
is the minimum of $\zeta_t$: 

\begin{align*}
N\vec{p}_{n-j+1} - 2\pi (n-j+1) + (n+1)\pi 
+\sum_k \Theta_k (\vec{p}_{n-j+1}, \vec{p}_{n-k+1}) & = 0.\\
N\vec{p}_{n-j+1} + 2\pi j - (n+1)\pi 
+\sum_k \Theta_k (\vec{p}_{n-j+1}, \vec{p}_{n-k+1}) & = 0.\\
-N\vec{p}_{n-j+1} - 2\pi j + (n+1)\pi 
-\sum_k \Theta_k (\vec{p}_{n-j+1}, \vec{p}_{n-k+1}) & = 0.\\
-N\vec{p}_{n-j+1} - 2\pi j + (n+1)\pi 
+\sum_k \Theta_k (-\vec{p}_{n-j+1}, -\vec{p}_{n-k+1}) & = 0.\\
\end{align*}
This means that the sequence $(-\vec{p}_{n-j+1} (t))_j$ is a minimum for $\zeta_t$, and 
as a consequence, for all $j$, 
$\vec{p}_{n-j+1} = - \vec{p}_j$.

\item \textbf{The numbers $\vec{p}_j(t)$ 
are all distinct:}

Let us consider the function 
\[\chi_t : \alpha \mapsto N \kappa_t (\alpha) + 
\sum_{k=1}^{n} \theta_t (\alpha, \alpha_k (t)),\]
where for all $k$, $\boldsymbol{\alpha}_k (t)$ 
is equal to $\kappa_t^{-1} (\vec{p}_k (t))$.
For all $j$, this function has value $\pi(2j-(n+1))$ in $\boldsymbol{\alpha}_j (t)$ (by 
the Bethe equations). The finite sequence 
$(\pi(2j-(n+1)))_j$ is increasing, thus 
it is sufficient to prove that the function $\chi_t$
is increasing. Its derivative is: 

\[\chi'_t : \alpha \mapsto N \kappa'_t (\alpha) 
+ \sum_{k=1}^n \frac{\partial \theta_t}{\partial \alpha} (\alpha, \boldsymbol{\alpha}_k (t)).\]

Since $t \in (0,\sqrt{2})$, $\sin(\mu_t)<0$, 
and thus this function is positive. As 
a consequence $\chi_t$ is increasing.

\end{itemize}

\end{proof}

\subsection{\label{section.hamiltonian} Diagonalisation of some Heisenberg Hamiltonian}

In this section, following the technique introduced by Lieb, Schultz and 
Mattis~\cite{LSM61}, we diagonalise some  Hamiltonian (which is a matrix acting  
on $\Omega_N$). 

\subsubsection{Bosonic creation and anihilation operators}

Let us recall that $\Omega_N = \mathbb{C}^2 \otimes ... 
\otimes \mathbb{C}^2$. In this section, for the purpose of 
notation, we identify $\{1,...,N\}$ with 
$\mathbb{Z}/N\mathbb{Z}$.

\begin{notation}
Let us denote $a$ and $a^{*}$ the matrices in $\mathcal{M}_{2} (\mathbb{C})$ 
equal to 
\[ a \equiv \left(\begin{array}{cc} 0 & 0 \\ 1 & 0 \end{array}\right), \quad  a^{*} \equiv \left(\begin{array}{cc} 0 & 1 \\ 0 & 0 \end{array}\right).\]
For all $j \in \mathbb{Z}/N\mathbb{Z}$, we denote 
$a_j$ (\textbf{creation} operator at position $j$) and $a^{*}_j$ (\textbf{anihilation} operator at position $j$) the matrices in $\mathcal{M}_{2^N}(\mathbb{C})$ equal 
to 
\[a_j \equiv id \otimes ... \otimes a \otimes ... \otimes id, \quad a^{*}_j \equiv id \otimes ... \otimes a^{*} \otimes ... \otimes id.\]
where $id$ denotes the identity, and $a$ acts on 
the $j$th copy of $\mathbb{C}^2$.
\end{notation}

In other words, the image 
of a vector $\ket{\boldsymbol{\epsilon}_1 ... \boldsymbol{\epsilon}_N}$
in the basis of $\Omega_N$ by $a_j$ (resp. $a^{*}_j$) is as follows:

\begin{itemize}
\item if $\boldsymbol{\epsilon}_j=0$ (resp. $\boldsymbol{\epsilon}_j = 1$), then the image vector is $\vec{0}$; 
\item if $\boldsymbol{\epsilon}_j=1$ (resp. $\boldsymbol{\epsilon}_j = 0$), then the image vector is $\ket{\boldsymbol{\eta}_1 ... \boldsymbol{\eta}_N}$ such that 
$\boldsymbol{\eta}_j = 0$ (resp. $\boldsymbol{\eta}_j = 1$) and 
for all $k \neq j$, $\boldsymbol{\eta}_k = \boldsymbol{\epsilon}_k$. 
\end{itemize}

\begin{remark}
The term creation (resp. anihilation) refer to 
the fact that for two elements $\boldsymbol{\epsilon}$, $\boldsymbol{\eta}$ of the 
basis of $\Omega_N$, $a_j [\boldsymbol{\epsilon},\boldsymbol{\eta}] \neq 0$ (resp. $a_j^{*} [\boldsymbol{\epsilon},\boldsymbol{\eta}] \neq 0$) implies that 
$|\boldsymbol{\eta}|_1 = |\boldsymbol{\eta}|+1$
(resp. $|\boldsymbol{\eta}|_1 = |\boldsymbol{\eta}|-1$). If we think of $1$ symbols 
as particles, this operator acts by creating (resp. anihilating)
a particle.
\end{remark}

\begin{lemma} \label{lemma.lowering.operators}
The matrices $a_j$ and $a^{*}_j$ verify the 
following properties, for all $j$ and $k \neq j$:
\begin{itemize}
\item $a_j a^{*}_j + a^{*}_j a_j = id$.
\item $a_j ^2 = {a^{*}_j}^2 = 0$.
\item $a_j$, $a^{*}_j$ 
commute both with $a_k$ and $a^{*}_k$.
\end{itemize}
\end{lemma}

\begin{proof}
\begin{itemize}
\item By straightforward computation, we get 
\[a a^{*} = \left( \begin{array}{cc} 0 & 0 \\ 1 & 0 \end{array}\right) \left( \begin{array}{cc} 0 & 1 \\ 0 & 0 \end{array}\right) = \left( \begin{array}{cc} 0 & 0 \\ 0 & 1 \end{array}\right)\]
and  
\[a^{*} a = \left( \begin{array}{cc} 0 & 1 \\ 0 & 0 \end{array}\right)\left( \begin{array}{cc} 0 & 0 \\ 1 & 0 \end{array}\right) = \left( \begin{array}{cc} 1 & 0 \\ 0 & 0 \end{array}\right)\]
Thus $a a ^{*} + a^{*} a$ is the indentity 
of $\mathbb{C}^2$. As a consequence, for all $j$, 
\[a_j a ^{*}_j + a^{*}_j a_j = id \otimes ... \otimes id,\]
which is the identity of $\Omega_N$. 
\item The second set of equalities comes directly from 
\[a ^2 = \left( \begin{array}{cc} 0 & 0 \\ 1 & 0 \end{array}\right) \left( \begin{array}{cc} 0 & 0 \\ 1 & 0 \end{array}\right) = \left( \begin{array}{cc} 0 & 0 \\ 0 & 0 \end{array}\right) \]
\[\left({a^{*}}\right) ^2 = \left( \begin{array}{cc} 0 & 1 \\ 0 & 0 \end{array}\right) \left( \begin{array}{cc} 0 & 1 \\ 0 & 0 \end{array}\right) =  \left( \begin{array}{cc} 0 & 0 \\ 0 & 0 \end{array}\right).\]
\item The last set derives from the fact that any 
operator on $\mathbb{C}^2$ commutes with the 
identity.
\end{itemize}
\end{proof}

\subsubsection{Definition and properties of the heisenberg hamiltonian}

\begin{notation}
Let us denote $H_N$ the matrix in $\mathcal{M}_{2^N} (\mathbb{C})$ defined as:
\[H_N = \displaystyle{\sum_{j \in 
\mathbb{Z}/N\mathbb{Z}} \left(a^{*}_{j} a_{j+1} + a_{j} a^{*}_{j+1}\right)}\]
\end{notation}

\begin{lemma}
\label{lemma.hamiltonian.properties}
This matrix $H_N$ is non-negative, symmetric
and for all $n$, its restriction to $\Omega_N^{(n)}$ is irreducible.
\end{lemma}

The proof of Lemma~\ref{lemma.hamiltonian.properties} is similar to the one of Lemma~\ref{lemma.properties.matrix.transfer}, 
following the interpretation of 
the action of $H_N$ described in Remark~\ref{remark.hamiltonian}:

\begin{remark} \label{remark.hamiltonian}
For all $j$, $a^{*}_{j} a_{j+1} 
+ a_j a^{*}_{j+1}$ acts on a vector 
$\boldsymbol{\epsilon}$ in the basis of $\Omega_N$ 
by exchanging the symbols in positions $j$ 
and $j+1$ if they are different. If they are not, 
the image of $\boldsymbol{\epsilon}$ by 
this matrix is $\vec{0}$.
As a consequence, for two vectors $\boldsymbol{\epsilon}$ and 
$\boldsymbol{\eta}$ in the basis of $\Omega_N$, $H_N[\boldsymbol{\epsilon},\boldsymbol{\eta}]
\neq 0$ if and only if $\boldsymbol{\eta}$ 
is obtained from $\boldsymbol{\epsilon}$ by exchanging a $1$ 
symbol of $\boldsymbol{\epsilon}$ with a $0$ in its neighborhood. The Hamiltonian $H_N$ thus 
corresponds to $H$ in \cite{Duminil-Copin.ansatz} 
for $\Delta=0$.
\end{remark}

\subsubsection{Fermionic creation 
and anihilation operators}

\begin{notation}
Let us denote $\sigma$ the matrix of 
$\mathcal{M}_{2} (\mathbb{C})$ defined as: 
\[\sigma = \left( \begin{array}{cc} 1 & 0 \\ 0 & -1 \end{array} \right).\]

Let us denote, for all $j \in \mathbb{Z}/N\mathbb{Z}$, $c_j$ and $c^{*}_j$
the matrices
\[c_j = \sigma \otimes ... \sigma \otimes a \otimes id \otimes ... \otimes id, \quad c^{*}_j = \sigma \otimes ... \sigma \otimes a^{*} \otimes id \otimes ... \otimes id.\]
\end{notation}

Let us recall that two matrices $P,Q$ anticommute when 
$PQ = - QP$.

\begin{lemma} \label{lemma.creation.operators}
These operators verify the following properties 
for all $j$ and $k \neq j$: 
\begin{itemize}
\item $c_j c^{*}_j + c^{*}_j c_j = id$.
\item $c^{*}_j$ and $c_j$ anticommute with both $c^{*}_k$ 
and $c_k$.
\item $a^{*}_{j+1} a_j = - c^{*}_{j+1} c_j$ 
and $a^{*}_j a_{j+1} = - c^{*}_{j} c_{j+1}$.
\end{itemize}
\end{lemma}

\begin{proof}
\begin{itemize}
\item Since $\sigma ^ 2 = id$, 
for all $j$, 
\[c_j c^{*}_j + c^{*}_j c_j = a_j a^{*}_j + a^{*}_j a_j.\]
From Lemma~\ref{lemma.lowering.operators}, we 
now that this operator is equal to identity.
\item We can assume without loss of generality that 
$j <k$. Let us 
prove that $c_j$ anticommutes with $c_k$ (the other 
cases are similar): 
\[c_j c_{k} = id \otimes ... \otimes a \sigma \otimes 
\sigma \otimes ... \otimes \sigma \otimes \sigma a \otimes 
id \otimes ... \otimes id.\]
\[c_j c_{k} = id \otimes ... \otimes \sigma a \otimes 
\sigma \otimes ... \otimes \sigma \otimes a \sigma \otimes 
id \otimes ... \otimes id.\]
Hence it is sufficient to see:
\[\sigma a = \left( \begin{array}{cc} 1 & 0 \\
0 & -1 \end{array}\right) \left( \begin{array}{cc} 0 & 0 \\
1 & 0 \end{array}\right) = \left( \begin{array}{cc} 0 & 0 \\
-1 & 0 \end{array}\right)\]
\[a \sigma = \left( \begin{array}{cc} 0 & 0 \\
1 & 0 \end{array}\right) \left( \begin{array}{cc} 1 & 0 \\
0 & -1 \end{array}\right) = \left( \begin{array}{cc} 0 & 0 \\
1 & 0 \end{array}\right)= - \sigma a\]
\item Let us prove the first equality (the other 
one is similar): 
\[c^{*}_{j+1} c_j = id \otimes ... \otimes id \otimes 
\sigma a \otimes a^{*} \otimes id \otimes ... \otimes id.\]
We have just seen in the last point 
that $\sigma a = - a$. As a consequence 
$c^{*}_{j+1} c_j = - a^{*}_{j+1} a_j$.
\end{itemize}
\end{proof}

\subsubsection{\label{section.action.orthogonal.matrix} Action of a symmetric orthogonal matrix}

Let us denote $c^*$
is the vector $(c^{*}_{1}, ... , c^{*}_{N})$ and $c^t$ 
is the transpose of the 
vector $(c_{1}, ... , c_{N})$.
Let us consider a symmetric and 
orthogonal matrix $U=  \left( u_{i,j} \right)_{i,j}$ 
in $\mathcal{M}_{N} (\mathbb{R})$ and 
denote $b$ and $b^{*}$ the matrices: 
\[b = U.c^{t} = (b_1, ... , b_N), \quad b^{*} = c^{*}.U^t = (b^{*}_1 , ... , b^{*}_N).\]

\begin{notation}
For all $\alpha \in \{0,1\}^N$, 
we denote: 
\[\psi_{\alpha} = (b^{*}_1)^{\alpha_1} ... 
(b^{*}_N)^{\alpha_N} \boldsymbol{\nu}_N,\]
where $\boldsymbol{\nu}_N = \ket{0,...,0}$. 
\end{notation}

\begin{lemma} \label{lemma.operateurs.b}
For all $j$ and $k \neq j$: 
\begin{itemize}
\item $b_j$ and $b^{*}_j$ anticommute with 
both $b_k$ and $b^{*}_k$ and $b_j b^{*}_j + b^{*}_j b_j 
= id$.
\item 
For all $\alpha \in \{0,1\}^N$, $\psi_{\alpha} \neq \vec{0}$.
\item For all $j$ and $\alpha$, we have:
\begin{enumerate}
\item $b^{*}_j b_j \psi_{\alpha} = \textbf{0}$ if $\alpha_j = 0$,
\item $b^{*}_j b_j \psi_{\alpha} = \psi_{\alpha}$ if $\alpha_j = 1$.
\end{enumerate}
\end{itemize}
\end{lemma}

\begin{proof}
\begin{itemize}
\item \textbf{Anticommutation relations:} 

Let us prove that $b_j$ and $b^{*}_k$ 
anticommute (the other statements of the first 
point have a similar proof). We rewrite the definition 
of $b_j$ and $b^*_k$:
\[b_j = \sum_i u_{i,j} c_i\qquad \text{and} \qquad b^{*}_k = \sum_i u_{k,i} c^{*}_i = 
\sum_i u_{i,k} c^{*}_i.\]
Thus \[b_j b^{*}_k = 
\sum_{i} \sum_{l \neq i} u_{i,j} u_{l,k} c_i c^{*}_l + \sum_i 
u_{i,j} u_{i,k} c_i c^{*}_i.\]
From Lemma~\ref{lemma.creation.operators}, 
\[b_j b^{*}_k = 
- \sum_{l} \sum_{i \neq l} u_{i,j} u_{l,k} c^{*}_l c_i + \sum_i u_{i,j} u_{i,k} (id -  c^{*}_i c_i).\]
Since the matrix $U$ is orthogonal, 
\[b_j b^{*}_k = 
- \sum_{l} \sum_{i \neq l} u_{i,j} u_{l,k} c^{*}_l c_i - \sum_i u_{i,j} u_{i,k} c^{*}_i c_i = - b^{*}_k b_j.\]
\textit{Let us notice that this step is the reason why we use the 
operators $c_i$ instead of the operators $a_i$}.
\item For all $k$, $b^{*}_k = \sum_{l} u_{k,l} a^{*}_l.$
As a consequence, for a sequence $k_1, ... , k_s$, 
\[b^{*}_{k_1} ... b^{*}_{k_s}.\boldsymbol{\nu}_N 
= \sum_{l_1} ... \sum_{l_s} \left(\prod_{j=1}^s u_{k_j,l_j} \right)
\left(\prod_{j=1}^s 
a^{*}_{l_j}\right).\boldsymbol{\nu}_N.\]
Since $(a^{*})^2 = 0$, the sum can be considered on 
the integers $l_1,...,l_s$ such that they are two by 
two distinct. The operator 
$a^{*}_{l_1} ... a^{*}_{l_s}$ acts on 
$\boldsymbol{\nu}_N$ by changing the $0$ on positions 
$l_1, ..., l_s$ into symbols $1$. The coefficient 
of the image of $\boldsymbol{\nu}_N$ by 
this operator in the vector $b^{*}_{k_1} ... b^{*}_{k_s}.\boldsymbol{\nu}_N$ is thus:
\[\sum_{\sigma \in \Sigma_s} \prod_{j=1}^s u_{k_j,l_{\sigma(j)}}.\]
If this coefficient was equal to zero for all $\sigma$, 
it would mean that any size $s$ sub-matrice of $U$
have determinant equal to zero, which 
is impossible since $U$ is orthogonal, and
thus invertible. As a consequence, none of the 
vectors $\psi_{\alpha}$ is to zero.
\item When $\alpha_j = 0$, from the fact 
that when $j \neq k$, $b_j$ and $b^{*}_k$ anticommute, 
we get that 
\[b_j \psi_{\alpha} = (-1)^{|\alpha|_1} (b^{*}_1)^{\alpha_1} ... 
(b^{*}_N)^{\alpha_N} b_j \boldsymbol{\nu}_N,\]
and $b_j \boldsymbol{\nu}_N = \vec{0}$, since for all $j$, 
$a_j \boldsymbol{\nu}_N = \vec{0}$. As a consequence $b^{*}_j b_j \boldsymbol{\nu}_N = \vec{0}$.
When $\alpha_j = 1$, by the anticommutation 
relations:
\[b^{*}_j b_j \psi_{\alpha} = (b^{*}_1)^{\alpha_1} ... 
(b^{*}_{j-1})^{\alpha_{j-1}}  b^{*}_j b_j b^{*}_j 
(b^{*}_{j+1})^{\alpha_{j+1}} ...
(b^{*}_N)^{\alpha_N} \boldsymbol{\nu}_N,\]
since the coefficients 
$-1$ introduced by anticommutation are canceled 
out by the fact that we use it for $b_j$ and $b^{*}$
From the first point:
\[b^{*}_j b_j \psi_{\alpha} = (b^{*}_1)^{\alpha_1} ... 
(b^{*}_{j-1})^{\alpha_{j-1}}  b^{*}_j (id- b^{*}_j b_j) 
(b^{*}_{j+1})^{\alpha_{j+1}} ...
(b^{*}_N)^{\alpha_N}.\]

\begin{align*}
b^{*}_j b_j \psi_{\alpha} & =  (b^{*}_1)^{\alpha_1} ... 
(b^{*}_{j-1})^{\alpha_{j-1}}  b^{*}_j (id- b^{*}_j b_j) 
(b^{*}_{j+1})^{\alpha_{j+1}} ...
(b^{*}_N)^{\alpha_N} \\
& = \psi_{\alpha} - (b^{*}_1)^{\alpha_1} ... 
(b^{*}_N)^{\alpha_N} b_j \boldsymbol{\nu}_N \\
& = \psi_{\alpha}.
\end{align*}
\end{itemize}
\end{proof}

\subsubsection{Diagonalisation of the Hamiltonian}

\begin{theorem}
\label{theorem.eigenvalues.hamiltonian}
The eigenvalues of $H_N$ are exactly the numbers: 
\[2 \sum_{\alpha_j =1} \cos\left(\frac{2\pi j }{N}\right),\]
for $\alpha \in \{0,1\}^N$.
\end{theorem}

\begin{proof}

\begin{enumerate}
\item \textbf{Rewriting $H_N$:} 

From Lemma~\ref{lemma.creation.operators}, we 
can write $H_N$ as: 
\[H_N = \sum_j c^{*}_j c_{j+1} + c^{*}_{j+1} c_j.\]
The Hamiltonian $H_N$ can be then rewritten 
as $H_N= c^* M c^{t}$, where 
$M$ is the matrix defined by blocks 
\[M = \frac{1}{2} \left( \begin{array}{cccc} \textbf{0} & id & & id \\
id & \ddots & \ddots & \\
& \ddots & \ddots & id \\
id & & id & \textbf{0} \end{array}\right),\]
where $id$ denotes the 
identity matrix on $\mathbb{C}^2$, 
and \textbf{0} denotes the null 
matrix. Let 
us denote $M'$ the matrix 
of $\mathcal{M}_{2^N} (\mathbb{R})$ 
obtained from $M$ by replacing 
$\textbf{0},id$ by $0,1$. 
\item \textbf{Diagonalisation of $M$:}
The matrix $M'$ is symmetric and 
thus can be diagonalised in $\mathcal{M}_{2^N} (\mathbb{R})$ in an orthogonal 
basis. It is rather 
straightforwards to see that 
the vectors $\psi^k$, for any $k \in \{0,...,N-1\}$ are an orthonormal family of eigenvectors 
of $M'$ for the eigenvalue 
$\lambda_k = \cos(\frac{2\pi k}{N})$, 
where for all $j \in \{1,...,N\}$, 
\[\psi^k_j = \sqrt{\frac{2}{N}} \left(\sin(\frac{2\pi kj}{N}),\cos(\frac{2\pi kj}{N})\right).\]
This comes from the equalities 
\[\cos(x-y)+\cos(x+y) = 2\cos(x)\cos(y),\]
\[\sin(x-y)+\sin(x+y) = 2\cos(x)\sin(y),\]
applied to $x=k(j-1)$ and $y=k(j+1)$.
This family of vectors is free, 
since the Vandermonde matrix 
with coefficients $e^{2\pi kj/N}$ 
is invertible. As a consequence, 
one can write 
\[U' M' {U'}^t = D',\]
where $D'$ is the diagonal 
matrix whose diagonal coefficients 
are the numbers $\lambda_k$, and 
$U'$ is the orthogonal matrix 
given by the vectors $\psi^k$. 
Replacing any coefficient of these 
matrices by the product 
of this coefficient with the 
identity, one gets an orthogonal matrix 
$U$ and a diagonal one $D$ such that:
\[U M U^t = D.\]
\item \textbf{Some eigenvectors of $H_N$:}
Let us consider the vectors $\psi_{\alpha}$ 
constructed in Section~\ref{section.action.orthogonal.matrix}
for the matrix $U$ of the 
last point, which is symmetric and 
orthogonal.
From the expression of $H_N$, 
we get that \[H_N \psi_{\alpha} = \left(2 
\sum_{j:\alpha_j = 1} \cos\left( \frac{2\pi j}{N} \right)\right).\psi_{\alpha}.\]
Since $\psi_{\alpha}$ is non zero, this is an eigenvector 
of $H_N$. 

\item \textbf{The family $(\psi_{\alpha})$ is a basis of $\Omega_N$:}

From cardinality of this family (the 
number of possible $\alpha$, equal to 
$2^N$), 
this is sufficient to prove that 
this family is free. For this purpose, let us assume 
that there are exists a sequence $(x_{\alpha})_{\alpha \in \{0,1\}}$
such that 
\[\sum_{\alpha \in \{0,1\}_N} x_{\alpha}.\psi_{\alpha} = \vec{0}.\] 
We apply first $b^{*}_1 b_1 .... b^{*}_N b_N$ 
and get that $x_{(1,...,1)} \psi_{(1,...,1)} = 0,$
and thus $x_{(1,...,1)}=0$ (by Lemma~\ref{lemma.operateurs.b},
the vector $\psi_{(1,...,1)}$ is 
not equal to zero).
Then we apply successively the operators 
$\prod_{j\neq k} b^{*}_j b_j$
for all $k$, and obtain that 
for all $\alpha \in \{0,1\}_N$
such that $|\alpha|_1 = N-1$,
$x_{\alpha}=0$. By repeating this argument, 
we obtain that all the coefficient $x_{\alpha}$ 
are null. As a consequence $(\psi_{\alpha})_{\alpha}$
is a base of eigenvectors for $H_H$, and 
the eigenvalues obtained in the 
last point cover all the eigenvalues of $H_N$.
\end{enumerate}
\end{proof}

\subsection{ \label{subsection.identification} Identification}

The proofs for the 
following two lemmas can be found in~\cite{Duminil-Copin.ansatz} (respectively Lemma 5.1 and Theorem 2.3). In Lemma~\ref{lemma.commutation}, our notation $H_N$ 
corresponds to their notations $H$ for $\Delta=0$, 
and $V_N (\sqrt{2})$ corresponds to $V$ 
for $\Delta=0$. In Lemma~\ref{lemma.eigenvalue.hamiltonian}, 
the equations $(E_j)[\sqrt{2},n,N]$ correspond 
to their $(BE)$, $\psi$ to $\psi$ for $\Delta=0$.

\begin{lemma}
\label{lemma.commutation}
For all $N \ge 1$, the Hamiltonian $H_N$ and $V_N (\sqrt{2})$ 
commute: 
\[H_N . V_N (\sqrt{2}) = V_N (\sqrt{2}) . H_N.\]
\end{lemma}

\begin{lemma}
\label{lemma.eigenvalue.hamiltonian}
For all $N$ and $n\le N$, let us denote $(\vec{p}_j)_j$ the solution 
of the system of equations $(E_j)[\sqrt{2},n,N]$, then denoting $\psi \equiv \psi_{\sqrt{2},n,N} (\vec{p}_1, ... , \vec{p}_n)$:
\[H_N . \psi = \left( 2 \sum_{k=1}^n \cos(\vec{p}_k)\right). \psi \]
\end{lemma}

Let us prove 
that for all $t \in (0,\sqrt{2})$, the greatest eigenvalue of $V_N (t)$ is given by the algebraic 
Bethe ansatz:

\begin{theorem}
For all $N$ and $n \le N/4$, and $t \in (0,\sqrt{2})$, 
\[\lambda_{2n+1,N} (t) = \Lambda_{2n+1,N} (t)[\textbf{p}_1(t) , ... , \textbf{p}_{2n+1} (t)].\]
\end{theorem}

\begin{proof}

\begin{enumerate}

\item \textbf{The Bethe vector 
is $\neq0$ for $t$ in a neighborhood 
of $\sqrt{2}$:}

\begin{itemize}
\item \textbf{Limit of the Bethe vector 
in $\sqrt{2}$:}

Let us denote $(\vec{p}_j (t))_j$ the solution 
of the system of equations $(E_j)[t,2n+1,N]$.

Let us recall [Theorem~\ref{theorem.coordinate.ansatz}]
that for all $t$, and $\boldsymbol{\epsilon}$ 
in the canonical basis of $\Omega_N$,
\[\psi_{t,2n+1,N} (\vec{p}_1 (t) , ... , \vec{p}_{2n+1} (t)) [\boldsymbol{\epsilon}]
= \sum_{\sigma \in \Sigma_{2n+1}} C_{\sigma}(t)[\textbf{p} (t)]
\prod_{k=1}^{2n+1} e^{i\textbf{p}_{\sigma(k) (t)} . q_k [\boldsymbol{\epsilon}]}.\]
This expression admits a limit 
when $t \rightarrow \sqrt{2}$, given by:
\[\sum_{\sigma \in \Sigma_{2n+1}} C_{\sigma}(\sqrt{2})[\textbf{p}(\sqrt{2})]
\prod_{k=1}^{2n+1} e^{i\textbf{p}_{\sigma(k) (\sqrt{2})} . q_k [\boldsymbol{\epsilon}]},\]
where $(\textbf{p}_k (\sqrt{2}))_k$ 
is solution of the system of equations 
$(E_k)[\sqrt{2},2n+1,N]$.

\item \textbf{The term $\epsilon(\sigma) C_{\sigma} (\sqrt{2}) [\textbf{p}(\sqrt{2})]$ is independent from $\sigma$:}

Indeed, we have:
\begin{align*}
\prod_{1\le k < l \le 2n+1} (1
+ e^{i (\textbf{p}_{\sigma(k)} (\sqrt{2})+\textbf{p}_{\sigma(l)} (\sqrt{2}))}) 
& = \prod_{1\le \sigma^{-1}(k) < \sigma^{-1}(l) \le 2n+1} (1
+ e^{i (\textbf{p}_{k} (\sqrt{2})+\textbf{p}_{l} (\sqrt{2}))}) \\
& = \prod_{1\le k < l \le 2n+1} (1
+ e^{i (\textbf{p}_{k} (\sqrt{2})+\textbf{p}_{l} (\sqrt{2}))}) 
\end{align*}

Indeed, for all $l \neq k$, one of the 
conditions
$\sigma^{-1} (k) < \sigma^{-1} (l)$
or $\sigma^{-1} (l) < \sigma^{-1} (k)$ is verified, 
exclusively. This means that  $(1
+ e^{i (\textbf{p}_{k} (\sqrt{2})+\textbf{p}_{l} (\sqrt{2}))})$ appears exactly once
in the product for each $l,k$ such that $l \neq k$.

\item \textbf{This term is not equal to zero:}

Indeed, none of the 
$\textbf{p}_{k} (\sqrt{2}) + \textbf{p}_{l} (\sqrt{2})$ can be equal to $\pm \pi$. 
This comes from the fact that 
the system of Bethe equations 
$(E_k)[\sqrt{2},2n+1,N]$ has a unique 
simple solution given by:
\[\textbf{p}_{k} (\sqrt{2}) = 
\frac{\pi}{N}
\left( 2k - \frac{(2n+1+1)}{2}\right) = 
\frac{2\pi}{N}
\left( k - (n+1) \right).\]
These numbers are enframed by:
\[\textbf{p}_{1} (\sqrt{2})= - 2\pi \frac{n-1}{N}, \quad \textbf{p}_{n} (\sqrt{2}) = 2\pi \frac{n-1}{N}.\]
Since $n \le N/4$ these numbers  
are in $[-\pi/2,\pi/2]$, and the 
possible sums of two different 
of these numbers 
is in $]-\pi,\pi[$.

\item \textbf{The limit of Bethe vectors 
is non-zero:}
As a consequence of last points, we have that 
the limit of Bethe vectors when $t \rightarrow\sqrt{2}$ is, up to a non-zero constant (last point):
\[\sum_{\sigma \in \Sigma_{2n+1}} \epsilon(\sigma). 
\prod_{k=1}^{2n+1} e^{i\textbf{p}_{\sigma(k) (\sqrt{2})} . q_k [\boldsymbol{\epsilon}]},\]
which is the determinant of the
matrix $\left(e^{i\textbf{p}_{\sigma(k) (\sqrt{2})} . q_l [\boldsymbol{\epsilon}]}\right)_{k,l}$, which is a submatrix of 
the matrix 
$\left(e^{i s_k . s'_l}\right)_{k,l \in \llbracket 1,N\rrbracket }$. 
where $(s_k)$ is a sequence of distinct numbers
in $]-\pi/2,\pi/2[$ 
such that for all $k \le 2n+1$, 
\[s_k = \textbf{p}_{\sigma(k) (\sqrt{2})},\]
and $(s'_l)_l$ is a sequence of distinct 
integers such that for all $l \le n$, 
\[s'_l = q_l [\boldsymbol{\epsilon}].\]
If the determinant is non-zero, 
then the sum above is non zero. 
This is the case since this last matrix is obtained from the Vandermonde 
matrix $\left(e^{i s_k . l}\right)_{k,l \in \llbracket 1,N\rrbracket }$, 
whose determinant is 
\[\prod_{k<l} \left( e^{is_l} - e^{is_k}\right) \neq 0,\]
by a permutation of the columns.

\end{itemize}

\item \textbf{From the Hamiltonian to the 
transfer matrix:} 

\begin{itemize}
\item \textbf{Eigenvector of $V_N (\sqrt{2})$ and $H_N$:}
Since that limit of Bethe vector
is not equal to zero, it is an eigenvector of the matrix $V_N (\sqrt{2})$. It is also an eigenvector of the Hamiltonian $H_N$, 
for the eigenvalue 
\[2\left(\displaystyle{\sum_{k=1}^{n-1}} \cos(\frac{2\pi k}{N}) + \displaystyle{\sum_{k=N-n+1}^{N}} \cos(\frac{2\pi k}{N})\right).\]
This is a consequence of Lemma~\ref{lemma.eigenvalue.hamiltonian}, since 
for all $j$, $N\vec{p}_j (\sqrt{2}) = 2\pi(j-(n+1)):$ the eigenvalue is 
\begin{align*}
2\sum_{k=1}^{2n+1} \cos\left( \vec{p}_k (\sqrt{2})\right) & = 2 \sum_{k=1}^{n} \cos\left( \vec{p}_k (\sqrt{2})\right)+ 2 \sum_{k=n+1}^{2n+1} \cos\left( N - \vec{p}_k (\sqrt{2})\right)\\
& = 2\left(\displaystyle{\sum_{k=1}^{n-1}} \cos(\frac{2\pi k}{N}) + \displaystyle{\sum_{k=N-n+1}^{N}} \cos(\frac{2\pi k}{N})\right)
\end{align*}
\item \textbf{Comparison with the other 
eigenvalues of $H$:}
From Theorem~\ref{theorem.eigenvalues.hamiltonian}, 
we know that this 
is the largest eigenvalue of $H_N$ 
on $\Omega_N^{(2n+1)}$. Indeed, 
it is straightforward that $\psi_{\alpha}$ 
is in $\Omega_N^{(2n+1)}$ if and only if 
the number of $k$ such that $\alpha_k = 1$ 
is $2n+1$. The sum 
in the statement of Theorem~\ref{theorem.eigenvalues.hamiltonian} is maximal 
amongst these sequences when:
\[\alpha_1 = ... = \alpha_{n-1}=1 = 
\alpha_{N-n+1} = ... = \alpha_{N}\]
and the other $\alpha_k$ are equal to $0$.

\item \textbf{Identification:}

As a consequence, from Perron-Frobenius 
theorem, the limit of Bethe vectors in $\sqrt{2}$
is positive, thus this is also 
true for $t$ sufficiently 
close to $\sqrt{2}$.
From the same theorem, it 
is associated to the maximal eigenvalue 
of $V_N (t)$. As a consequence, 
the Bethe value $\Lambda_{2n+1,N} (t) [\textbf{p}_1 (t), ... \textbf{p}_{2n+1} (t)]$ is equal 
to the largest eigenvalue $\lambda_{2n+1,N} (t)$
of $V_N (t)$ on $\Omega_N^{(2n+1)}$ for these values 
of $t$. Since these two 
functions are analytic in $t$ (by the Implicit 
functions theorem on the characteristic 
polynomial, using the fact 
that the largest eigenvalue is simple), 
one can identify these two functions 
on the interval $(0,\sqrt{2})$. 
\end{itemize}

\end{enumerate}
\end{proof}

\section{\label{section.asymptotics} 
Asymptotic properties of Bethe roots}

Let us fix some $d \in [0,1/2]$, 
and $(N_k)_k$ and $(n_k)_k$ some sequences 
of integers such that for all $k$, 
$n_k \le N_k / 2 +1$ and 
$n_k / N_k \rightarrow d$.
In this section, we study the asymptotic behavior 
of the sequences $(\boldsymbol{\alpha}^{(k)}_j (t))_j$, where 
\[(\vec{p}^{(k)}_j (t))_j \equiv (\kappa_t ({\boldsymbol{\alpha}}^{(k)}_j(t)))_j\]
is solution of the system of Bethe equations 
$(E_j)[t,n_k,N_k]$, $j \le n_k$,
when $k$ tends towards $+\infty$. 
 For this purpose, we 
introduce in Section~\ref{section.counting.functions} 
the counting functions $\xi_t^{(k)}$ associated to the corresponding Bethe roots.
Roughly, these functions 'represent' the density 
of Bethe roots in the real line.
In Section~\ref{section.characterization.limit}, we prove that the sequence of functions 
$(\xi_t^{(k)})_k$ converges uniformly 
on any compact to a function 
$\boldsymbol{\xi}_{t,d}$. 
In Section~\ref{section.condensation}, 
we then prove the following, which will be used in the last Section~\ref{section.computation.entropy} in order to compute 
entropy of square ice: for all function $f: \mathbb{R} \rightarrow \mathbb{R}$ 
which is continuous and bounded, 

\[\frac{1}{N_k} \sum_{j=\lceil n_k /2 \rceil +1}^{n_k} f(\boldsymbol{\alpha}^{(k)}_j (t)) \rightarrow \int_{0}^{\boldsymbol{\xi}_{t,d}^{-1} (d)} f(\alpha) \boldsymbol{\xi}_{t,d} (\alpha) d\alpha.\]

\subsection{\label{section.counting.functions} 
The counting functions 
associated to Bethe roots}

In this section, we define the counting functions 
and prove some additional 
preliminary facts on the auxiliary 
functions $\theta_t$ and $\kappa_t$ that 
we will use in the following [Section~\ref{section.definition.counting.functions}].
We prove also that the number of Bethe roots vanishes as one get close to $+\infty$, with 
a speed that does not depend on $k$ [Section~\ref{section.rarefaction.bethe.roots}].

\subsubsection{\label{section.definition.counting.functions}Definition}

\begin{notation}
For all $t \in (0,\sqrt{2})$, 
and all integer $k$, let 
us denote $\xi^{(k)}_t: \mathbb{R}
\rightarrow \mathbb{R} $ the \textbf{counting function}
defined as follows:

\[\xi^{(k)}_t: \alpha \mapsto \left( \frac{1}{2\pi} \kappa_t (\alpha) + \frac{n_k+1}{2N_k} + \frac{1}{2\pi N_k} 
\sum_{j} \theta_t (\alpha,{\boldsymbol{\alpha}}^{(k)}_j(t))\right),\]
\end{notation}

\begin{fact}
Let us notice some properties of these functions, 
that we will use in the following: 
\begin{enumerate}
\item 
By Bethe equations, for all $j$ and $k$, 
\[\xi^{(k)}_t (\boldsymbol{\alpha}^{(k)}_j (t)) =\frac{j}{N_k} \equiv \rho_{j}^{(k)}.\] 
\item For all $k,t$, the derivative of $\xi^{(k)}_t$ 
is the function 
\[\alpha \mapsto \frac{1}{2\pi} \kappa'_t (\alpha) + \frac{1}{2\pi N_k} \sum_j \frac{\partial  \theta_t}{\partial \alpha} (\alpha,\boldsymbol{\alpha}_j (t)) > 0.\]
Indeed, this comes directly from 
the fact that $\mu_t \in \left( \frac{\pi}{2}, \pi\right)$. As a consequence, the counting 
functions are increasing.
\end{enumerate}
\end{fact}

We will use also the following: 

\begin{proposition}
\label{proposition.limits.auxiliary.functions}
We have the following limits for the 
functions $\kappa_t$ and $\theta_t$ on 
the border of their domains:
\[\displaystyle{\lim_{+\infty} \kappa_t = -\lim_{-\infty} \kappa_t} = \pi-\mu_t\] and that for all $\beta \in \mathbb{R}$, 
\[\lim_{+\infty} \theta_t (\alpha,y)  = - \lim_{-\infty} \theta_t (\alpha,y) = 
2\mu_t - \pi.\] 
\end{proposition}

\begin{proof}
Let us prove this property for $\kappa_t$, 
the limits for $\theta_t$ are obtained applying 
the same reasoning. Let us recall that for 
all $\alpha \in \mathbb{R}$,  
\[\kappa'_t (\alpha) = \frac{\sin(\mu_t)}{\cosh(\alpha)-\cos(\mu_t)}.\]
Since this function is positive, $\kappa_t$ is 
increasing, and thus admits a limit in $\pm \infty$. Since $\kappa'_t$ is integrable, these limits are finite.
Since for all $\alpha$,
\[e^{i\kappa_t (\alpha)} = \frac{e^{i\mu_t} - e^{\alpha}}{e^{i\mu_t + \alpha}-1},\]
and the limit of this expression when 
$\alpha$ tends to $+\infty$ 
is $-e^{-i\mu_t}$, then there exists some $k \in \mathbb{Z}$ such that: 
\[\displaystyle{\lim_{+\infty}} \kappa_t 
= 2k\pi + \pi-\mu_t\]
Since $\kappa_t$ is a bijective map 
from $\mathbb{R}$ to $I_t$ [Proposition~\ref{proposition.k}], then $k=0$. Thus 
we have 
\[\displaystyle{\lim_{+\infty}} \kappa_t 
=  \pi-\mu_t.\]
The limit in $-\infty$ is obtained by symmetry.
\end{proof}

\begin{notation}
For any compact interval $I \subset \mathbb{R}$, we denote 
\[\mathcal{V}_I (\epsilon,\eta) = \{z \in \mathbb{C}: |\text{Im}(z)|<\eta, d(\text{Re}(z),I)<\epsilon\}.\]
\end{notation}

\subsubsection{\label{section.rarefaction.bethe.roots}Rarefaction of Bethe roots 
near infinities}

For all $k$, $t$ and $M >0$, we denote 
\[P_t^{(k)} (M) \equiv \left\{j \in \llbracket 1 , n_k\rrbracket: \boldsymbol{\alpha}_j^{(k)} (t) \notin [-M,M]\right\}.\]

\begin{theorem}
\label{theorem.rarefaction}
For all $t \in (0, \sqrt{2})$, $\epsilon>0$, there exists some $M>0$ and $k_0$ 
such that for all $k \ge k_0$, 
\[\frac{1}{N_k} \left|P_t^{(k)}(M)\right| \le \epsilon.\]
\end{theorem}

\noindent \textbf{Idea of the proof:}
\textit{In order to prove 
this statement, we formulate it
as the equality of a number $q_t$ 
defined as a $\limsup$ (of an expression 
depending on an integer and an interval) to zero. We 
extract a sequence of integers $(\nu(k_l))_l$ 
and $(I_l)_l = \left([-M_l,M_l]\right)_l$  that realises 
this $\limsup$. For these sequences, we 
enframe the smallest (resp. greatest) 
integer such that 
the corresponding Bethe root is greater 
than $M_l$ (resp. smaller than $-M_l$), by 
a lower bound and an upper bound. Using 
Bethe equations and properties of $\kappa_t$ 
and $\theta_t$ (boundedness and monotonicity), 
we prove a inequality relating these 
two bounds. Taking the limit $ l \rightarrow + \infty$, we obtain an inequality that 
forces $q_t=0$.
}

\begin{proof}

In this proof, we assume, in order 
to simplify the computations, that 
for all $k$, $n_k$ is even, and we 
denote $n_k = 2m_k$. However, similar 
arguments are valid for any sequence 
$(n_k)_k$. Moreover, if $d =0$, the 
statement is trivial, and as a consequence, 
we assume in the remaining of the proof 
that $d>0$. It is sufficient to prove 
then that for all $\epsilon>0$, 
there exists some $M$ and $k_0$ such that 
for all $k \ge k_0$ 
\[\frac{1}{n_k} \left|P_t^{(k)}(M)\right| \le \epsilon.\]

\begin{itemize}
\item \textbf{Formulation with superior limits:} 

If $\limsup_m \boldsymbol{\alpha}_{n_k}^{(k)}$ is finite, then the Bethe roots are bounded independently from $k$ (from below 
this comes from the asymmetry of $\boldsymbol{\alpha}^{(k)}$), and thus 
the statement is verified. 

Let us thus assume that $\limsup_k \boldsymbol{\alpha}_{n_k}^{(k)} = +\infty,$
meaning that there exists some $\nu : \mathbb{N}
\rightarrow \mathbb{N}$ such that 
\[\boldsymbol{\alpha}_{n_{\nu(k)}}^{(\nu(k))}
\rightarrow +\infty \]
Let us denote for all $k,t$ and $M>0$ the 
proportion $q_t^{(k)} (M)$ of positive Bethe roots $\boldsymbol{\alpha}_j^{(k)}$ that are greater than 
$M$. Since for all $k$, $\boldsymbol{\alpha}^{(k)}$ is an antisymmetric and increasing 
sequence, $\boldsymbol{\alpha}_j^{(k)}>0$
implies that $j \ge m_k +1$, and we define
this proportion as:
\[q_t^{(k)} (M) = \frac{1}{m_k} \left| 
\left\{ j \in \llbracket m_k +1, 2m_k\rrbracket : 
\boldsymbol{\alpha}_{j}^{(k)} \ge M\right\}
\right|.\]
We also denote $q_t (M) = \limsup_k q_t^{(\nu(k))} (M)$ 
and 
\[q_t = \limsup_M q_t (M) \le 1.\] 
By 
construction, there exists an increasing sequence 
$(M_l)_l$ of real numbers and a 
sequence $(k_{l})_{l}$ of integers 
such that for 
all $\epsilon>0$, there exists some $l_0$ 
and for all $l \ge l_0$:
\[ q_t - \epsilon < q_t (M_l) - \frac{\epsilon}{2} < q_t^{(\nu(k_{l}))} (M_l) < q_t(M_l) + 
\frac{\epsilon}{2} < q_t + \epsilon.\]

The proof of the statement reduces 
to prove that $q_t = 0$. 

\item \textbf{Bounds for the 
cutting integers sequence:}

\begin{figure}[!h]
\[\begin{tikzpicture}[scale=0.6]
\fill[gray!10] (8.75,-0.5) rectangle (12,0.5); 
\draw[dashed] (8.75,-0.5) rectangle (12,0.5); 
\draw[dashed] (8,-1) -- (8,1);
\draw[dashed] (9.5,-2) -- (9.5,2);

\draw (0,0) -- (12,0); 
\foreach \x in {0,0.5,1,1.5,12, 11.5,11,10.5,5.5,8,10}
{\draw[line width =0.3mm] (\x,-0.15) -- (\x,0.15);}
\node at (0,-1) {1};
\node at (12.5,-1) {$2m_{\nu(k_l)}$};
\node at (10,-1) {j};
\node at (5.5,-1) {$m_{\nu(k_l)} +1$};
\draw[latex-latex] (0,0.75) -- (8,0.75);
\node at (4,1.25) {$\underline{a}_{l}$};
\draw[latex-latex] (0,-1.75) -- (9.5,-1.75);
\node at (5.5,-2.25) {$\overline{a}_{l}$};
\draw[-latex] (13,1.25) -- (10,1.25) -- (10,0.5);
\node at (15.5,1.25) {$\boldsymbol{\alpha}_j^{(\nu(k_{l}))} \ge M$};
\end{tikzpicture}\]
\caption{\label{figure.lower.bound.cutting}Illustration of the definition 
and lower bound of the cutting integer.}
\end{figure}
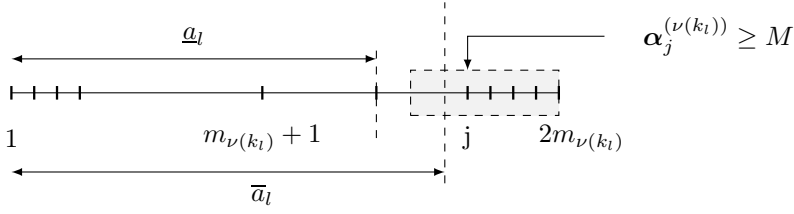

\begin{enumerate}

\item \textbf{Lower bound:}

As a consequence of the first point, 

\[(q_t +\epsilon)  m_{\nu(k_{l})} \ge 
\left|\left\{ 
j \in \llbracket   2m_{\nu(k_{l})} +1 , 2m_{\nu(k_{l})} \rrbracket: 
\boldsymbol{\alpha}_{j}^{(\nu(k_l))} \ge M_l
\right\}\right|.\]

\begin{align*}
\left|\left\{ 
j \in \llbracket  m_{\nu(k_{l})} +1 , 2m_{\nu(k_{l})} \rrbracket: 
\boldsymbol{\alpha}_{j}^{(\nu(k_{l}))} < M_l
\right\}\right| & = m_{\nu(k_{l})}  \\
& \quad - \left|\left\{ 
j \in \llbracket  m_{\nu(k_{l})}+1 , 2m_{\nu(k_{l})} \rrbracket: 
\boldsymbol{\alpha}_{j}^{(\nu(k_{l}))} \ge M_l
\right\}\right| & \\
& \ge   m_{\nu(k_{l})}. (1-q_t - \epsilon).
\end{align*}

Thus the cutting integer (which separates 
the Bethe roots according to their position 
relative to $M_l$, or equivalently 
the greatest $j$ such that the 
associated Bethe root satisfies the inequality 
$\boldsymbol{\alpha}_{j}^{(\nu(k_{l}))} < M_l$) is bounded
from below by:

\[m_{\nu(k_{l})} +  m_{\nu(k_{l})} . \max ( 0, 1-\epsilon-q_t) 
\ge  \max (0, 2 m_{\nu(k_{l})} (1-\epsilon-q_t)).\]
Since it is an integer, it is also greater than 
\[\underline{a}_{l} \equiv \max \left(0, \lfloor 2 m_{\nu(k_{l})} .(1-\epsilon
-q_t) \rfloor \right).\]

\item \textbf{Upper bound:}

Let us also denote 
$\overline{a}_{l} = \lfloor 2  m_{\nu(k_{l})} .(1+\epsilon
-q_t) \rfloor +1$.
For a similar reason, the cutting integer 
is smaller than $\overline{a}_{l}$.
See a schema on Figure~\ref{figure.lower.bound.cutting}. 

\item \textbf{Another similar bound:}

Moreover, since $l \ge l_0$, 
\[q_t ^{(\nu(k_l))} (M_{l_0}) \ge q_t ^{(\nu(k_l))} (M_{l})> q_t (M_l) - \frac{\epsilon}{2}.\]

As a consequence of a reasoning 
similar to the first point, 
\[\left|\left\{ 
j \in \llbracket  m_{\nu(k_{l})} +1 , 2 m_{\nu(k_{l})} \rrbracket: 
\boldsymbol{\alpha}_{j}^{(\nu(k_{l}))} < M_{l_0}
\right\}\right| \ge    m_{\nu(k_{l})} . (1-q_t - \epsilon),\]
and thus for all $j \le \underline{a}_l$, 
$\boldsymbol{\alpha}_j^{\nu(k_l)} < M_{l_0}$. 

\end{enumerate}

\item \textbf{Inequality involving $\underline{a}_{l}$ and $\overline{a}_{l}$ through Bethe equations:}

By summing values of the counting function, 
\begin{align*}
\sum_{k=\overline{a}_{l}}^{2 m_{\nu(k_{l})}} \xi_t^{(\nu(k_{l}))} 
(\boldsymbol{\alpha}_{k}^{(\nu(k_{l}))}) & = 
\frac{1}{2\pi} \sum_{k=\overline{a}_{l}}^{2 m_{\nu(k_{l})}} 
\kappa_t (\boldsymbol{\alpha}_{k}^{(\nu(k_{l}))}) 
+ \frac{2 m_{\nu(k_{l})}+1}{2 N_{\nu(k_{l})}} (2 m_{\nu(k_{l})} + 1 - \overline{a}_{l})
\\ & \qquad + \frac{1}{2\pi N_{\nu(k_{l})}} \sum_{k=\overline{a}_{l}}^{2 m_{\nu(k_{l})}} \sum_{k'=1}^{2 m_{\nu(k_{l})}}
\theta_t ( \boldsymbol{\alpha}_{k}^{(\nu(k_{l}))}, \boldsymbol{\alpha}_{k'}^{(\nu(k_{l}))}).
\end{align*}

By Bethe equations, 
\[\sum_{k=\overline{a}_{l}}^{2 m_{\nu(k_{l})}} \xi_t^{(\nu(k_{l}))} 
(\boldsymbol{\alpha}_{k}^{(\nu(k_{l}))}) = \frac{1}{N_{\nu(k_{l})}}
\sum_{k=\overline{a}_{l}}^{2 m_{\nu(k_{l})}} k =  \frac{ (2 m_{\nu(k_{l})}+\overline{a}_{l})(2 m_{\nu(k_{l})} - \overline{a}_{l} +1)}{2 N_{\nu(k_{l})}}.\]

As a direct consequence, and since $\theta_t$ 
is increasing in its first variable and $\theta_t (\alpha,\alpha) = 0$ for all $\alpha$,
\begin{align*}
\frac{(2 m_{\nu(k_{l})}-\overline{a}_{l} +1)(\overline{a}_{l}-1)}{2 N_{\nu(k_{l})}}
& = \frac{1}{2\pi} \sum_{k=\overline{a}_{l}}^{2  m_{\nu(k_{l})}} 
\kappa_t (\boldsymbol{\alpha}_{k}^{(\nu(k_{l}))}) 
+ \frac{1}{2\pi N_{\nu(k_{l})}} \sum_{k=\overline{a}_{l}}^{2 m_{\nu(k_{l})}} \sum_{k'=1}^{2 m_{\nu(k_{l})}}
\theta_t ( \boldsymbol{\alpha}_{k}^{(\nu(k_{l}))}, \boldsymbol{\alpha}_{k'}^{(\nu(k_{l}))})\\
& \ge \frac{1}{2\pi} \sum_{k=\overline{a}_{l}}^{2 m_{\nu(k_{l})}} 
\kappa_t (\boldsymbol{\alpha}_{k}^{(\nu(k_{l}))}) 
+ \frac{1}{2\pi N_{\nu(k_{l})}} \sum_{k=\overline{a}_{l}}^{2 m_{\nu(k_{l})}} \sum_{k'<k}
\theta_t ( \boldsymbol{\alpha}_{k}^{(\nu(k_{l}))}, \boldsymbol{\alpha}_{k'}^{(\nu(k_{l}))}).
\end{align*}
As well, using again the fact that $\theta_t$ 
is increasing in its first variable, 
we use the bound 
\[\theta_t ( \boldsymbol{\alpha}_{k}^{(\nu(k_l))}, \boldsymbol{\alpha}_{k'}^{(\nu(k_l))}) \ge \theta_t (M_l, M_{l_0})\]
when $k \ge \overline{a}_l$ and $k' \le \overline{a}_l$ (this is a consequence of 
the third bound proved in the last point).
The terms corresponding to other pairs $(k,k')$ 
are bounded by $0$. We also 
use the fact that $\kappa_t$ is increasing. 
This is written: 

\begin{align*}
\frac{(2 m_{\nu(k_{l})}-\overline{a}_{l} +1)(\overline{a}_{l}-1)}{2  N_{\nu(k_{l})}} & \ge (2 m_{\nu(k_{l})}-\overline{a}_{l} +1) \frac{1}{2\pi} \kappa_t (M_l) + (2 m_{\nu(k_{l})}-\overline{a}_{l} +1) \frac{\underline{a}_l}{2\pi  N_{\nu(k_{l})}} \theta_t (M_l, M_{l_0}).\\
\frac{(\overline{a}_{l}-1)}{2 N_{\nu(k_{l})}} & \ge  \frac{1}{2\pi} \kappa_t (M_l) + \frac{\underline{a}_l}{2\pi N_{\nu(k_{l})}} \theta_t (M_l, M_{l_0}).\\
\end{align*}

We take the limit when $l \rightarrow +\infty$, 
and obtain: 

\[\frac{d}{2} (1-\epsilon-q_t) \ge \frac{\pi-\mu_t}{2\pi} + \frac{d}{2} \frac{2\mu_t - \pi}{\pi} (1-q_t).\]

Taking the limit when $\epsilon \rightarrow 0$, 
\[\frac{d}{2} (1-q_t) \ge \frac{\pi-\mu_t}{2\pi} + \frac{d}{2} \frac{2\mu_t - \pi}{\pi} (1-q_t).\]
This inequality can be rewritten: 

\[ (1-q_t) \left( \frac{d}{2} 
- \frac{d}{2} \frac{2\mu_t -\pi}{\pi}\right) \ge \frac{\pi-\mu_t}{2\pi}\]
Finally: $1-q_t \ge \frac{1}{2d} \ge 1$, and thus $q_t = 0$.
\end{itemize}
\end{proof}

\subsection{\label{section.characterization.limit} Convergence of the sequence of counting functions $(\xi_t ^{(k)})_k$}

In this section, we prove that 
the sequence of functions $(\xi_t ^{(k)})_k$ 
converges uniformly on any compact 
to a function $\boldsymbol{\xi}_{t,d}$. 
After some recalls on complex analysis [Section~\ref{section.complex.background}], 
we prove that if a subsequence of 
this sequence of functions converge on any compact of their domain 
towards a function, then 
this function verifies a Fredholm integral equation [Section~\ref{section.limit.subsequences}], which is solved through Fourier analysis, and 
the solution is proved to be unique, in Section~\ref{section.continuous.bethe.equation}, by solving a similar equation verified by the 
derivative of this function.
We deduce in Section~\ref{section.convergence.counting.functions} that this 
fact implies that the sequence of 
counting functions 
converge to $\boldsymbol{\xi}_{t,d}$. \bigskip

For all $t$, here exists $\tau_t > 0$ such that for all $k$, 
the functions $\kappa_t$, 
$\Theta_t$ and $\xi^{(k)}_t$ 
can be extended analytically on the 
set $\mathcal{I}_{\tau_t} 
\equiv \{z \in \mathbb{C}: \left|\text{Im}(z)\right| < \tau_t \} \subset \mathbb{C}$. For the purpose of notation, the extended functions are denoted like their restriction on $\mathbb{R}$. 

\subsubsection{\label{section.complex.background} Some complex analysis background}

Let us recall some results of complex analysis 
that we will use in the following of this section.
Let $U$ be an open subset of $\mathbb{C}$. 

\begin{definition}
We say that a sequence $(f_m)_m$ of functions
$U \rightarrow \mathbb{C}$ is \textbf{locally bounded} when for all $z \in U$, 
the sequence $(|f_m (z)|)_m$ is bounded.
\end{definition}

\begin{theorem}[Montel]
Let $(f_m)_m$ be a locally bounded sequence of holomorphic functions 
$U \rightarrow \mathbb{C}$. There 
exists a subsequence of $(f_m)_m$ 
which converges uniformly on any compact subset
of $U$.  
\end{theorem}

\begin{lemma}
\label{lemma.convergence}
Let $(f_m)_m$ be a locally bounded sequence of continuous functions 
$U \rightarrow \mathbb{C}$ and $f : U \rightarrow \mathbb{C}$ such that any
subsequence of $(f_m)_m$ which converges uniformly on any 
compact subset of $U$ towards some function, 
then this limit is $f$. Then $(f_m)_m$ converges uniformly 
on any compact towards $f$. 
\end{lemma}

\begin{proof}
Let us assume that $(f_m)_m$ does not converge towards $f$. Then there 
exists some $\epsilon>0$, compact $K \subset U$ and a non-decreasing function $\nu : \mathbb{N} \rightarrow \mathbb{N}$
such that for all $m$, 
\[||\left(f_{\nu(m)}- f\right)_K||_{\infty} \ge \epsilon.\]
From Montel theorem, one can extract a subsequence of $(f_{\nu(m)})_m$ 
which converges towards $f$ uniformly on any compact of $U$, and 
in particular on the compact $K$. This is in contradiction to the 
above inequality, and we deduce that $(f_m)_m$ converges towards $f$. 
\end{proof}

\begin{theorem}[Cauchy formula]
Let us assume that $U$ is simply connected and let $f: U \rightarrow \mathbb{C}$ 
be a holomorphic function and $\gamma$ a lace included in $U$ that is homeomorphic 
to a circle positively oriented. Then for all $z$ which in the interior domain of 
the lace, 
\[f(z) = \frac{1}{2\pi i} \oint_{\gamma} \frac{f(s)}{s-z}ds.\]
\end{theorem}

\noindent Let us also recall a sufficient condition for a holomorphic function 
to be biholomorphic: 

\begin{theorem}
\label{theorem.condition.holomorphism}
Let $f : U \rightarrow \mathbb{C}$ be a holomorphic function onto an open 
and simply connected set $U$. Let $V \subset U$ 
and $\gamma$ a lace included in $U$ that is homeomorphic 
to a circle positively oriented, and such that $V$ is included 
in the interior domain of $\gamma$. We assume that:  

\begin{enumerate}
\item for all $z \in V$ 
and $s \in \gamma$, $f(z) \neq f(s)$, 
\item and for all $z \in V$, $f'(z) \neq 0$. 
\end{enumerate}

Then $f$ is a \textbf{biholomorphism} from $V$ onto its image, 
meaning that there exists some holomorphic function $g : f(V) \rightarrow U$ 
such that for all $z \in f(V)$, $f(g(z)) = z$ 
and for all $z \in U$, $g(f(z))= z$.  
Moreover, for all $z \in f(V)$, 
\[g(z) = \frac{1}{2\pi i} \int_{\gamma} 
s \frac{f'(s)}{f(s)-z}ds.\]
\end{theorem}

\subsubsection{\label{section.limit.subsequences} The limits of subsequences of $(\xi_t^{(k)})_k$ satisfy a Fredholm 
integral equation}

In this section, we prove the following: 

\begin{theorem}
\label{theorem.convergence.subsequences}
Let $\nu : \mathbb{N} \rightarrow \N$ 
be a non-decreasing function, and assume that 
$\left( \xi_t^{(\nu(k))}\right)_m$ converges uniformly on any 
compact of $\mathcal{I}_{\tau_t}$ towards a function $\xi_t$. Then 
this function satisfies the following equation for all $\alpha \in \mathcal{I}_{\tau}$:
\[\xi'_t (\alpha) = \frac{1}{2\pi} \kappa'_t (\alpha) + 
\int_{\mathbb{R}} \frac{\partial \theta_t}{\partial \alpha} (\alpha,\beta) \xi'_t  (\beta) d\beta.\]
Moreover, $\xi_t (0) = d/2$.
\end{theorem}

\begin{proof}

\begin{itemize}

\item \textbf{Convergence of the derivative 
of the counting functions:}

Since any compact of $\mathcal{I}_{\tau_t}$ can be included in the interior domain 
of a rectangle lace, through derivation of Cauchy formula,
the derivative of $\xi_t^{(\nu(k))}$ converges also uniformly on 
any compact, towards $\xi'_t$. Since 
$\left|\left(\xi_t^{(k)}\right)'\right|$ is bounded by a constant 
that does not depend on $m$, and that 
$s \mapsto |\theta_t (\alpha,s)|$ 
is integrable on $\mathbb{R}$ for all $\alpha$, 
then $s \mapsto \theta_t (\alpha,s) \xi'_t (s) ds$ is integrable on 
$\mathbb{R}$.

\item \textbf{Some notations:}

Let us fix some $\epsilon>0$, and $\alpha_0 \in \mathbb{R}$.
In the following, we consider some 
\textit{irrational} number (and 
as a consequence not the image 
of a Bethe root)  $M>1$ such that: 

\begin{enumerate}
\item $M \in \overset{\circ}{\xi_t(\mathbb{R})}$
\item such that: $|P_t^{(k)} (M)| \le \frac{\epsilon}{2(2\mu_t-\pi)}$ for all $k$ greater 
than some $k_0$ (in virtue of Theorem~\ref{theorem.rarefaction}),
\item  
and $\alpha_0 \in \xi_t^{-1} ([-M,M])$.
\end{enumerate}

Since $\xi_t (\mathbb{R})$ 
is an interval (this function is increasing on $\mathbb{R}$), one can take $M$ arbitrarily 
close to the supremum of this interval. 
When $M$ tends towards this supremum, 
$\xi_t^{-1} (M)$ tends 
to $+\infty$: if it did not, then this would contradict the fact that this is the supremum (again by monotonicity).
One can assume that $M$ is such that 

\[\frac{1}{2\pi}\left|\int_{(\xi_t^{-1}([-M,M]))^c} \theta_t (\alpha,\beta) \xi'_t (\beta) d\beta \right| \le
\frac{\epsilon}{4}.\]

Let us also denote $J_{t} = \xi_t^{-1} ([-M,M])$.

\item \textbf{The derivative of $\xi_t$ 
relative to the axis $i\mathbb{R}$ is non-zero when close enough to $\mathbb{R}$:}

Indeed, for all $\alpha,\lambda \in \mathbb{R}$, 
\[\xi_t^{(k)} (\alpha+i \lambda) = \frac{1}{2\pi} \kappa_t (\alpha+i\lambda) 
+ \frac{n_k+1}{2N_k} + \frac{1}{2\pi N_k} \sum_j \theta_t (\alpha+i\lambda, 
\boldsymbol{\alpha}_j^{(k)} (t)).\]
As a direct consequence the derivative of the function $\lambda \mapsto -i \xi_t^{(k)} (\alpha+i \lambda)$ in $0$ is:
\[\frac{1}{2\pi} \kappa'_t (\alpha) + \frac{1}{2 \pi N_k } \sum_j 
\theta_t (\alpha,\boldsymbol{\alpha}_j^{(k)} (t)) = (\xi_t^{(k)})' (\alpha) \ge \frac{1}{2\pi} \kappa'_t (\alpha) > 0.\]
Thus for all $\alpha$, the 
derivative of the function 
$\lambda \mapsto - i \xi_t (\alpha+i\lambda)$ in $0$ is greater than 
\[\frac{1}{2\pi} \kappa'_t (\alpha).\]

Moreover, since the second 
derivative of $\lambda \mapsto -i \xi_t^{(k)} (\alpha+i \lambda)$ is 
a bounded function of $\alpha$, with 
a bound that is independant from $k$, through 
Taylor integral formula, there exists a constant $p_t >0$ such that 
for all $\lambda \in \mathbb{R}$ and $\alpha \in \mathbb{R}$:

\[|\xi_t(\alpha+i\lambda) - i \xi'_t (\alpha) . \lambda
- \xi_t (\alpha)| \le p_t \lambda^2,\]
which implies:
\[\left|\text{Im}\left(\xi_t(\alpha+i\lambda)\right) - \xi'_t (\alpha) . \lambda
\right|\le p_t \lambda^2.\]

\[\text{Im}\left(\xi_t(\alpha+i\lambda)\right) \ge  \xi'_t (\alpha) . \lambda
- p_t \lambda^2.\]

\item \textbf{The restriction of 
$\xi_t$ on some $\mathcal{V}_{J_{t}} (\eta_t,\epsilon_t)$ is a biholomorphism onto 
its image:}

Since $M$ is defined so that 
\[M \in \overset{\circ}{\xi_t(\mathbb{R})},\] 
then $J_{t}$ is compact. This means, as 
a consequence of last point, that there exists some positive number  $\sigma_t < \tau_t$ such that 
for all $z \in  \mathcal{V}_{J_{t}}(\sigma_t,1) \backslash \mathbb{R}$, then 
$\xi_t (z) \notin \mathbb{R}$. 

Let us consider the lace $\gamma_{t} = \partial \mathcal{V}_{J_{t}} (\sigma_{t},1)$ (see 
an illustration on Figure~\ref{figure.proof.biholomorphism}). 

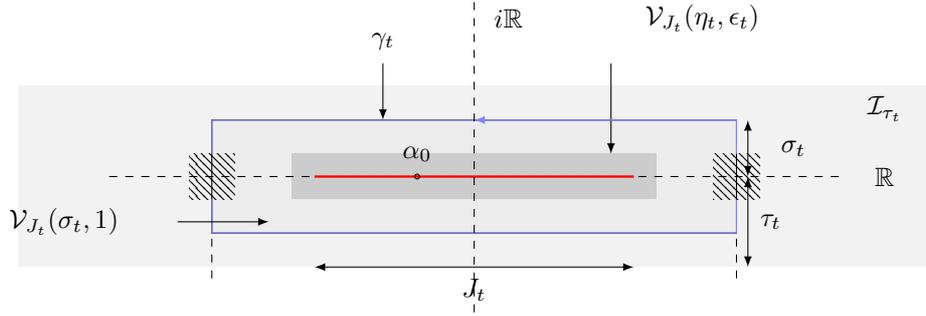
\begin{figure}[!h]

\[\begin{tikzpicture}[scale=0.6]
\fill[gray!10] (-5,-0.5) rectangle (15,3.5); 
\fill[gray!15] (-0.75,0.25) rectangle (10.75,2.75);
\fill[gray!40] (1,1) rectangle(9,2);

\draw[dashed] (-3,1.5) -- (13,1.5);
\node at (14,1.5) {$\mathbb{R}$};

\draw (-0.75,0.25) rectangle (10.75,2.75);

\draw[color=blue!50] (-0.75,1) -- (-0.75,0.25) -- (10.75,0.25) -- (10.75,1);

\draw[-latex] (8,4) -- (8,2);
\node at (10,5) {$\mathcal{V}_{J_{t}} (\eta_t,\epsilon_t)$};

\node at (3,4.5) {$\gamma_{t}$};
\draw[-latex] (3,4) -- (3,2.75);

\fill[pattern=north west lines] (-1.25,1) rectangle (-0.25,2);
\fill[pattern=north west lines] (11.5-1.25,1) rectangle (11.5-0.25,2);

\draw[line width = 0.3mm,color=red] (1.5,1.5) -- (8.5,1.5);

\draw[-latex] (-1.5,0.5) -- (0.5,0.5);
\node at (-4,0.5) {$\mathcal{V}_{J_{t}} (\sigma_t,1)$};
\node at (5,-1) {$J_{t}$};
\draw[latex-latex] (11,1.5) -- (11,2.75);
\draw[latex-latex] (11,1.5) -- (11,-0.5);
\node at (12,2.125) {$\sigma_t$};

\node at (11.5,0.5) {$\tau_t$};

\draw[-latex,color=blue!50] (10.75,2.75) -- (5,2.75);
\draw[color=blue!50] (10.75,2) -- (10.75,2.75) -- (-0.75,2.75) -- (-0.75,2);

\draw[latex-latex] (1.5,-0.5) -- (8.5,-0.5);

\draw (3.75,1.5) circle (1.5pt);
\node at (3.75,2) {$\alpha_0$};

\draw[dashed] (-0.75,-0.75) -- (-0.75,0.25);
\draw[dashed] (10.75,-0.75) -- (10.75,0.25);
\draw[dashed] (5,-1.5) -- (5,5.5);
\node at (5.75,5) {$i\mathbb{R}$}; 
\node at (14,3) {$\mathcal{I}_{\tau_t}$};
\end{tikzpicture}\]
\caption{\label{figure.proof.biholomorphism} Illustration of the proof 
that $\xi_t$ is a biholomorphism on a neighborhood 
of $J_{t}$.}
\end{figure}

Let us prove that there exist some $\epsilon_t >0$ and $\eta_t >0$ such that 
the values taken by 
the function $\xi_t$ on $\mathcal{V}_{J_{t}} (\eta_t,\epsilon_t)$ are distinct from 
any value taken by the same function on the lace $\gamma_{t}$. This is done 
in two steps, as follows:

\begin{enumerate}

\item First, we consider open neighbourhoods (illustrated by 
dashed squares on Figure~\ref{figure.proof.biholomorphism}) for the two points of 
$\gamma_{t} \cap \mathbb{R}$ such that the values taken by $\xi_t$
on these sets are distant by more than a positive constant 
from the values taken on $J_{t}$. 
This is possible since $\xi_t$ is strictly increasing on $\mathbb{R}$. 

\item On the part of $\gamma_{t}$ that is not included in these two open sets, 
the function $\xi_t$ takes non-real values, and the set of values taken is 
compact, by continuity. As a consequence, the set of values taken on the lace $\gamma_{t}$ 
is included into a compact that does not intersect the set of values taken 
on $J_{t}$. Thus one can separate these two sets of values 
with open sets, meaning that there exist some $\epsilon_t >0$ 
and $\eta_t >0$ such that the set of 
values taken by $\xi_t$ on $\mathcal{V}_{J_{t}} (\eta_t ,\epsilon_t)$ 
does not intersect the set of values taken 
by this function on $\gamma_t$.
\end{enumerate}

In virtue of Theorem~\ref{theorem.condition.holomorphism}, this means that $\xi_t$ is 
a biholomorphism from $\mathcal{V}_{J_{t}} (\eta_t,\epsilon_t)$ 
onto its image on this set. As a consequence, it is also an open function, 
and its image on $\mathcal{V}_{J_{t}} (\eta_t,\epsilon_t)$ 
contains the image of $J_{t}$, which 
$[-M,M]$ by definition.

\item \textbf{Asymptotic biholomorphism property 
for $\xi_t^{(\nu(k))}$:}

It derives from the last point that there 
exists some $k_1 \ge k_0 $ such that for all $k \ge k_1$, 
the values of $\xi_t^{(\nu(k))}$ on $\gamma_t$ are distinct from the values of $\xi_t^{(k)}$ on 
$\mathcal{V}_{J_t} (\eta_t,\epsilon_t)$, 
and as a consequence, for the same reason
as the last point, $\xi_t^{(\nu(k))}$ is 
a biholomorphism from $\mathcal{V}_{J_t} (\eta_t,\epsilon_t)$ onto its image on this set. 
Moreover, since $\xi_t^{(\nu(k))}$ converges 
uniformly to $\xi_t$ on $\overline{\mathcal{V}_{J_t} (\eta_t,\epsilon_t)}$, it converges 
also uniformly on $\mathcal{V}_{J_t} (\eta_t,\epsilon_t)$, and $\xi_t (\mathcal{V}_{J_t} (\eta_t,\epsilon_t)) $ contains $\mathcal{V}_{[-M,M]}(\eta'_t,\epsilon'_t)$, then there exists some $\eta'_t, \epsilon'_t >0$ and some $k_2 \ge k_1$ such that 
for all $k \ge k_2$, $\xi_t^{(\nu(k))} (\mathcal{V}_{J_t} (\eta_t,\epsilon_t))$ contains 
$\mathcal{V}_{[-M,M]}(\eta'_t,\epsilon'_t)$.

\item \textbf{Lace integral expression of the counting functions and approximation 
of $\xi_t^{(\nu(k))}$:}

We deduce that for all $k \ge k_2$, and $\sigma
< \eta'_t$, 
positive such that the lace:
\begin{align*}
\Gamma_{t}^{\sigma} & \equiv 
\left\{-M,M \right\} \times 
\llbracket -\sigma, \sigma \rrbracket 
\bigcup \left[ -M,
M\right]
\times \{ -\sigma , \sigma \} 
\end{align*} 

is included into $\xi_t^{(\nu(k))} (\mathcal{V}_{J_{t}} (\eta_t,\epsilon_t))$.
See Figure~\ref{figure.definition.lace} for an illustration.

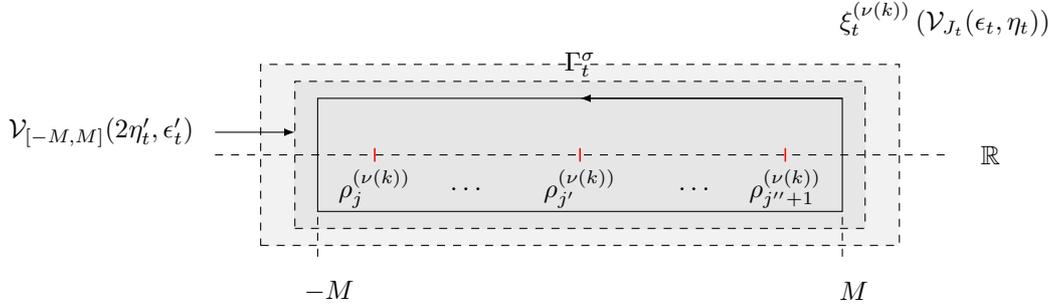
\begin{figure}[!h]
\[\begin{tikzpicture}[scale=0.6]
\fill[gray!10] (-2,-0.5) rectangle (12,3.5); 
\draw[dashed] (-2,-0.5) rectangle (12,3.5); 
\fill[gray!20] (-1.25,-0.125) rectangle 
(11.25,3.125);
\draw[dashed] (-1.25,-0.125) rectangle 
(11.25,3.125);
\node at (13,4.5) {$\xi_t^{(\nu(k))} \left( \mathcal{V}_{J_{t}} (\epsilon_t,\eta_t)\right)$};
\draw[dashed] (-3,1.5) -- (13,1.5);
\node at (14,1.5) {$\mathbb{R}$};
\draw (-0.75,0.25) rectangle (10.75,2.75);
\node at (5,3.5) {$\Gamma_t^{\sigma}$};
\draw[-latex] (10.75,2.75) -- (5,2.75);
\node at (-0.5,-1.5) {$-M$};
\node at (11,-1.5) {$M$};
\draw[dashed] (-0.75,-0.75) -- (-0.75,0.25);
\draw[dashed] (10.75,-0.75) -- (10.75,0.25);
\draw[color=red, line width =0.2mm] (0.5,1.35) -- (0.5,1.65);
\node at (0.5,0.75) {$\rho_j^{(\nu(k))}$};
\draw[color=red, line width =0.2mm] (9.5,1.35) -- (9.5,1.65);
\draw[color=red, line width =0.2mm] (5,1.35) -- (5,1.65);
\draw[-latex] (-3,2) -- (-1.25,2);
\node at (-5.5,2) {$\mathcal{V}_{[-M,M]} (2\eta'_t,\epsilon'_t)$};
\node at (9.5,0.75) {$\rho_{j''+1}^{(\nu(k))}$};
\node at (2.5,0.75) {$\hdots$};
\node at (5,0.75) {$\rho_{j'}^{(\nu(k))}$};
\node at (7.5,0.75) {$\hdots$};
\end{tikzpicture}\]
\caption{\label{figure.definition.lace} 
Illustration for the definition of 
the lace $\Gamma^{\sigma}_t$.}
\end{figure}

We then have, since $\alpha_0 \in J_t$ the following equation for all $k$, $t,\sigma$:
\[\frac{1}{2\pi N_{\nu(k)} } \sum_{j \in P_t^{(\nu(k))}(M)} \theta_t (\alpha_0, 
\boldsymbol{\alpha}_j^{(\nu(k))} (t)) = \frac{1}{2\pi} \bigointssss_{\Gamma_t^{\sigma}} \theta_t\left(\alpha_0,\left(\xi_t^{(\nu(k))}\right)^{-1} (s)\right) \frac{e^{2i\pi N_{\nu(k)} s}}{\left(e^{2 i\pi N_{\nu(k)} s}-1\right)}  ds \]

Indeed, there are no poles 
for $\xi_t^{(\nu(k))}$ on $\Gamma_t^{\sigma}$ since 
$M$ is irrational. The poles of the function 
inside the domain delimited by $\Gamma_t^{\sigma}$ are exactly the 
numbers $\rho_j^{(\nu(k))}$. 
By the residues theorem, and since for all $j$, 
$\xi_t^{(\nu(k))} (\boldsymbol{\alpha}_j (t))
= \rho_j^{(\nu(k))}$:

\begin{align*}\bigointssss_{\Gamma_t^{\sigma}} \theta_t\left(\alpha_0,\left(\xi_t^{(\nu(k))}\right)^{-1} (s)\right) \frac{e^{2 i\pi N_{\nu(k)} s}}{\left(e^{2 i\pi N_{\nu(k)} s}-1\right)}  ds & = 2\pi i \sum_{j \in P_t^{(\nu(k))}(M)} \frac{1}{2i\pi N_{\nu(k)}} \theta_t (\alpha_0,\boldsymbol{\alpha}_j^{(\nu(k))} (t))\\
\end{align*}

\item \textbf{Approximations:}

We deduce that for all $k \ge k_2$ and all $\sigma< \eta'_t$, 

\[\left|\xi_t^{(\nu(k))}(\alpha_0) - 
\frac{1}{2\pi} \kappa_t (\alpha_0) - \frac{n_{\nu(k)}+1}{2N_{\nu(k)}} - \frac{1}{2\pi} \bigointssss_{\Gamma_t^{\sigma}} \theta_t\left(\alpha_0,\left(\xi_t^{(\nu(k))}\right)^{-1} (s)\right) \frac{e^{2 i\pi N_{\nu(k)} s}}{\left(e^{2 i\pi N_{\nu(k)} s}-1\right)} ds\right|\]
is smaller than 
\[\sum_{j \notin P_t^{(\nu(k))}(M)} \left|\theta_t (\alpha_0,
\boldsymbol{\alpha}_j^{(\nu(k))} (t))\right|
\le (2\mu_t -\pi) \left| \left\{ j
\in \llbracket 1, n_{\nu(k)} \rrbracket : \boldsymbol{\alpha}_k^{(\nu(k))} (t) \notin [-M,M]
\right\}\right| \le \frac{\epsilon}{2},\]
by notations of the second point of this proof.
Let us also note $k_3 \ge k_2$ some integer 
such that for all $m \ge k_3$, 
\[\left|\frac{n_{\nu(k)}+ 1}{2N_{\nu(k)}} - \frac{d}{2} \right|\le \frac{\epsilon}{8}.\]

We then evaluate convergence of various 
terms: 

\begin{enumerate}

\item \textbf{Convergence of the bottom part 
of the lace integral to an integral 
on a real segment when $\sigma \rightarrow 0$:} 

By continuity of $\xi_t^{-1}$, 
there exists some $\sigma_0>0$ 
such that for all $k \ge k_3$, $\sigma \le \sigma_0$, 
\[\left|\int_{[-M,M]} \theta_t (\alpha_0, \xi_t^{-1} (\beta - i \sigma)) d \beta
- \int_{[-M,M]} \theta_t(\alpha_0,\xi_t^{-1} (\beta)) d \beta \right| \le 
\frac{\epsilon}{16}.\]
By change of variable in the second integral: 
\[\left|\int_{[-M,M]} \theta_t (\alpha_0,\xi_t^{-1} (\beta - i \sigma)) d \beta 
- \int_{\xi_t^{-1}([-M,M])} \theta_t(\alpha_0,\beta) \xi'_t (\beta) d \beta \right| \le 
\frac{\epsilon}{16}.\]

\item \textbf{Bounding the lateral parts 
of the lace integral for $\sigma \rightarrow 0$:} 

There exists some $\sigma_1 > 0$ 
such that $\sigma_1 \le \sigma_0$ such that 
for all $\sigma \le \sigma_1$,
$k \ge k_3$, 
\[\left|\frac{1}{2\pi} \int_{-\sigma}^{\sigma}
 \theta_t(\alpha_0,\left(\xi_t^{(\nu(k))}\right)^{-1} (\pm M+i\lambda))) \frac{e^{2i\pi N_{\nu(k)} (\pm M+i\lambda) }}{\left(e^{2i\pi N_{\nu(k)} (\pm M+i\lambda)}-1\right)}  d \lambda \right| \le \frac{\epsilon}{64}.\]

\item \textbf{Convergence of the top and bottom parts of 
the lace integral when $k \rightarrow +\infty$:} 

Then there exists some $k_4 \ge k_3$ 
such that for all $k \ge k_4$, 

\[\left|\frac{1}{2\pi} \int_{-M}^{M}
 \theta_t(\alpha_0,\left(\xi_t^{(\nu(k))}\right)^{-1} (\beta + i \sigma_1 )) \frac{e^{2i\pi N_{\nu(k)}(\beta + i \sigma_1) }}{\left(e^{2i\pi N_{\nu(k)} (\pm (\beta + i \sigma_1)}-1\right)}  d \beta \right| \le \frac{\epsilon}{64}\]
 
 \[\left|\frac{1}{2\pi} \int_{-M}^{M}
 \theta_t(\alpha_0,\left(\xi_t^{(\nu(k))}\right)^{-1} (\beta - i \sigma_1 )) \left( \frac{e^{2i\pi N_{\nu(k)} (\beta - i \sigma_1) }}{\left(e^{2i\pi N_{\nu(k)} (\pm (\beta - i \sigma_1)}-1\right)} - 1 \right)  d\beta \right|\le \frac{\epsilon}{64}.\]
 \end{enumerate}
 
 All these inequalities together with 

\[\frac{1}{2\pi} \left| \int_{(\xi_t^{-1} (-[M,M]))^c}  \theta_t (\alpha_0,\beta) \xi'_t (\beta) d\beta \right| \le \frac{\epsilon}{4}\]

 imply, by multiple applications of the triangular 
 inequality,
 that for all $k \ge k_4$, 
 
 \[\left| \xi_t^{(\nu(k))} (\alpha_0) - \frac{1}{2\pi} \kappa_t (\alpha_0) - \frac{d}{2} -
 \frac{1}{2\pi} \int_{-\infty}^{\infty} 
 \theta_t (\alpha_0, \beta) \xi'_t(\beta) d\beta \right| \le 
 \frac{\epsilon}{2} + \frac{\epsilon}{8} 
 + 2 \frac{\epsilon}{16} + 3\frac{\epsilon}{64}
 = \epsilon.\]
 
 \item \textbf{Integral equations:}
 
 As a consequence, since this is true 
 for all $\epsilon>0$
 we have the following equality for all $\alpha \in \mathbb{R}$: 
 
 \[\xi_t (\alpha_0) =  \frac{1}{2\pi} \kappa_t (\alpha) + \frac{d}{2} +
 \frac{1}{2\pi} \int_{-\infty}^{\infty} 
 \theta_t (\alpha_0, \beta) \xi'_t(\beta) d\beta.\]

Moreover, this equality is verified for any $\alpha$,
and differentiating it 
relatively to $\alpha$:

 \[\xi'_t (\alpha) =  \frac{1}{2\pi} \kappa'_t (\alpha) + 
 \frac{1}{2\pi} \int_{-\infty}^{\infty} 
 \frac{\partial \theta_t}{\partial \alpha} (\alpha, \beta) \xi'_t(\beta) d\beta.\]

\item \textbf{Value of $\xi_t (0)$:}

Since $\xi_t^{(k)}$ is increasing for all $k$, 
we have directly: 

\[\frac{\lfloor n_k /2 \rfloor}{N_{k}} = \xi_t^{(k)} (\boldsymbol{\alpha}_{\lfloor n_k /2 \rfloor}^{(k)} (t)) \le \xi_t^{(k)} (0) \le \xi_t^{(k)} (\boldsymbol{\alpha}_{\lceil n_k /2 \rceil+1}^{(k)} (t))= \frac{\lceil n_k /2 \rceil+2}{N_k}.\]

As a consequence $\xi_t (0) = d/2$.

\end{itemize}
\end{proof}

\subsubsection{\label{section.continuous.bethe.equation} Solution of the Fredholm equation}

In this section, we prove that 
the integral equation on $\xi_t$ 
in the statement of Theorem~\ref{theorem.convergence.subsequences} 
is unique and compute its solution:

\begin{proposition} \label{proposition.continuous.bethe.equation}
 Let $t \in (0,\sqrt{2})$ and 
 $\rho$ a continuous 
function in $L^1(\mathbb{R}, \mathbb{R})$ 
such that for all $\alpha \in \mathbb{R}$, 
\[\rho(\alpha) = \frac{1}{2\pi} \kappa'_t (\alpha) + \frac{1}{2\pi} \int_{-\infty}^{+\infty} 
\frac{\partial \theta_t}{\partial \alpha} 
(\alpha,\beta) \rho(\beta) d\beta.\]
Then for all $\alpha$, 
\[\rho(\alpha) =\frac{1}{4\mu_t \cosh\left( \pi \alpha / 2\mu_t\right)}.\]

\end{proposition}

\begin{proof}

The proof consists essentially in the application 
of Fourier transform techniques. We will denote, for convenience, for all $\alpha$ and $\mu$, 
\[\Xi_{\mu} (\alpha) = \frac{\sin(\mu)}{\cosh(\alpha)-\cos(\mu)}.\]

\begin{itemize}
\item \textbf{Application of Fourier transform:}

Let us denote $\hat{\rho}$ the Fourier transform 
of $\rho$: for all $\omega$,
\[\hat{\rho} (\omega) = \int_{-\infty}^{+\infty}
\rho(\alpha) e^{i\omega \alpha} d \alpha,\]
which exists since $\rho$ is $L^1(\mathbb{R})$.
As well, denote ${\hat{\Xi}}_{\mu}$ the Fourier transform
of $\Xi_{\mu}$.
Thus, since 
\[\int_{-\infty}^{+\infty} \frac{\partial \theta_t}{\partial \alpha} (\alpha,\beta) \rho(\beta) d \beta = -\int_{-\infty}^{+\infty} \Xi_{\mu} (\alpha-\beta) \rho(\beta) d \beta ,\] 
this defines a convolution product, which 
is transformed in a simple product 
through the Fourier transform, so that for all $\omega$: \[\hat{\rho } (\omega)
= \frac{1}{2\pi} \hat{\Xi}_{\mu_t} (\omega) - \frac{1}{2\pi} \hat{\Xi}_{2\mu_t} (\omega) \hat{\rho} (\omega).\]
\[2\pi \hat{\rho} (\omega) = 
\frac{\hat{\Xi}_{\mu_t} (\omega)}{1+ \frac{1}{2\pi} \hat{\Xi}_{2\mu_t} (\omega)}\]

\item \textbf{Computation of $\hat{\Xi}_{\mu}$:}

\begin{itemize}

\item \textbf{Singularities of this function:}
The singularities of the 
function $\Xi_{\mu}$ 
are exactly 
the numbers $i(\mu+2k\pi)$ 
for $k \ge 0$ and $i(-\mu+2k\pi)$ 
for $k \ge 1$, 
since for $\alpha \in \mathbb{C}$, 
$\cosh(\alpha)=\cos(\mu)$ 
if and only if 
\[\cos(i\alpha) = \cos(\mu),\]
and this implies that 
$\alpha = i(\pm \mu + 2k\pi)$ for some $k$.

\item \textbf{Computation of the residues:}

For all $k$, the residue of 
$\Xi_{\mu}$ in $i(\mu+2k\pi)$ 
is 

\[\text{Res}(\Xi_{\mu},i\mu+2k\pi)= \frac{e^{i\gamma.i(\mu+2k\pi)}}{i} = \frac{1}{i}
e^{-\gamma (\mu+2k\pi)}.\]
As well, 
\[\text{Res}(\Xi_{\mu},-i\mu+2k\pi)= \frac{e^{i\gamma.i(-\mu+2k\pi)}}{i} = -\frac{1}{i}
e^{-\gamma (-\mu+2k\pi)}.\]

We have, for all $\gamma$:
\[\int_{-\infty}^{+\infty} \Xi_{\mu} (\alpha) e^{i\alpha\gamma} d\alpha = 2\pi \frac{\sinh[(\pi-\mu)\gamma]}{\sinh(\pi\gamma)}\]

\item \textbf{Residue theorem:}

Let us denote, for all integer $n$, 
the lace $\Gamma_n = [-n,n] + i[0,n]$. 
The residues $\Xi_{\mu}$ inside the domain 
delimited by this lace are the $i(\mu + 
2k\pi)$ with $k \ge 0$, and the $i(-\mu + 
2k\pi)$ with $k\ge 1$. For all $n$,

\[\int_{\Gamma_n} \Xi_{\mu} (\alpha) e^{i\alpha\gamma} d\alpha = \int_{\Gamma_n} \frac{\sinh(i\mu)}{i( \cosh(\alpha)-\cosh(i\mu))} e^{i\alpha\gamma} d\alpha
\]

By the residue theorem, 
\[\int_{\Gamma_N}  \Xi_{\mu} (\alpha) e^{i\alpha\gamma} d\alpha = 2\pi i \left(\sum_{k \ge 0} \text{Res}(\Xi_{\mu},i(\mu+2k\pi))- \sum_{k\ge 1} \text{Res}(\Xi_{\mu},i(-\mu+2k\pi))\right).\]

\item \textbf{Asymptotic behavior:}

Since only the contribution 
on $[-n,n]$ of the integral 
is non zero asymptotically, 
and by convergence of the integral 
and the sums, 

\begin{align*} \int_{-\infty}^{+\infty} \Xi_{\mu} (\alpha) e^{i\alpha\gamma} d\alpha 
& = 2\pi e^{-\gamma \mu}+ 2\pi \sum_{k=1}^{+\infty} (-e^{\gamma \mu} + e^{-\gamma \mu}) e^{-2 \gamma k \pi} \\
& = 
2\pi e^{-\gamma \mu}+ 2\pi (-e^{\gamma \mu} + e^{-\gamma \mu})\left(\frac{1}{1-e^{-2\gamma\pi}}-1\right) \\
& = 
2\pi e^{-\gamma \mu} + 2\pi (-e^{\gamma \mu}
+ e^{-\gamma \mu}) \frac{e^{-\gamma \pi}}{e^{\gamma \pi} - e^{-\gamma \pi}} \\
& = 
2\pi \frac{e^{-\gamma (-\pi+\mu)} 
- e^{\gamma(-\pi-\mu)} 
- e^{\gamma (\mu-\pi)} + e^{\gamma (-\pi-\mu)}}{e^{\gamma \pi} - e^{-\gamma \pi}} \\
& = 2\pi \frac{\sinh(\gamma(\pi-\mu))}{\sinh(\gamma \pi)}.
\end{align*}
\end{itemize}

\item \textbf{Computation of $\hat{\rho}$:}

Using this expression of the Fourier transform of 
$\Xi_{\mu}$, for all $\omega$, 
\begin{align*}
2\pi \hat{\rho}(\omega) & = \frac{2\pi\sinh(\omega(\pi-\mu_t))}{\sinh(\pi\omega)+\sinh(\omega(\pi-2\mu_t))}\\
& = \frac{4\pi\sinh(\omega(\pi-\mu_t))}{e^{\omega\pi}.(1+ 
e^{-2\mu_t \omega}) - e^{-\omega\pi}.(1+
e^{2\mu_t \omega})}\\
& = \frac{4\pi\sinh(\omega(\pi-\mu_t))}{e^{\omega(\pi-\mu_t)}.(e^{\mu_t \omega}+ 
e^{-\mu_t \omega}) - e^{-\omega(\pi-\mu_t)}.(e^{-\mu_t \omega}+
e^{\mu_t \omega})}\\
& = \frac{\pi}{\cosh(\mu_t \omega)}.
\end{align*}

\item \textbf{Inverse transform:}

We thus have for all $\alpha$:

\[2\pi \rho(\alpha) = \frac{1}{2\pi} \int_{-\infty}^{\infty}
\frac{\pi}{\cosh(\mu_t \omega)} e^{-i\omega \alpha} d\omega = \frac{1}{\mu_t}\int_{-\infty}^{\infty}
\frac{1}{2\cosh(u)} e^{-i\frac{u}{\mu_t} \alpha} du,\]
where we used the variable change $u=\mu_t \omega$.
Using the computation 
of the Fourier transform of $\Xi_{\mu}$
for $\mu=\pi/2$, 
\[\int_{-\infty}^{+\infty} \frac{1}{\cosh(\alpha)} e^{i\alpha\gamma} d\alpha = 
2\pi \frac{\sinh(\pi\gamma/2)}{\sinh(\pi\gamma)} = \frac{\pi}{\cosh(\pi\gamma/2)}.\]
Thus we have 
\[2\pi \rho(\alpha) = \frac{1}{2\mu_t} \frac{\pi}{\cosh(\pi\alpha /2\mu_t)} = \frac{\pi}{2\mu_t} \frac{1}{\cosh(\pi\alpha /2\mu_t)}\]
\end{itemize}
\end{proof}

\subsubsection{\label{section.convergence.counting.functions}  Convergence of $\xi_t^{(k)}$:}

\begin{theorem}
There exists a function $\boldsymbol{\xi}_{t,d} : \mathbb{R} \rightarrow \mathbb{R}$ 
such that $\xi_t^{(k)}$ converges uniformly 
on any compact towards $\boldsymbol{\xi}_{t,d}$. 
Moreover, this function satisfies the following 
equation for all $\alpha$: 

\[\boldsymbol{\xi}'_{t,d} (\alpha) = \frac{1}{2\pi} \kappa'_t (\alpha) + \frac{1}{2\pi} \int_{\mathbb{R}} \frac{\partial \theta_t}{\partial \alpha} (\alpha,\beta) 
\boldsymbol{\xi}'_{t,d} (\beta) d\beta,\]
and $\boldsymbol{\xi}_{t,d} (0) = d/2$.

\end{theorem}

\begin{proof}
Consider any subsequence of $(\xi_t^{(k)})_k$ 
which converges uniformly on any compact of $\mathcal{I}_{\tau_t}$ to 
a function $\xi_t$. Via Cauchy formula, 
the derivative of $(\xi_t^{(k)})$ 
converges uniformly on any compact to 
$\xi'_t$. Since the functions $\xi_t^{(k)}$ 
are uniformly bounded by a constant which is 
independant from $k$, and that for all $k$, $(\xi_t^{(k)})$, $\xi'_t$ is positive and $\xi_t$ is bounded, and thus $\xi_t$ is in $L^1(\mathbb{R},\mathbb{R})$. 
From Theorem~\ref{theorem.convergence.subsequences}, we get that 
$\xi'_t$ verifies a Fredholm equation, which 
has a unique solution in $L^1 (\mathbb{R},\mathbb{R})$ [Proposition~\ref{proposition.continuous.bethe.equation}]. 
From Theorem~\ref{theorem.convergence.subsequences}, $\xi_t$, as 
a function on $\mathbb{R}$, is the unique 
primitive function of this one which has value 
$d/4$ on $0$. Since this function is analytic, 
it determines its values on the whole stripe $\mathcal{I}_{\tau_t}$. In virtue of Lemma~\ref{lemma.convergence}, 
$(\xi_t^{(k)})_k$ converge towards this function.
\end{proof}

\begin{proposition}
\label{limits.xi.limit}
The limit of the function $\boldsymbol{\xi}_{t,d}$
in $+\infty$
is $\frac{d}{2}+\frac{1}{4}$, and 
the limit in $-\infty$ is $d/2-1/4$.
\end{proposition}

\begin{proof}
For all $\alpha$, 
\[\boldsymbol{\xi}_{t,d} (\alpha) = \frac{d}{2} 
+ \frac{1}{4 \mu_t} \int_{0}^{\alpha} 
\frac{1}{\cosh(\pi x /2\mu_t)} dx 
= \frac{d}{2} + \frac{1}{2\pi} \int_{0}^{2\mu_t \alpha / \pi} \frac{1}{\cosh(x)} dx.
\]
This converges in $+\infty$ to: 
\[\frac{d}{2} + \frac{1}{\pi} \int_{0}^{+\infty}
\frac{e^x}{e^{2x}+1} dx = \frac{d}{2} + \frac{1}{\pi} \int_{0}^{+\infty}
(\arctan (\exp))'(x) dx = \frac{d}{2} + 
\frac{1}{2} - \frac{1}{\pi} \frac{\pi}{4} = \frac{d}{2} + \frac{1}{4}.\]
For the same reason, the limit in 
$-\infty$ is $d/2-1/4$.
\end{proof}

\begin{remark}
As a consequence, this limit is $>d$ 
when $d<1/2$ and equal to $d$ when $d=1/2$.
\end{remark}

\subsection{\label{section.condensation} Condensation of Bethe roots relative 
to some functions}

In this section, we prove that if 
$f$ is a continuous function  
$(0,+\infty) \rightarrow (0,+\infty)$, decreasing
and integrable, then the scaled sum 
of the values of $f$ on the Bethe roots 
converges to an integral involving $f$ 
and $\boldsymbol{\xi}_{t,d}$ [Theorem~\ref{theorem.convergence.relative}].
Let us denote, for all $t,m$ and $M>0$: 
\[Q_t^{(k)} (M) \equiv \left\{j \in \llbracket 1 , n_k\rrbracket: \boldsymbol{\xi}_{t,d}^{-1} \left(\frac{j}{N_k}\right) \notin [-M,M]\right\},\]
and for two fine sets $S,T$, we denote $S \Delta T = S \backslash T \cup T \backslash S$.
For a compact set $K \subset \mathbb{R}$, 
we denote its diameter $\delta (K) 
\equiv \max_{x,y \in K} |x-y|$. For $I$ 
a bounded interval of $\mathbb{R}$, 
we denote $\mathit{l} (I)$ its length.
When 
\[J= \displaystyle{\bigcup_j I_j}\] 
with $I_k$ bounded and disjoint intervals, the 
length of $J$ is 
\[\mathit{l} (J) 
= \sum_j \mathit{l} (I_j).\]

\begin{theorem}
\label{theorem.convergence.relative}
Let $f: (0,+\infty) \rightarrow (0,+\infty)$ 
a continuous, decreasing and integrable function.
Then:
\[\frac{1}{N_k} \sum_{j=\lceil n_k /2 \rceil+1}^{n_k} f(\boldsymbol{\alpha}^{(k)}_j (t)) 
\rightarrow \int_{0}^{\boldsymbol{\xi}_{t,d}^{-1} (d)} f(\alpha) \boldsymbol{\xi}'_{t,d} (\alpha) d\alpha,\]
where we denote $\boldsymbol{\xi}_{t,1/2}^{-1}(1/2) = +\infty$.
\end{theorem}

\begin{remark}
This is another version of a statement proved 
in~\cite{kozlowski} for bounded continuous and 
Lipshitz functions, which is not sufficient for 
the proof of Theorem~\ref{theorem.main}.
\end{remark}

\begin{proof}
In all the proof, the indexes $j$ in 
the sums are in $\llbracket \lceil n_k /2 \rceil+1,n_k\rrbracket$.
\begin{itemize}
\item \textbf{Setting:}
Let $\epsilon>0$ and $t \in (0,\sqrt{2})$. Let us fix some $M$ 
such that for all $k$ greater than some $k_0$: 
\[\frac{1}{N_k} |P_t^{(k)}(M)|\le \frac{\epsilon}{2||f_{[M,+\infty)}||_{\infty}+1},\] 
\[\left|\int_{[M,+\infty)} f(\alpha) \boldsymbol{\xi}'_{t,d} (\alpha) d\alpha
\right| \le \frac{\epsilon}{2}.\]
and if $d<1/2$,
\[M>\boldsymbol{\xi}_{t,d}^{-1}(d),\]
which is possible in virtue of Proposition~\ref{limits.xi.limit}.

\item \textbf{Using the rarefication of Bethe roots:}

\begin{align*}
\frac{1}{N_k} \sum_{j= \lceil n_k /2 \rceil+1}^{n_k} f(\boldsymbol{\alpha}^{(m)}_j (t)) & = \frac{1}{N_k} \sum_{j= \lceil n_k /2 \rceil+1}^{n_k} f\left(\left(\xi_t^{(k)}\right)^{-1}\left(\frac{j}{N_k}\right)\right)\\
& = \frac{1}{N_k} \sum_{j \notin P_t^{(k)}(M)} f\left(\left(\xi_t^{(k)}\right)^{-1}\left(\frac{j}{N_k}\right)\right) \\
& \quad + \frac{1}{N_k} \sum_{j \in P_t^{(k)}(M)} f\left(\left(\xi_t^{(k)}\right)^{-1}\left(\frac{j}{N_k}\right)\right)
\end{align*}

As a consequence of the first point
\begin{align*}
\left| \frac{1}{N_k} \sum_{j=\lceil n_k /2 \rceil+1}^{n_k} f(\boldsymbol{\alpha}^{(k)}_j) - \frac{1}{N_k} \sum_{j \notin P_t^{(k)}(M)} f\left(\left(\xi_t^{(k)}\right)^{-1}\left(\frac{j}{N_k}\right)\right)\right| & \le \frac{1}{N_k}\left|P_t^{(k)}(M)\right|.||f_{[M,+\infty)}||_{\infty} \\
& \le  \frac{\epsilon}{2||f_{[M,+\infty)}||_{\infty}+1} ||f_{[M,+\infty)}||_{\infty}\\
&  \le \frac{\epsilon}{2},
\end{align*}
since by definition, if $j \in P_t^{(k)}(M)$ 
and $j \ge \lceil n_k /2 \rceil+1$, then 
\[\left(\xi_t^{(k)}\right)^{-1} \left( \frac{j}{N_k}\right) 
\ge M.\]

\item \textbf{On the asymptotic 
cardinality of $(P_t^{(k)}(M))^c \Delta (Q_t^{(k)}(M))^c$:}

\[\frac{1}{N_k} |(P_t^{(k)}(M))^c \Delta (Q_t^{(k)}(M))^c| \rightarrow 0.\]

Indeed, $(P_t^{(k)}(M))^c \Delta (Q_t^{(k)}(M))^c$ is equal to the set
\[\left\{ j \in \llbracket 1,n_k \rrbracket : \frac{j}{N_k} \in  \left(\xi_t ^{(k)} ([-M,M])\right) \Delta \left(\boldsymbol{\xi}_{t,d} ([-M,M])\right)\right\},\]
thus its cardinality is smaller than 
\[\delta \left( N_k \left( \left(\xi_t ^{(k)} ([-M,M])\right) \Delta \left(\boldsymbol{\xi}_{t,d} ([-M,M])\right)\right) \right) +1,\] 
which is equal to 
\[N_k \delta  \left( \left(\xi_t ^{(k)} ([-M,M])\right) \Delta \left(\boldsymbol{\xi}_{t,d} ([-M,M])\right) \right)+1.\]
As a consequence: 

\[\frac{1}{N_k} \left| (P_t^{(k)}(M))^c \Delta (Q_t^{(k)}(M))^c\right| \le \delta  \left( \left(\xi_t ^{(k)} ([-M,M])\right) \Delta \left(\boldsymbol{\xi}_{t,d} ([-M,M])\right) \right) + 
\frac{1}{N_k}.\]

Since $\xi_t^{(k)}$ converges to $\boldsymbol{\xi}_{t,d}$ on any compact, and in particular $[-M,M]$, 
the diameter on the right of this inequality 
converges to $0$ when $k$ tends towards $+\infty$.

\item \textbf{Replacing $P_t^{(k)}(M)$ 
by $Q_t^{(k)}(M)$ in the sum:} 

Since $f$ is decreasing and positive, for all $j \in \llbracket \lceil n_k /2 \rceil+1,n_k \rrbracket$, 
\[\frac{1}{N_k} \left| f\left(\left(\xi_t^{(k)}\right)^{-1}\left(\frac{j}{N_k}\right)\right) \right| \le 
\int_{\left[\frac{j-1}{N_k},\frac{j}{N_k}\right]} 
f\left(\left(\xi_t^{(k)}\right)^{-1} (x) \right) dx.\]

As a consequence, the difference 

\[\frac{1}{N_k} \left| \sum_{j \notin P_t^{(k)}(M)} f\left(\left(\xi_t^{(k)}\right)^{-1}\left(\frac{j}{N_k}\right)\right) - \sum_{j \notin Q_t^{(k)}(M)} f\left(\left(\xi_t^{(k)}\right)^{-1}\left(\frac{j}{N_k}\right)\right) \right|\]
is smaller than 
\[\int_{J_k} \left|f\left( \left(\xi_t^{(k)}\right)^{-1}(x)\right)\right| dx = 
\int_{\xi_t^{(k)} (I_k)} f(x) 
\left(\xi_t^{(k)}\right)' (x) dx,\]
where $J_k$ is the union of the intervals
\[ \left[\frac{j-1}{N_k},
\frac{j}{N_k}\right],\]
for $j \in (P_t^{(k)}(M))^c \Delta (Q_t^{(k)}(M))^c$. Since the functions $\xi_t^{(k)}$ 
are uniformly bounded independently of $k$, there 
exists a constant $C_t >0$ such that for 
all $k$: 

\[\int_{\xi_t^{(k)} (J_k)} f(x) 
\left(\xi_t^{(k)}\right)' (x) dx \le C_t \int_{\xi_t^{(k)} (J_k)} f(x) dx,
\]

Since $f$ is 
decreasing,
\[\int_{\xi_t^{(k)} (J_k)} f(x) 
\le \int_{[0,\mathit{l}(\xi_t^{(k)} (J_k))]} f(x),\]
From the fact that $\xi_t^{(k)}$ is increasing:
\[\mathit{l} (\xi_t^{(k)} (J_k)) 
= \int_{J_k} (\xi_t^{(k)})'(\alpha) d\alpha\]
Since the 
derivative of $\xi_t^{(k)}$ is bounded uniformly 
and independently of $k$, and that 
the length of $J_k$ is smaller than 
$\frac{1}{N_k} \left| (P_t^{(k)}(M))^c \Delta (Q_t^{(k)}(M))^c\right|$, 
\[\mathit{l}(\xi_t^{(k)} (J_k)) \rightarrow 0.\]
From the integrability of $f$ on $(0,+\infty)$:
\[\int_{[0,\mathit{l}(\xi_t^{(k)} (J_k))]} f(x) 
\rightarrow 0.\]
As a consequence, there exists exists some $k_1 \ge k_0$ 
such that for all $k \ge k_1$, 

\[\frac{1}{N_k} \left| \sum_{j \notin P_t^{(k)}(M)} f\left(\left(\xi_t^{(k)}\right)^{-1}\left(\frac{j}{N_k}\right)\right) - \sum_{j \notin Q_t^{(k)}(M)} f\left(\left(\xi_t^{(k)}\right)^{-1}\left(\frac{j}{N_k}\right)\right) \right| \le \frac{\epsilon}{4}.\]

\item \textbf{Approximating 
$\xi_t^{(k)}$ by $\boldsymbol{\xi}_{t,d}$ 
in the sum:}

\begin{enumerate}

\item \textbf{Bounding the contribution 
in a neighborhood of $0$:}

With an argument similir to the one used in 
the last point (bounding with integrals), 
there exists $\sigma>0$ smaller than $M$
such that for all $k$,

\[\frac{1}{N_k} \sum_{j \in (Q_t^{(k)} (\sigma))^c \cap (Q_t^{(k)}(M))^c} \left| f\left(\left(\xi_t^{(k)}\right)^{-1}\left(\frac{j}{N_k}\right)\right) \right| \le \frac{\epsilon}{8}.\]

\item \textbf{Using the convergence of $\xi_t^{(k)}$ on a compact away from $0$:}

There exists some $k_2 \ge k_1$ such that 
for all $k \ge k_2$:

\[\frac{1}{N_k} \sum_{j \in Q_t^{(k)} (\sigma) \cap (Q_t^{(k)}(M))^c} \left| f\left(\left(\xi_t^{(k)}\right)^{-1}\left(\frac{j}{N_k}\right)\right) - f\left(\boldsymbol{\xi}_{t,d}^{-1}\left(\frac{j}{N_k}\right)\right) \right| \le \frac{\epsilon}{16}.\]

Indeed, for all the integers $j$ in the 
sum, $\boldsymbol{\xi}_{t,d}^{-1}\left(\frac{j}{N_k}\right)
\in [\sigma,M]$, and by uniform 
convergence of $\left(\xi_t^{(k)}\right)^{-1}$ on 
the compact $\boldsymbol{\xi}_{t,d} ([\sigma,M])$, for 
$k$ great enough, the real numbers $\left(\xi_t^{(k)}\right)^{-1} \left(\frac{j}{N_k}\right)$ and 
$\boldsymbol{\xi}_{t,d}^{-1} \left(\frac{j}{N_k}\right)$
for these integers $k$ and these 
indexes $j$ all lie in the same 
compact interval. Since $f$ 
is continuous, there exists some $\eta>0$ such that 
whenever for $x,y$ lie in this compact interval and
$|x-y| \le \eta$, then $|f(x)-f(y)|\le \frac{\epsilon}{8}$.
Since $\left(\xi_t^{(k)}\right)^{-1}$ 
converges uniformly towards $\boldsymbol{\xi}_{t,d}^{-1}$ on 
the compact $\boldsymbol{\xi}_{t,d} ([\sigma,M])$, 
there exists some $k_3 \ge k_2$ such 
that for all $k \ge k_3$, and for all $j$ 
such that $j \in 
Q_t^{(k)} (\sigma)$ and $j \notin Q_t^{(k)} (M)$,
\[\left| \left(\xi_t^{(k)}\right)^{-1} \left( \frac{j}{N_k}\right) - \boldsymbol{\xi}_{t,d}^{-1} \left( \frac{j}{N_k}\right)\right| \le \eta.\]
As a consequence, we obtain the announced inequality.

\end{enumerate}

\item \textbf{Convergence of the remaining sum:}

The following sum is a Riemman sum:  
\[\sum_{j \notin Q_t^{(k)}(M)} f\left(\boldsymbol{\xi}_{t,d}^{-1}\left(\frac{j}{4m}\right)\right),\]
and if $d=1/2$ it converges towards 
\[\int_{\boldsymbol{\xi}_{t,d} (0)}^{\boldsymbol{\xi}_{t,d}(M)} f \left(\boldsymbol{\xi}_{t,d}^{-1} (\alpha)\right) d \alpha= \int_{0}^M 
f (\alpha) \boldsymbol{\xi}'_{t,d} (\alpha) d\alpha,\]
by a change of variable. If $d<1/2$, 
it converges towards 
\[\int_{\boldsymbol{\xi}_{t,d} (0)}^{d} f \left(\boldsymbol{\xi}_{t,d}^{-1} (\alpha)\right) d \alpha= \int_{0}^{d} 
f (\alpha) \boldsymbol{\xi}'_{t,d} (\alpha) d\alpha,\]
since in this case, $M$ was chosen such that $M>\boldsymbol{\xi}_{t,d}^{-1} (d)$.

As a consequence, there exists some $k_4 \ge k_3$ 
such that for all $ k \ge k_4$, if $d=1/2$:

\[\left| \sum_{j \notin Q_t^{(k)}(M)} f\left(\boldsymbol{\xi}_{t,d}^{-1}\left(\frac{j}{N_k}\right)\right) - \int_{0}^M 
f (\alpha) \boldsymbol{\xi}'_{t,d} (\alpha) d\alpha\right| \le \frac{\epsilon}{16} \]

If $d<1/2$:

\[\left| \sum_{j \notin Q_t^{(k)}(M)} f\left(\boldsymbol{\xi}_{t,d}^{-1}\left(\frac{j}{N_k}\right)\right) - \int_{0}^d 
f (\alpha) \boldsymbol{\xi}'_{t,d} (\alpha) d\alpha\right| \le \frac{\epsilon}{16} \]

\item \textbf{Assembling the inequalities:}

All put together, we have for all $k \ge k_4$, 

\[\left| \frac{1}{N_k} \sum_{j=\lceil n_k /2 \rceil+1}^{n_k} 
f(\boldsymbol{\alpha}_j^{(k)} (t))- \int_{0}^{\boldsymbol{\xi}^{-1}_t (d)} 
f (\alpha) \boldsymbol{\xi}'_{t,d} (\alpha) d\alpha\right| \le \frac{\epsilon}{2} + \frac{\epsilon}{4} + \frac{\epsilon}{8} +\frac{\epsilon}{16} +\frac{\epsilon}{16}  = \epsilon. \]

Since for all $\epsilon>0$ there exists 
such an integer $k_4$, this proves the statement.

\end{itemize}
\end{proof}

\section{\label{section.computation.entropy} Computation of square ice entropy}

In this last section, 
we compute the entropy of 
the square ice model.

\begin{notation}
For all $d \in [0,1/2]$, we denote:
\[F(d) = - 2 \int_{0}^{\boldsymbol{\xi}_{t,d}^{-1} (d)} 
\log_2 \left( 2 |\sin(\kappa_t (\alpha)/2)|\right) \rho_t (\alpha)d\alpha.\]
\end{notation}

\begin{lemma}
\label{lemma.F}
Let us consider $(N_k)$ some sequence 
of integers, and $(n_k)$ another 
sequence such that for all $k$, $n_k \le (N_k-1) /4 $, and $(2n_k+1)/ N_k \rightarrow d \in [0,1/2]$.
Then 
\[\log_2 (\lambda_{2n_k+1,N_k} (1)) \rightarrow F(d). \]
\end{lemma}

\begin{proof}

In this  
proof, for all $k$ we denote: 
\[(\vec{p}_j^{(k)})_j = 
(\kappa_t (\boldsymbol{\alpha}_j^{(k)}))_j,\]
the solution of the system of Bethe 
equations $(E_k)[1,2n_k+1,N_k]$.
The for all $k$:

\[\lambda_{2n_k+1,N_k} (1) = \Lambda_{2n_k+1,N_k} [\vec{p}^{(k)}] = \left( 2 + (N_k-1) 
+ \sum_{j \neq (n_k+1)} 
\frac{\partial \Theta_1}{\partial x} \left(0,\vec{p}_j^{(k)}\right)\right)\prod_{j=1}^{n_k} M_1(e^{i\vec{p}_j^{(k)}}),\]
since by antisymmetry of $\vec{p}^{(k)}$, 
and that $2n_k +1$ is odd, $\vec{p}^{(k)}_{n_k+1}=0$.

For all $z$ such that $|z|=1$, 
\[M_1(z) = \frac{z}{z-1}.\]
By antisymmetry of the sequences $\vec{p}^{(k)}$, 
for all $k$:
\[\displaystyle{\prod_{j=1}^{n_k} e^{i\vec{p}_j^{(k)}} = \prod_{j=1}^{n_k} e^{i\vec{p}_j^{(k)}/2}= 1}.\] As a consequence: 
\begin{align*}
\Lambda_{2n_k+1,N_k} [\vec{p}^{(k)}] & = 
\left( 2 + (N_k-1) 
+ \sum_{j \neq (n_k+1)} 
\frac{\partial \Theta_1}{\partial x} \left(0,\vec{p}_j^{(k)}\right)\right) \prod_{j=1}^{n_k} \frac{1}{e^{i\vec{p}_{j}^{(k)}}-1} \\ 
& = \left( 2 + (N_k-1) 
+ \sum_{j \neq (n_k+1)} 
\frac{\partial \Theta_1}{\partial x} \left(0,\vec{p}_j^{(k)}\right) \right) \prod_{j=1}^{n_k} \frac{e^{-i\vec{p}_{j}^{(k)}/2}}{e^{i\vec{p}_{j}^{(k)}/2}-e^{-i\vec{p}_{j}^{(k)}/2}}\\
\end{align*}

Since this eigenvalue is positive, 
\[\Lambda_{2n_k+1,N_k} [\vec{p}^{(k)}] = |\Lambda_{2n_k+1,N_k} [\vec{p}^{(k)}]|
= \left| 2 + (N_k-1) 
+ \sum_{j \neq (n_k+1)} 
\frac{\partial \Theta_1}{\partial x} \left(0,\vec{p}_j^{(k)}\right) \right|\prod_{j=1}^{n_k} \frac{1}{2\left|\sin(\vec{p}_{j}^{(k)}/2)\right|}.\]
As a consequence, since $\partial \Theta_1 / \partial x$ is a bounded function, 
\begin{align*}
\lim_k \log_2 (\lambda_{2n_k+1,N_k} (1)) & = - \lim_k \left(\frac{1}{N_k} \sum_{j=1}^{n_k} \log_2 \left(2\left|\sin(\vec{p}_{j}^{(k)}/2)\right|\right) + O\left( \frac{\log_2 (N_k)}{N_k}\right)\right) \\
& =  - 2\lim_k \frac{1}{N_k} \sum_{j=\lceil n_k /2\rceil+1}^{n_k} \log_2 \left(2\left|\sin(\kappa_t(\boldsymbol{\alpha}_{k}^{(k)})/2)\right|\right)\\
& = - 2\int_{0}^{\boldsymbol{\xi}^{-1}_{t,\vec{d}} (\vec{d})} 
\log_2 (2|\sin(\kappa_t (\alpha)/2)|)\rho_t(\alpha)d\alpha\\
& = F(\vec{d}).
\end{align*}
where $\rho_t = \boldsymbol{\xi}'_{t,d}$, 
and we used the antisymmetry of the Bethe 
roots vectors in the second equality.
For the other equalities, they are a consequence of Theorem~\ref{theorem.convergence.relative}, 
since the function defined as $\alpha \mapsto 
-\log_2 (2|\sin(\kappa_t (\alpha)/2)|)$ 
on $(0,+\infty)$ is 
continuous, integrable, decreasing and positive: 

\begin{enumerate}
\item \textbf{Positive:} 
For all $\alpha >0$, $\kappa_t (\alpha)$ 
is in 
\[(0,\pi-\mu_t) = \left(0, \frac{\pi}{3} \right).\]
As a consequence, $2\sin(\kappa_t (\alpha)/2)$ 
is in $(0,1)$, and this implies that for all $\alpha >0$, 
\[-\log_2 (2|\sin(\kappa_t (\alpha)/2)|) > 0.\]
\item \textbf{Decreasing:}

This comes from the fact that $-\log_2$ is 
decreasing, and $\kappa_t$ is increasing, 
and the sinus is increasing on $(0,\pi/6)$.

\item \textbf{Integrable:}

Since $\kappa_t (0) = 0$ and $\kappa'_t(0) >0$,
for $\alpha$ positive  sufficiently close to $0$  
$2\sin(\kappa_t (\alpha)/2) \le 
2\kappa'_t (0) \alpha$. As a consequence, 
\[-\log_2 (2|\sin(\kappa_t (\alpha)/2)|) \le 
- \log_2 ( 2 \kappa'_t (0) \alpha).\]
Since the logarithm is integrable on any bounded neighborhood of $0$, the function 
$\alpha \mapsto -\log_2 (2|\sin(\kappa_t (\alpha)/2)|) $ is integrable.

\end{enumerate}

The other limit is obtained by antisymmetry of $\kappa_t$. 
\end{proof}

\begin{reptheorem}{theorem.main}
The entropy of square ice is 
\[h(X^s) = \frac{3}{2} \log_2 \left( \frac{4}{3} \right).\]
\end{reptheorem}

\begin{remark}
This value corresponds 
to $\log_2 (W)$ in~\cite{Lieb67}.
\end{remark}

\begin{proof}
Here we fix $t=1 \in (0,\sqrt{2})$.
As a consequence, $\mu_t = 2\pi/3$.

\begin{itemize}

\item \textbf{Entropy of $X^s$ and asymptotics of 
the maximal eigenvalue:} 

Let us recall that the entropy of 
$X^s$ is given by:
\[h(X^s) = \lim_N \frac{1}{N} \max_{n \le (N-1)/4} \log_2(\lambda_{2n+1,N}(1)).\]

For all $N$, we denote $\nu(N)$ the smallest 
$\le (N-1)/4$ such that for all $n \le (N-1) /4$, 
\[\lambda_{2\nu(N)+1,N}(1) \ge \lambda_{2n+1,N}(1).\] 
By compacity, there exists an increasing sequence $(N_k)$ such that $(2\nu (N_k)+1)/N_k$ converges
towards some non-negative real 
number $\textbf{d}$. 
Since for all $k$, $\nu(N_k) \le (N_k-1) /4$, 
then $\textbf{d} \le 1/2$.
In virtue of Lemma~\ref{lemma.F}, 
$h(X^s) = F(\vec{d})$.

\item \textbf{Comparison with the asymptotics 
of other eigenvalues:} 

Moreover, if $d$ is another number 
$d \in [0,1/2]$, there exists 
$\nu' : \mathbb{N} \rightarrow \mathbb{N}$ such 
that
\[(2\nu'(N)+1) / N \rightarrow d.\]
For all $k$, 
\[\lambda_{2\nu'(N_k)+1,N_k} (1) \le \lambda_{2\nu(N_k)+1,N_k} (1).\]

Also in virtue of Lemma~\ref{lemma.F}, 
$h(X^s) \ge F(\vec{d})$, and thus 
\[F(\vec{d}) = \max_{d \in [0,1/2]} F(d).\]

This maximum is realized only for $d=1/2$. 
As a consequence $\vec{d}=1/2$.

\item \textbf{Rewritings:}

As a consequence, 
\[h(X^s) = - 2 \int_{0}^{+\infty} \log_2 (2|\sin(\kappa_t (\alpha) /2 |) \rho_t (\alpha) d\alpha\]

Let us rewrite this expression of $h(X^s)$ 
using
\[|\sin(x/2)| = \sqrt{\frac{1-\cos(x)}{2}}.\]

This leads to:
\[h(X^s) = - \frac{\log_2(2)}{2} \int_{-\infty}^{+\infty} \rho_t (\alpha)d\alpha - \frac{1}{2}
\int_{-\infty}^{+\infty} \log_2 (1-\cos(\kappa_t (\alpha))).\rho_t (\alpha)d\alpha.\]
Thus,
\[h(X^s) = -\frac{1}{2}\int_{-\infty}^{+\infty} \log_2 (2-2\cos(\kappa_t(\alpha))).
\rho_t (\alpha)d\alpha.\]

Let us recall that for all $\alpha$, 
\[\rho(\alpha) = \frac{1}{4\mu_t \cosh(\pi\alpha/2\mu_t)} = \frac{3}{8 \pi \cosh(3\alpha/4)} \]
\[\cos(\kappa_t (\alpha))= \frac{\sin^2(\mu_t)}{\cosh(\alpha)-\cos(\mu_t)} 
- \cos(\mu_t) = \frac{3}{4(\cosh(\alpha)+1/2)} + \frac{1}{2},\]
We have that: 
\[h(X) = -\frac{3}{16\pi} \int_{-\infty}^{+\infty} \log_2 \left(1-\frac{3}{2\cosh(\alpha)+1}\right)
\frac{1}{\cosh(3\alpha/4)}d\alpha.\]
Using the variable change $e^{\alpha} = x^4$, $d\alpha .x = 4dx$,

\[h(X) = -\frac{3}{16\pi} \int_{0}^{+\infty} \log_2 \left(1-\frac{3}{x^4+1/x^4+1}\right)
\frac{2}{(x^3 + 1/x^3)}\frac{4}{x} dx.\]

By symmetry of the integrand:

\[h(X) = -\frac{3}{4\pi} \int_{-\infty}^{+\infty} \frac{x^2dx}{x^6+1} 
\log_2 \left( \frac{(2x^4 - 1-x^8) }{1+x^4+x^8}\right)dx\]

\[h(X) = -\frac{3}{4\pi} \int_{-\infty}^{+\infty} \frac{x^2dx}{x^6+1} 
\log_2 \left( \frac{(x^2-1)^2 (x^2+1)^2}{1+x^4+x^8}\right)dx\]

\item \textbf{Application of 
the residues theorem:}

In the following, we use the standard determination of 
the logarithm on $\mathbb{C} \backslash \mathbb{R}_{-}$.

We apply the residue 
theorem to obtain (the poles of 
the integrand are $e^{i\pi/6}, 
e^{i\pi/2}, e^{i5\pi/6})$:
\[\int_{-\infty}^{+\infty} \frac{x^2\log_2 (x+i)}{x^6+1} dx=
2\pi i \left(\sum_{k=1,3,5} \frac{e^{ik\pi/3} \log_2 (e^{ik\pi/6}+i)}{6 e^{i5k\pi/6}}\right).\]

\[\int_{-\infty}^{+\infty} \frac{x^2\log_2 (x-i)}{x^6+1} dx = -
2\pi i \left(\sum_{k=7,9,11} \frac{e^{ik\pi/3} \log_2 (e^{ik\pi/6}-i)}{6 e^{i5k\pi/6}}\right)\]

By summing these two equations, 
we obtain that 
$\int_{-\infty}^{+\infty} \frac{x^2 \log_2 (x^2+1)}{x^6+1} dx$ 
is equal to:

\[\begin{array}{c} 
\frac{\pi}{3} \left[\log_2 (e^{i\pi/6}+i) 
- \log_2 (e^{i\pi/2}+i) + \log_2 (e^{i5\pi/6}+i) \right.\\
\left.
+\log_2 (e^{i7\pi/6} -i) - \log_2 (e^{i9\pi/6} -i)+\log_2 (e^{i11\pi/6} -i)\right]\end{array}\]
This is equal to 
\[\frac{\pi}{3} (\log_2 (|e^{i\pi/6}+i|^2))-\log_2 (|e^{i\pi/2}+i|)
+ \log_2 (|e^{i5\pi/6}+i|^2) = \frac{2\pi}{3}\log_2\left(\frac{3}{2}\right).\]

\item \textbf{Other computations:}

We do not include the following
computation, since it is very similar to 
the previous one: 

\[\int_{-\infty}^{+\infty} \frac{x^2}{x^6+1} \log_2 (1+x^4x^8)dx
= \frac{2\pi}{3} \log_2 \left( \frac{8}{3}\right).\]

For the last integral, 
we write $\log_2 ((x^2-1)^2) = 2 \text{Re}(\log_2 (x-1)+\log_2 (x+1))$
and obtain:
\begin{align*}
\int_{-\infty}^{+\infty} \frac{x^2}{x^6+1} 
\log_2 ((x^2-1)^2) & = \text{Re}\left( \int_{-\infty}^{+\infty} \frac{x^2}{x^6+1} 
\log_2 (x-1) + \int_{-\infty}^{+\infty} \frac{x^2}{x^6+1} 
\log_2 (x+1)\right) \\
& = \frac{2\pi}{3} \log_2 \left(\frac{1}{2}\right)\end{align*}

\item \textbf{Summing these integrals:}

As a consequence 
\[h(X^s) =-\frac{3}{4\pi} \frac{2\pi}{3}
\left(\log_2 \left(\frac{1}{2}\right)+ 2 \log_2 \left(\frac{3}{2}\right)- \log_2 \left(\frac{8}{3}\right)\right) =\frac{1}{2}\log_2 \left(\frac{4^3}{3^3}\right)= \frac{3}{2} \log_2 \left( \frac{4}{3} \right).\]

\end{itemize} 

\end{proof}

\section{\label{section.comments} Comments}

\subsection{On the limits of the 
computing method}

This text is meant as a ground 
for further research, that would aim at extending 
the computation method that we exposed 
to a broader set 
of multidimensional SFT, including 
for instance Kari-Culik 
tilings~\cite{KC96}, the monomer-dimer model [see for instance \cite{FriedlandPeled}], subshifts of square ice~\cite{GS17}, 
the hard square shift~\cite{P12}, the eight-vertex model~\cite{baxter} or 
a three-dimensional 
version of the six vertex model.
Adaptations for these models 
may be possible, but would 
not be immediate at all. We explain here 
at which points the method has limitations, 
each of them coinciding with a specific property 
of square ice. 

\subsubsection{Symmetry and irreducibility}

Let us recall that we called Lieb path 
an analytic function 
of transfer matrices $t \mapsto V_N (t)$ 
such that for all $t$, $V_N (t)$ is 
an irreducible non-negative and symmetric matrix
on $\Omega_N$.

\paragraph{Implications of the symmetry 
and mixing properties of square ice.}
Although the definition of transfer matrices 
admits straightforward generalization to multidimensional SFT and their non-negativity does not seem difficult to achieve, the 
property of \textit{symmetry} of the matrices $V_N (t)$ relies on symmetries 
of the alphabet and local rules of the SFT. 
S.Friedland~\cite{F97} proved 
that under these symmetry constraints (which 
are verified for instance by the monomer-dimer 
and hard square models, but a priori 
not by Kari-Culik tilings), entropy 
is algorithmically computable, through
a generalisation of the gluing 
argument exposed in Lemma~\ref{lemma.toroidal.ice}. 
Outside of the class of SFT defined 
by these symmetry restrictions, as 
far as we know, only strong mixing or 
measure theoretic conditions 
ensure algorithmic
computability of entropy, leading 
for instance to relatively efficient 
algorithms approximating the hard square 
shift entropy~\cite{P12}.
On the other hand, the \textit{irreducibility} 
of the matrices $V_N (t)$  
derives from the irreducibility property of 
the stripes subshifts $X^s_N$ [Definition~\ref{definition.stripes.subshift}], 
that can be derived from the \textit{linear block 
gluing property} of $X^s$~\cite{GS17}. 
This property consists in the possibility 
for any pair of pattern on $\mathbb{U}^{(2)}_N$ 
to be glued in any relative positions, 
provided that the distance between the 
two patterns is greater than a minimal 
distance, which is $O(N)$. 

\paragraph{Possible relaxations of 
some arguments.}

Lemma~\ref{lemma.toroidal.ice}, 
which relies on a horizontal 
symmetry of the model,
is a simplification in the 
proof of Theorem~\ref{theorem.main}, whose implication is that entropy 
of $X^s$ can be computed through 
entropies of subshifts $\overline{X}^s_N$, 
and thus simplifies Algebraic Bethe ansatz, 
that we will expose in another text.
One can see in~\cite{VieiraLimaSantos} that 
it is possible to use the ansatz without 
Lemma~\ref{lemma.toroidal.ice}.
However, this application of 
the ansatz would lead to different Bethe equations, and it is not clear if these
equations admits solutions, and if we can 
evaluate their asymptotic behavior.
The symmetry is also involved in the 
equality of entropy of $\overline{X}^s_{n,N}$ 
and entropy of $\overline{X}^s_{N-n,N}$. 
Without this equality, we don't know how 
to identify the greatest eigenvalue of $V_N (t)$ 
with the candidate eigenvalue obtained via 
the ansatz.

\subsection{On the gap between 
mathematics and mathematical physics}

The difficulties that were encountered in 
proving Theorem~\ref{theorem.main}, besides 
partial arguments, were related primarily to the form of the literature on the subject, as 
a field of research in mathematical physics,
and the gap that there exists with mathematical 
literature. In Section~\ref{section.gap.mathematics.physics}, we provide 
a short analysis of this gap, that relies 
on the concept of \textit{discursive formation} developped by M. Foucault 
in the \textit{Archaeology of knowledge}, 
which generalises the 
particularly organised and socialy 
structured forms of 
discourse that are sciences, or philosophy, 
and important aspects of discursive formations 
that are the mode of creation, existence and coexistence of concepts within it and the 
conception of units of meaning. We describe 
there mathematics and mathematical physics 
as distinct discursive formation: we analyse, 
from our point of view, their contemporary 
form and the consequence it has on how 
the units of meaning and objects 
of discourse (theorems, proofs) are conceived.
A significant part of 
the work done in the core of the 
present text is an analog of translation, from 
discursive formation of to another, 
and we provide some examples, 
in Section~\ref{section.translation.difficulties}, of the difficulties implied in the translation 
process by the distinction of mathematics 
and mathematical physics as discursive formations.

\subsubsection{\label{section.gap.mathematics.physics} Distinct discursive formations}

From our point of view, 
the most saillant difference 
that separate mathematics from mathematical 
physics - which use the same elements of meanings, 
the same formalism - lies on 
the use and the structuration of language. 

\paragraph{On mathematics:}

As far as we understand it, 
contemporary 
mathematics could be conceived as a space 
of constant exchange, accumulation and communication of proof techniques (communicated 
as proofs of theorems) that derive from the 
impossibility to foresee the use that 
tools which emerged in another area
can have to solve a problem. This 
implies a logic of accumulation of the texts, 
and their (cognitive) content. 
This constant exchange and the inflation of 
accessible information that follows seems to imply 
a shift in the function of the mathematician, 
who has to understand primarily which technique 
is suitable to which problem. The 
universality of the language of mathematics 
is fundamental for these exchanges, 
but there is also a necessity for 
the mathematical text to match the 
functions of memory, in particular 
the optimisation of the 
time and attention allocated to 
reading, and thus the accessibility 
of the (cognitive) content of the text.
The text thus is assumed to allow the extraction 
of pertinent information, from the 
point of view of the reader, whose 
background is a priori unknown. The text 
relates how the various techniques 
involved are articulated and the context 
in which they are applicable, etc.
It exhibits a hierarchisation of its 
content, from an overall point of view 
that includes motivations for the reader 
to get into the text, to the many details, 
and including markers of the function of 
each articulation in this hierarchy (including 
the functions of lemma, definition, comments, 
section abstract, etc), but also markers 
for the possibility of further development.
The necessity for a motivation for the reader 
implies in particular that the ability 
of a technique or a set of techniques to 
actually prove a theorem defines the unity 
of discourse (the article). The dynamics 
of contemporary mathematics seems 
also to have an impact on the way concepts 
are formed (or equivalently named), since 
the name is a marker (that helps 
for bibliographical orientation) 
of deepness and level 
of connectivity to other notions, and 
on the inclusion in the pair theorem-proof 
of the a priori implementability in 
another context of the articulation 
of techniques involved in the proof, 
given as prerequisite only the understanding 
of elementary mathematical objects involved. 

\paragraph{On mathematical physics:}

On the other hand, as far as we understand it, physics are 
structured around the contradiction of 
general theories explaining domains of phenomena 
appearing in the world, and the 
value of a theory is subjected 
to experimentation, which selects (as in 
natural selection) the technical tools 
that are adapted to explanation. 
Any theory is thus temporary, as well 
as any technical development within it, 
which is supposed to provide 
technical tools to compute some 
caracteristics of the physical 
system studied. These technical 
developments, when appearing in 
a recurrent way, are then turned into 
an intuition on the behavior of 
these systems, by essentialisation.
The same principle of optimisation of 
information treatment that leads 
in mathematics to memory-structured texts 
seems to lead in physics to meaning units (texts) 
that are centered on isolated (computing) non-rigorous techniques that presuppose a knowledge of 
their context. These techniques are selected 
by their reccurence in various directions of 
research within this context, 
before an attempt of a rigorous version, 
the expectation of which relies on 
the close relation to objects that have 
a meaning in the reality. 

\subsubsection{\label{section.translation.difficulties}Translation difficulties} 

From the point of view of information 
treatment, an aspect of the relation between 
mathematics and mathematical physics 
is analogical to the relation in 
the brain between myelin and neurons.
While neurons seem to be formed by an a priori production followed 
by selection, the myelin is 
constructed selectively around some 
neurons in order to accelerate 
their information processing. 
Following this analogy, the 
general process of \textit{myelinisation}, 
of which the present text is an element, 
exhibits a lot of difficulties, 
that are not related to grammar of 
the respective languages of mathematics 
and mathematical physics (since they are 
the same) but to their structures as 
\textit{discursive formations}.

Indeed, a primary effect of the conception of 
meaning units as defined by
isolated techniques is the 
non-neglectable distance that there exists 
between an existing group of techniques 
and an actual proof of a theorem using 
these techniques, for a subject within 
the space of mathematics. This effect 
is due in part to the fluctuation of 
notations (there is an interesting analogy 
between the use of common notations 
and prototyping in programming) and terms used 
from a text to another to designate the 
same objects that come 
along with non-explicit one-to-one 
transformations of the objects considered 
(as it is for instance from~\cite{YY66}
to \cite{kozlowski} for the 
definition of the function 
$\theta_t$ [Section~\ref{section.notations}]), in part also to non-explicit 
reference to definitions or 
other techniques, the knowledge 
of which is presupposed in the text. 
The difficulties that come from the fluctuation 
of notations have a particular effect on 
cases distinctions, since we tend, in order 
to accelerate reading, 
to identify cases (for an example, 
the phases of the six-vertex model, 
that corresponds to domains for the parameter $t$ in the present text) to properties of 
the objects considered in these cases: 
this acceleration becomes an obstacle 
in the presence of a change of notations. 

Some other difficulties come from the 
multiplicity of methods whose mode of 
coexistence is not explicit in 
the literature (for instance the coordinate 
Bethe ansatz, exposed for instance in~\cite{Duminil-Copin.ansatz}, 
and the algebraic Bethe ansatz (does 
it worth to invest time in understanding the 
other technique when one is more directly 
accessible?). Also, the absence of markers 
for the generality of the techniques used 
can be misleading, as well as 
a formulation of the generality of the method, 
where the degree of generality is ambiguous, 
since it is dependant upon implicit 
prerequisite of the knowledge of the field 
(for the algebraic Bethe ansatz in the 
literature for instance). 
Some particular difficulties come more 
directly from the absence of separation 
of statements having different functions 
(definition or lemma), underlocalised autoreference to parts of the text, in 
particular in the case of multiple references, 
or the absence of parastructural comments or 
object typing for the various mathematical 
objects considered, which swipe off the 
possible ambiguities in writing. 
 
At a higher level, the ambiguity of the 
distinction between mathematics and 
mathematical physics is itself a difficulty, 
that demand specific tools, such as 
the ones developed by M. Foucault in 
the \textit{Archaeology of knowledge}, 
to make visible the border between the 
discursive formations and explain it, 
in order to develop general 
translation tools for this border, 
from discursive formation to discursive 
formation, such as, simply, a change of 
perception of the unity of the text 
or the multiplication of the sources 
in order to understand the concepts involved.

\end{document}